\documentclass[11pt,final]{article}

\usepackage[margin=1.2in]{geometry}

\usepackage[dvipsnames, svgnames]{xcolor}
\definecolor{vert}{RGB}{15,120,5}
\definecolor{gris}{RGB}{128,128,128}
\definecolor{bleu}{RGB}{0,50,150}
\definecolor{rouge}{RGB}{149,24,24}
\usepackage[colorlinks = true,linkcolor=bleu,citecolor=vert]{hyperref}

\usepackage{lmodern}
\usepackage[T1]{fontenc}
\usepackage{thmtools}

\usepackage{comment}

\usepackage[colorinlistoftodos,obeyFinal]{todonotes}
\usepackage{xpatch}
\usepackage[explicit]{titlesec}

\def\thissectiontitle{}
\def\thissectionnumber{}
\def\thissubsectiontitle{}
\def\thissubsectionnumber{}

\newtoggle{todoSection}
\newtoggle{todoSubsection}

\titleformat{\section}
  {\normalfont\Large\bfseries\sffamily}
  {\thesection}
  {0.5em}
  {\gdef\thissectiontitle{#1}\gdef\thissectionnumber{\thesection}#1}

\titleformat{\subsection}
  {\normalfont\large\bfseries\sffamily}
  {\thesubsection}
  {0.5em}
  {\gdef\thissubsectiontitle{#1}\gdef\thissubsectionnumber{\thesubsection}#1}

  \pretocmd{\section}{\global\toggletrue{todoSection}}{}{}
  \pretocmd{\subsection}{\global\toggletrue{todoSubsection}}{}{}

\AtBeginDocument{%
  \xpretocmd{\todo}{%
    \iftoggle{todoSubsection}{
     \addtocontents{tdo}{\protect\contentsline{subsection}%
        {\protect\numberline{\thissubsectionnumber}{\thissubsectiontitle}}{}{} }
      \global\togglefalse{todoSubsection}
        }{}
    }{}{}%
  }
\AtBeginDocument{%
  \xpretocmd{\todo}{%
    \iftoggle{todoSection}{
     \addtocontents{tdo}{\protect\contentsline{section}%
        {\protect\numberline{\thissectionnumber}{\thissectiontitle}}{}{} }
      \global\togglefalse{todoSection}
        }{}
    }{}{}%
  }

\usepackage[english]{babel} 
\usepackage[utf8]{inputenc}
\usepackage{mathtools}
\usepackage{pict2e}

\usepackage{mathrsfs} 
\usepackage{amsfonts}
\usepackage{marvosym} 
\usepackage{amsmath,amssymb,amsthm}

\usepackage{bm}

\usepackage[nameinlink]{cleveref}
\crefname{equation}{}{}
\usepackage{bbm}

\usepackage{epsfig}

\usepackage{tikz-cd} 

\usepackage{stmaryrd}
\usepackage{manfnt}
\usepackage{multicol}
\usepackage{array}
\usepackage{multirow}

\usepackage{fancyhdr} 
\usepackage{fancybox} 

\usepackage{calligra}
\title{Pro-étale motives and solid rigidity}
\author{Raphaël Ruimy, Swann Tubach and Sebastian Wolf}
\date{}

\usepackage{tikz}
\usetikzlibrary{calc, intersections}

\usepackage{adjustbox}

\theoremstyle{plain}
\newtheorem{thm}{Theorem}[section]
\newtheorem{prop}[thm]{Proposition}

\newtheorem{introthm}{Theorem}

\newtheorem{introprop}[introthm]{Proposition}

\newtheorem{lem}[thm]{Lemma}
\newtheorem{cor}[thm]{Corollary}

\theoremstyle{definition}
\newtheorem{defi}[thm]{Definition}
\newtheorem{emptypar}[thm]{}
\newtheorem{constr}[thm]{Construction}
\newtheorem{rec}[thm]{Recollection}

\theoremstyle{remark}

\newtheorem{rem}[thm]{Remark}

\newtheorem{ex}[thm]{Example}

\numberwithin{thm}{subsection}
\newcommand{\Sr}{\mathrm{S}}

\newcommand{\Q}{\mathbb{Q}}
\newcommand{\Z}{\mathbb{Z}}
\newcommand{\N}{\mathbb{N}}

\newcommand{\Spec}{\operatorname{Spec}}
\newcommand{\A}{\mathbb{A}}

\newcommand{\HH}{\mathrm{H}}
\newcommand{\D}{\mathrm{D}}

\newcommand{\ocal}{\mathcal{O}}
\newcommand{\oscr}{\mathscr{O}}

\newcommand{\acal}{\mathcal{A}}
\newcommand{\ucal}{\mathcal{U}}
\newcommand{\bcal}{\mathcal{B}}

\newcommand{\fcal}{\mathcal{F}}

\newcommand{\tcal}{\mathcal{T}}
\newcommand{\ccal}{\mathcal{C}}
\newcommand{\gcal}{\mathcal{G}}
\newcommand{\dcal}{\mathcal{D}}
\newcommand{\mcal}{\mathcal{M}}
\newcommand{\ncal}{\mathcal{N}}

\newcommand{\mc}{\mathcal}

\newcommand{\Ani}{\mathrm{Ani}}
\newcommand{\mb}{\mathbb}
\newcommand{\proet}{\mathrm{pro\acute{e}t}}

\newcommand{\restr}[2]{\ensuremath{\left.{#1}\right|_{#2}}}

\newcommand{\Hom}{\mathrm{Hom}}

\newcommand{\Gm}{{\mathbb{G}_m}}

\DeclareMathOperator{\sHom}{\mathscr{H}\text{\kern -3pt {\calligra\large om}}\,}
\DeclareMathOperator{\ProFin}{ProFin}
\DeclareMathOperator{\cond}{cond}
\DeclareMathOperator{\Cond}{Cond}
\DeclareMathOperator{\Sp}{Sp}
\DeclareMathOperator{\Sing}{Sing}
\DeclareMathOperator{\pr}{pr}
\DeclareMathOperator{\mot}{mot}
\newcommand{\Ext}{\mathrm{Ext}}

\newcommand{\Fun}{\mathrm{Fun}}
\newcommand{\CAlg}{\mathrm{CAlg}}

\newcommand{\DM}{\mathrm{DM}}

\newcommand{\et}{\mathrm{\acute{e}t}}
\newcommand{\Mod}{\mathrm{Mod}}
\newcommand{\un}{\mathbbm{1}}
\newcommand{\Sh}{\mathrm{Sh}}

\newcommand{\WSm}{\mathrm{WSm}}

\newcommand{\Sch}{\mathrm{Sch}}
\newcommand{\op}{\mathrm{op}}

\newcommand{\catinfty}{\mathrm{Cat}_\infty}

\newcommand{\colim}{\operatorname{\mathrm{colim}}}
\newcommand{\aff}{\operatorname{\mathrm{aff}}}

\newcommand{\In}{\operatorname{\mathrm{Ind}}}

\newcommand{\PSh}{\operatorname{\mathrm{PSh}}}

\newcommand{\heart}{\heartsuit}

\newcommand{\rar}{\to}
\newcommand{\PrL}{\mathrm{Pr}^\mathrm{L}}

\newcommand{\Sm}{\mathrm{Sm}}

\newcommand{\id}{\mathrm{Id}}
\DeclareMathOperator{\fp}{fp}

\DeclareMathOperator{\SH}{SH}

\newcommand{\ZhatP}{\widehat{\Z}_\mc{P}}

\usepackage{appendix}
\makeatletter
  \newcommand{\adjunction}{\@ifstar\named@adjunction\normal@adjunction}
  \newcommand{\normal@adjunction}[4]{%
    #1\colon #2%
    \mathrel{\vcenter{%
      \offinterlineskip\m@th
      \ialign{%
        \hfil$##$\hfil\cr
        \longrightharpoonup\cr
        \noalign{\kern-.3ex}
        \smallbot\cr
        \longleftharpoondown\cr
      }%
    }}%
    #3 \noloc #4%
  }
  \newcommand{\named@adjunction}[4]{%
    #2%
    \mathrel{\vcenter{%
      \offinterlineskip\m@th
      \ialign{%
        \hfil$##$\hfil\cr
        \scriptstyle#1\cr
        \noalign{\kern.1ex}
        \longrightharpoonup\cr
        \noalign{\kern-.3ex}
        \smallbot\cr
        \longleftharpoondown\cr
        \scriptstyle#4\cr
      }%
    }}%
    #3%
  }
  \newcommand{\longrightharpoonup}{\relbar\joinrel\rightharpoonup}
  \newcommand{\longleftharpoondown}{\leftharpoondown\joinrel\relbar}
  \newcommand\noloc{%
    \nobreak
    \mspace{6mu plus 1mu}
    {:}
    \nonscript\mkern-\thinmuskip
    \mathpunct{}
    \mspace{2mu}
  }
  \newcommand{\smallbot}{%
    \begingroup\setlength\unitlength{.15em}%
    \begin{picture}(1,1)
    \roundcap
    \polyline(0,0)(1,0)
    \polyline(0.5,0)(0.5,1)
    \end{picture}%
    \endgroup
  }
\makeatother

\begin{document}

\maketitle
\begin{abstract}
	We introduce coefficient systems of pro-étale motives and pro-étale motivic spectra with coefficients in any condensed ring spectrum and show that they afford the six operations. 
	Over locally étale bounded schemes, étale motivic spectra embed into pro-étale motivic spectra. 
	We then use the framework of condensed category theory to define a solidification process for any $\widehat{\Z}$-linear condensed category.
	Pro-étale motives naturally enhance to a condensed category and we show that their solidification is very close to the category of solid sheaves defined by Fargues-Scholze, suitably modified to work on schemes: this is a rigidity result. As a consequence, we obtain that in contrast with the rigid-analytic setting, solid sheaves on schemes afford the six operations, and we obtain a solid realization functor of motives, extending the $\ell$-adic realization functor. The solid realization functor is compatible with change of coefficients, which allows one to recover the $\Q_\ell$-adic realization functor while remaining in a setting of presentable categories.
\end{abstract}
\tableofcontents



\section*{Introduction}
\subsection*{Coefficient systems for \texorpdfstring{$\ell$}{l}-adic cohomology}
Betti cohomology can be expressed in terms of sheaves on the topological space of complex points; likewise de Rham cohomology is best understood through the theory of $\mathscr{D}$-modules; finally torsion étale cohomology is defined through torsion sheaves on the small étale site. In all of the above examples, cohomology is expressed by means of a \emph{coefficient system}: in Grothendieck’s vision, this is a category fibered over the category of schemes, equipped with additional functoriality (such as the six operations) allowing for the definition of cohomology with coefficients in such systems.

Finding well suited coefficient systems for $\ell$-adic cohomology is more challenging. Let us first remark that there should actually be two such coefficient systems: the \emph{big} (or more accurately \emph{presentable}) one and the \emph{small} one that should lie inside of it as the subcategory of constructible objects. The theory of constructible $\ell$-adic sheaves is well understood. Its initial construction was done for finite-type schemes over a field in \cite{MR0751966} by taking a (homotopy) limit:
\[\D^b_c(-,\Z_\ell)=\lim_n \D^b_c(-,\Z/\ell^n\Z)\]
with the right-hand side given by constructible sheaves on the small étale site. Ekedahl then constructed a presentable coefficient system over noetherian schemes in \cite{zbMATH00016091}. The latter can actually be recovered by taking the $\ell$-completion $\D((-)_{\et},\Z)^\wedge_\ell$ of the derived category of the abelian category of étale sheaves (this is the approach of \cite{MR3477640}).
One downside of Ekedahl's coefficient system is that it does not interact well with rationalization, in fact
\[
	\D((-)_{\et},\Z)^\wedge_\ell \otimes \D(\Q) \simeq 0.
\]
In order to obtain a well behaved theory one needs to keep track of the natural topologies on the rings $\Z_\ell$ and $\Q_\ell$.
The now classical idea from \cite{zbMATH06479630} is that instead of $\ell$-completing in the categorical sense as in Ekedahl's theory, one can keep track of these topologies by using the \emph{pro-étale site} of a scheme.
Indeed, one can see both $\Z_\ell$ and $\Q_\ell$ as sheaves on the pro-étale site in a manner that takes into account their topology.
A natural candidate for a coefficient system is now the derived category $\D((-)_{\proet},\Z_\ell)$ of the abelian category of $\Z_\ell$-modules in the pro-étale site.
Indeed, in this case one has
\[
	\D((-)_{\proet},\Z_\ell) \otimes \D(\Q) \simeq \D((-)_{\proet},\Q_\ell),
\]
as expected.
This system is however far from having all the desired functoriality properties.
For instance, the proper base change theorem already fails for this coefficient system \cite[Remark 7.15]{bastietal}.
One of the goals of this work is to provide a theory that has all the expected functoriality properties and is still compatible with rationalization in the above sense.

\subsection*{Voevodsky motives and their \texorpdfstring{$\ell$}{l}-adic realization}
Another motivation comes from the theory of Voevodsky motives and especially $\ell$-adic realization functors.
Voevodsky étale motives form the universal étale-local coefficient system for schemes $\DM_{\et}$; the latter (in particular all the key functorialities) was constructed in \cite{MR2687724,MR3205601,MR3477640}. The theory of $\ell$-adic cohomology is connected to the theory of étale motives via $\ell$-adic realization functors.
A powerful tool in understanding these realization functors is the \emph{rigidity theorem} (see \textit{e.g.} \cite{bachmannrigidity} for the most general such theorems): left Kan-extension from \'etale to smooth $X$-schemes gives a natural map 
\[\D((-)_{\et},\Z)\to \DM_{\et}\]
which becomes an equivalence after $\ell$-completion.
Thus $\ell$-completion provides a natural map to Ekedahl's coefficient system
\[\DM_{\et}\to (\DM_{\et})^\wedge_\ell \simeq \D((-)_{\et},\Z)^\wedge_\ell,\]
which we call the \emph{$\ell$-adic realization}.
Another way of thinking about the incompatibility of $\D((-)_{\et},\Z)^\wedge_\ell$ with rationalization mentioned above, is the fact $\Q$ is sent to $0$ by this functor. 
Furthermore, one can deduce from the above rigidity theorem that $\ell$-adic cohomology is representable in $\DM_{\et}$ by the $\ell$-completion $\Z^\wedge_\ell$ of the unit. This motive has the unpleasant feature that it is not stable under base change in mixed characteristic (see \cite{arXiv:2503.24033}): it is shown in loc. cit. that such stability would imply
independence-of-$\ell$ results, which are actually too strong to be true.

The other main objective of this paper is to build a new framework in which $\ell$-adic realization functors are exceptionally well-behaved: the idea is to construct a coefficient system $\DM_{\proet}$ of \emph{pro-étale motives}, which is an enlargement of $\DM_{\et}$ and allows one to consider motives with coefficients in topological rings.
One then again has a canonical map 
\[\D((-)_\proet,\Z_\ell)\to \DM_\proet(-,\Z_\ell)\]
which we would very much like to be an equivalence: $\ell$-adic realization functors would then simply be given by extending scalars to the topological ring $\Z_\ell$ which would thus represent $\ell$-adic cohomology and be stable under base change.
This is however far too optimistic as the right-hand side is a coefficient system while the left-hand side is not as it does not have proper base change.
To salvage this situation, we introduce a way of \emph{solidifying} the above categories, building on ideas of \cite{condensedpdf} and \cite{Fargues-Scholze}.

\subsection*{Solid sheaves on schemes}
Condensed abelian groups (sheaves of abelian groups on the pro-étale site of a point) can be thought of as a model for topological abelian groups.
Solidity then aims to describe a version of non-archimedean completeness: solid abelian groups form a full subcategory of condensed abelian groups and for instance contain $p$-adic Banach spaces as a full subcategory. In the rigid-analytic world, Fargues and Scholze developed the theory of solid sheaves with $\widehat{\Z}$-coefficients in \cite[Chapter~VII]{Fargues-Scholze}. They are subcategories of pro-étale sheaves of $\widehat{\Z}$-modules that are small enough to afford some nice functorialities but large enough to be closed under all limits and colimits. Indeed, in the analytic setting, solid sheaves can be endowed with a so-called 5 functors formalism: the main feature is that any pullback functor admits a left adjoint and that arbitrary base change is true. It is however not possible to endow solid sheaves with a six functors formalism  in this setting as the proper projection formula is not true in general (\cite[Warning~VII.2.5]{Fargues-Scholze}).

The first step of this paper is to adapt the framework of solid sheaves to schemes. Let $X$ be a quasi-compact and quasi-separated (qcqs) scheme, let $\mathcal{P}$ be a subset of the set of prime numbers and let $\widehat{\Z}_\mathcal{P}=\prod_{p\in \mathcal{P}}\Z_p$.
If $j\colon U\to X$ is a pro-étale map that can be written as a cofiltered limit of étale maps $j_\alpha \colon U_\alpha \to X$ with each $U_\alpha$ affine, we set \[j_\sharp^\blacksquare\ZhatP\coloneqq \lim_\alpha (j_\alpha)_\sharp \ZhatP.\]  The abelian category of \emph{solid sheaves} on $X$ is then defined as the localization $\Sh(X,\ZhatP)^{\blacksquare}$ of $\Sh(X_{\proet},\ZhatP)$ with respect to the natural maps $j_\sharp \ZhatP\to j_\sharp^\blacksquare \ZhatP$. We then say that a complex of pro-étale sheaves of $\ZhatP$-modules is solid if its cohomology sheaves are solid. We denote by $\D(X,\ZhatP)^{\blacksquare}$ the full subcategory of $\D(X_\proet,\ZhatP)$ spanned by solid objects. Finally, if $\Lambda$ is a solid $\widehat{\Z}_\mathcal{P}$-algebra, we set \[\D(X,\Lambda)^{\blacksquare}\coloneqq \Mod_\Lambda(\D(X,\ZhatP)^{\blacksquare}).\]

We prove that these categories have all the basic properties that one expects from a good category of solid sheaves.
\begin{introthm}(\Cref{thm:solid}, \Cref{prop:derivedSolid_Lambda}, \Cref{prop:cons_are_solid} and \Cref{prop:cons_are_solid_rational}) Let $X$ be a qcqs scheme.

\begin{enumerate}
\item The subcategory $\Sh(X,\ZhatP)^{\blacksquare}$ of $\Sh(X_\proet,\ZhatP)$ is abelian and closed under extensions, small limits, colimits and base change. It is compactly generated by the pro-category of constructible étale sheaves of $\mc{P}$-torsion. 
\item Constructible complexes in the sense of \cite{MR4609461} are solid, \textit{i.e.} we have an inclusion $\D^b_c(X,\ZhatP)\subseteq \D(X,\ZhatP)^{\blacksquare}$. This is also true with $\Q_\ell$-coefficients. 
\item Let $\Lambda$ be a solid $\widehat{\Z}_\mathcal{P}$-algebra. The inclusion
\[\D(X,\Lambda)^{\blacksquare}\subseteq \D(X_\proet,\Lambda)\]
admits a left adjoint $(-)^\blacksquare$ and is compatible with colimits and base change. There is a unique symmetric monoidal structure $\otimes^\blacksquare$ on $\D(X,\Lambda)^{\blacksquare}$ making the functor $(-)^\blacksquare$ symmetric monoidal. 
\end{enumerate}
\end{introthm}

One of the main results of this paper is that, unlike in the analytic setting, the categories $\D(-,\Lambda)^{\blacksquare}$ support a full six functor formalism over schemes on which all primes in $\mc{P}$ are invertible (see \Cref{Main_Thm}).

\subsection*{Pro-étale motives}
We will construct the six-operations on solid sheaves by comparing them to a suitable category of pro-\'etale motives.
Note that if we know that $\DM_\et$ has the six functors, the rigidity theorem implies that over schemes such that $n$ is invertible, the same holds for $\D((-)_{\et},\Z/n\Z)$. We want to adapt this approach for solid sheaves.
The first step is to define a pro-étale version of motives and motivic spectra. 

Let us now recall the (étale-local version of) Morel and Voevodsky's construction of $\A^1$-invariant motivic spectra over a scheme $X$: we start from the category $\Sm_X$ of smooth $X$-schemes endowed with the étale topology and build the category $\Sh_\et(\Sm_X,\mathrm{Sp})$ of hypersheaves of spectra over $\Sm_X$; the category $\SH_{\et}(X)$ of étale motivic spectra is then the $\mb{P}^1$-stabilization of the subcategory of $\A^1$-local sheaves. For our purposes, the category of $\Sm_X$ is too small: we want to replace the étale topology with the pro-étale topology. 
To that end, we define the category $\WSm_X$ of \emph{weakly smooth} $X$-schemes as the category of $X$-schemes whose structure morphism is a finite composition of smooth morphisms and weakly étale morphisms. As for smooth maps, the local structure of weakly smooth maps is simple: Zariski-locally on the source they can be written as a composition 
\[Y\xrightarrow{f} \A^n_X \to X\] 
with $f$ weakly étale (this is \Cref{lem:locWL}). 
The category $\SH_\proet(X)$ is then defined as the $\mb{P}^1$-stabilization of the subcategory $\SH_\proet^{\mathrm{S}^1}(X)$ of $\A^1$-local objects in the category $\Sh_\proet(\WSm_X,\mathrm{Sp})$ of pro-étale hypersheaves of spectra on $\WSm_X$. 
The main result is that pro-étale motivic spectra afford the six functors in a slightly weaker form than the six functors on usual motivic spectra (we need finite presentation assumptions to get exceptional functors and not just the usual assumption that they are of finite type); furthermore étale motivic spectra embed into pro-étale motivic spectra under some finiteness assumptions:

\begin{introthm}(\Cref{cor:6FF_proet_mot} and \Cref{thm:embedding_etale_motives})
     We have a functor \[\mathrm{SH}_\proet\colon \Sch^\op\to \CAlg(\PrL)\] that is part of a six functors formalism. More precisely, 
     \begin{enumerate}
    \item We have pairs of adjoint functors given by 
        \begin{itemize}
            \item For any morphism of schemes $f$ we have an adjunction $(f^*,f_*)$ with $f^*$ the functor $\SH_\proet(f)$.
            \item For a finitely presented $f$ between qcqs schemes, we have an adjunction $(f_!,f^!)$. There is a natural transformation $f_!\to f_*$ which is an equivalence if $f$ is furthermore proper.
            \item Over any $X$, we have a closed monoidal structure $(\otimes_X,\underline{\Hom}_X)$ on $\mathrm{SH}_\proet(X)$.
        \end{itemize} 
    
    \item Let $f\colon Y\to X$ be a finitely presented morphism between qcqs schemes. Then the formation of $f_!$ is compatible with pullbacks.
    \item For any finitely presented morphism $f\colon X\rar Y$ of qcqs schemes, we have a natural isomorphism $$(f_!M)\otimes N \rar f_!(M\otimes f^*N).$$
    \item In fact one can do better and obtain a six functors formalism as a lax symmetric monoidal functor on a category of correspondences that encode most of the wanted coherence of 1., 2. and 3. (see \Cref{thm:motivic_6FF}).
    \item For any $f\colon X\to S$ finitely presented smooth morphism of qcqs schemes, then with the notations of \Cref{def:Sigma^M}, there is a natural equivalence $ f^! \simeq \Sigma^{\Omega_f} \circ f^*$.
    \item For any closed immersion $i\colon F\rar X$ with qcqs open complement $j\colon X\setminus F\rar X$, we have exact triangles of natural transformations $$j_!j^!\rar \id \rar i_*i^*$$
    $$i_!i^!\rar \id \rar j_* j^*$$ given by the unit and counit of the appropriate adjunctions.
\end{enumerate}
Finally, if $X$ is a locally étale bounded (\Cref{def:etale_bounded}) qcqs scheme, then the canonical functor
	\[
		\SH_\et(S)\to \SH_\proet(S)
	\]
	is fully faithful.
\end{introthm}

Note in the above theorem we do in almost all statements need to assume that the morphisms in question are of finite presentation.
The reason that this is necessary lies within the sixth point, which is the crucial ingredient to prove all the other properties above.
For a non-finitely presented closed immersion of qcqs schemes, the open complement might not be quasi-compact and there are no localization sequences, see \Cref{rem:counterex_to_loc}.
This also explains the different behavior in the rigid analytic setting of \cite[\S VII]{Fargues-Scholze}.
In analytic geometry it is rather rare for the open complement of a closed immersion to be quasi-compact and thus the arguments in this paper do not apply in that setting.

\subsection*{Condensed category theory}
Let $\Lambda$ be a condensed $\ZhatP$-algebra. The coefficient system $\DM_\proet(-,\Lambda)$ of pro-étale motives with $\Lambda$-coefficients is then defined as the category of $H\Lambda$-modules in $\SH_\proet$. By abstract nonsense, it also affords a six functor formalism.
Left Kan extension defines a functor
\[\D((-)_\proet,\Lambda)\to \DM_\proet(-,\Lambda).\]
As we have seen above this map can not be an equivalence since the domain does not have the six operations while the codomain does.
The idea is that in oder to fix this, we need to solidify the above map.
We have seen what solidifying the left-hand side means and we need to provide an analog of this procedure for the right-hand side. To that end, we use internal category theory as developed in \cite{MYoneda,MWColimits,MWPresentable}.
The basic idea is that a \emph{condensed $\infty$-category} is a sheaf of $\infty$-categories on the site of profinite sets and that a condensed functor is simply a map of sheaves. A condensed $\infty$-category is \emph{presentable} if it is pointwise presentable and some squares are adjointable (see \Cref{def:cond_presentable}, this in fact amounts to having colimits in a condensed sense). A condensed functor has a left or right adjoint whenever it has an adjoint pointwise and some squares are adjointable (see \Cref{prop:cond_left_adj}, again this amounts to compatibility with condensed colimits). The category $\PrL_\mathrm{cond}$ is then defined as the category of presentable condensed categories with condensed left adjoint functors. As for the usual category $\PrL$, this category was endowed with a symmetric monoidal structure $-\otimes^\mathrm{cond}-$ in \cite{MWPresentable}. The latter is a condensed version of the Lurie tensor product. Given a condensed category $\underline{\ccal}$, one can ``forget the topology'' and obtain the underlying category $\underline{\ccal}(*)$ by taking global sections.

The basic examples of condensed $\infty$-categories that are used throughout the paper are 
\[\underline{\D}(X_\proet,{\Lambda})\colon K \mapsto \D((X\times K)_\proet, \Lambda)\]
and \[\underline{\mathrm{Solid}}_{\Lambda}\colon K\mapsto \D(K,{\Lambda})^{\blacksquare},\] where $K$ ranges through the category of profinite sets. We also set 
$\underline{\Mod}_{\Lambda}\coloneqq \underline{\D}(*_\proet,{\Lambda})$. The \emph{solidification} of a $\underline{\Mod}_{\Lambda}$-linear condensed presentable $\infty$-category $\underline{\ccal}$ is then defined as $\underline{\ccal}^\blacksquare\coloneqq \underline{\ccal} \otimes^{\mathrm{cond}}_{\underline{\Mod}_{\Lambda}} \underline{\mathrm{Solid}}_{\Lambda}$. We can recover solid sheaves on schemes using these tools:
\begin{introprop}(\Cref{LikeFS})
    Let $X$ be a qcqs scheme. Then \[\D(X,\Lambda)^{\blacksquare}=\underline{\D}(X_\proet,\Lambda)^\blacksquare(*).\]
\end{introprop}
\subsection*{The rigidity theorem}
We are now armed to prove the rigidity theorem. For it to be true, we assume all primes in $\mathcal{P}$ to be invertible on our schemes, because otherwise étale cohomology may not be $\A^1$-invariant. The most natural idea is to look at the condensed presentable $\infty$-category defined by \[\underline{\DM}_\proet(X,\Lambda)\colon K\mapsto \DM_\proet(X\times K,\Lambda),\] to solidify it and take global sections. This yields a presentable $\infty$-category that we denote by $\D^{\A^1,\Gm}(\WSm_X,\Lambda)^\blacksquare$.\footnote{Note that since the abstract solidification we introduced above is only defined for coefficients in a $\ZhatP$-algebra, we do not yet have a good category of solid motives integrally. Our methods also don't immediately generalize to the setting of integral coefficients: The main obstacle is that \Cref{limIsRlim} does not hold integrally. We plan to further address this point in future work.}
We automatically get a map  
\[\D(-,\Lambda)^\blacksquare\to \D^{\A^1,\Gm}(\WSm_{(-)},\Lambda)^\blacksquare\]
that we expect to be an equivalence.
At the moment however we do not know how to prove this.
The main obstacle is that we do not know if the unit in the codomain is complete.
To circumvent this issue, we further localize this category. Let us first take a step back and start from the effective case. We let $\D^{\A^1}_\proet(\WSm_X,\Lambda)$ be the category of $\Lambda$-modules in $\mathrm{SH}^{S^1}_\proet(X)$; we naturally get a condensed presentably monoidal $\infty$-category
\[\underline{\D}^{\A^1}_\proet(\WSm_X,\Lambda)\colon K \mapsto \D^{\A^1}_\proet(\WSm_{X\times K},\Lambda)\]
The $\infty$-category $\D^{\A^1}(\WSm_X,\Lambda)^\blacksquare$ of \emph{solid $\A^1$-invariant sheaves} over $X$ is then defined as $\underline{\D}^{\A^1}_\proet(\WSm_X,\Lambda)^\blacksquare(*).$ The following result essentially amounts showing that solid sheaves are homotopy invariant which we bootstrap from the torsion étale case:
\begin{introthm}(\Cref{thm:embedding_rigidity})
	Assume that each prime number in $\mc{P}$ is invertible on $X$. The functor
	\[\rho_\sharp^\blacksquare\colon \D(X,\Lambda)^\blacksquare\to\D^{\A^1}(\WSm_X,\Lambda)^\blacksquare\] obtained by Kan extension along the inclusion $X_\proet\to \WSm_X$ followed by $\A^1$-localization is fully faithful, commutes with all small limits and colimits, all pullbacks and with $f_\sharp$ for $f$ weakly étale.
\end{introthm}
Observe now that if $\ZhatP(1)$ is the free sheaf over the pointed scheme $(\Gm_X,1)$, we have a natural map $\ZhatP(1)\to \lim_{n\in \Pi\mc{P}} \Z/n\Z(1)\simeq \mu_{\infty,\mathcal{P}}$ where $\Pi\mc{P}$ is the set of those $n$ which are only divisible by primes in $\mathcal{P}$ and with 
$\mu_{\infty,\mathcal{P}}\coloneqq\lim_{n\in \Pi\mc{P}}\mu_n$. The usual torsion rigidity theorem implies that torsion étale sheaves are local with respect to this map. Thus, so are pro-(torsion étale) sheaves. A bit more work shows that both the source and the target of this map are compact (see the proof of \Cref{thm:solid_rigidity_V1}) and thus we see that $\rho_\sharp^\blacksquare$ lands in the subcategory of local objects. This is a very good sign as $\mu_{\infty,\mathcal{P}}$ is $\otimes$-invertible, so that localizing with respect to this map also enforces Tate-stability. This motivates the following definitions: we let $\mathrm{DM}^\mathrm{eff}(X,\Lambda)^\blacksquare$ (resp. $\mathrm{DM}(X,\Lambda)^\blacksquare$) be the localization of $\D^{\A^1}(\WSm_X,\Lambda)^\blacksquare$ (resp. $\D^{\A^1,\Gm}(\WSm_X,\Lambda)^\blacksquare$) at the maps
\[M(1)\to M\otimes_{\widehat{\Z}_\mathcal{P}}\mu_{\infty,\mathcal{P}}\] for all $M$. We can now finally state our rigidity result:

\begin{introthm}\label{Main_Thm}(\Cref{thm:solid_rigidity_V1} and \Cref{prop:solid_rigidity_V2})
	Assume that each prime number in $\mathcal{P}$ is invertible on $X$.
	The functors \[\D(X,\Lambda)^\blacksquare \to \D^{\mathbb{A}^1}(\WSm_X,\Lambda)^\blacksquare\to \D^{\A^1,\Gm}(\WSm_X,\Lambda)^\blacksquare\] induce equivalences
	\[
		\D(X,\Lambda)^\blacksquare \to \mathrm{DM}^\mathrm{eff}(X,\Lambda)^\blacksquare\to \mathrm{DM}(X,\Lambda)^\blacksquare.
	\]
	In particular, the $\D(-,\Lambda)^\blacksquare$ are endowed with the six functors on schemes $X$ such that each prime number in $\mathcal{P}$ is invertible on $X$. If $R$ is a commutative (discrete) ring, this yields a \emph{solid $\mathcal{P}$-adic realization functor}:
\[
		\rho_\blacksquare\colon \DM_{\et}(X,R)  \to\mathrm{DM}_\proet(X,R\otimes_\Z \ZhatP) \to \mathrm{DM}(X,R\otimes_\Z \ZhatP)^\blacksquare\simeq\mathrm{D}(X,R\otimes_\Z \ZhatP)^\blacksquare
	\]
\end{introthm}

We can then compare our realization with the classical $\ell$-adic realization.

\begin{introprop}
	(\Cref{cor:MagicSquare})
	We have a commutative square 
	\[\begin{tikzcd}
	{\mathrm{DM}_\et(X,\Z)} & {\D(X,\Z_\ell)^\blacksquare} \\
	{\mathrm{DM}_\et(X,\Q)} & {\D(X,\Q_\ell)^\blacksquare}
	\arrow["{\rho_\blacksquare}", from=1-1, to=1-2]
	\arrow["{\otimes\Q}"', from=1-1, to=2-1]
	\arrow["{\otimes\Q}", from=1-2, to=2-2]
	\arrow["{\rho_\blacksquare}"', from=2-1, to=2-2]
\end{tikzcd}\] of symmetric monoidal left adjoints which restricts to the classical square 
\[\begin{tikzcd}
	{\mathrm{DM}_{\et,\mathrm{gm}}(X,\Z)} & {\D^b_c(X_\et,\Z_\ell)} \\
	{\mathrm{DM}_{\et,\mathrm{gm}}(X,\Q)} & {\D^b_c(X_\et,\Q_\ell)}
	\arrow["{\rho_\ell}", from=1-1, to=1-2]
	\arrow["{\otimes\Q}"', from=1-1, to=2-1]
	\arrow["{\otimes\Q}", from=1-2, to=2-2]
	\arrow["{\rho_{\Q_\ell}}"', from=2-1, to=2-2]
\end{tikzcd},\] where $\rho_\ell$ is the $\ell$-adic realization functor of \cite{MR3477640} and $\rho_{\Q_\ell}$ is the $\Q_\ell$-adic realization functor of \cite[2.1.2]{MR4061978}.
\end{introprop}
Over schemes of characteristic zero, the above result also holds for $\widehat{\Z}$ and the ring of finite adeles $\A_{\Q,f}$ instead of $\Z_\ell$ and $\Q_\ell$. This allows for another definition of the algebra $\mathscr{N}\in\mathrm{DM}^{\et}(\Spec(\Q),\Z)$ such that modules over it are Nori motives, as considered in \cite{integralNori}. In forthcoming work, we plan to study a spectral analogue of the constructions of this paper. This would lead to a definition of spectral Nori motives.

\subsection*{Acknowledgements}
RR thanks Frédéric Déglise, Fabrizio Andreatta and Jean Fasel for their support and their interest in this project. ST thanks Sophie Morel for her support.
SW would like to thank his advisor Denis-Charles Cisinski for suggesting that he should think about this topic and many helpful discussions.
SW would also like to think Remy van Dobben de Bruyn for helpful discussions related to the contents of this paper.

RR acknowledges support from the Prin 2022, \emph{The arithmetic of
motives and L-functions} and the ANR-21-CE40-0015 HQDIAG. ST acknowledges support by the European Research Council (ERC) under
the European Union’s Horizon 2020 research and innovation programme ``EMOTIVE'' (grant agreement
no. 101170066) and the ANR-21-CE40-0015 HQDIAG.
SW gratefully acknowledges support from the SFB 1085 Higher Invariants in Regensburg, funded
by the DFG and was also supported by the Danish National Research Foundation through the Copenhagen Centre for
Geometry and Topology (DNRF151).
\subsection*{Notations and Conventions}
We denote by $\mathrm{Ani}$ the category of anima\footnote{also known as spaces, $\infty$-groupoids, Kan complexes, sets...}.
Let $\ccal$ be a site.
We denote by $\PSh(\ccal)$ (resp. $\Sh(\ccal)$) the category of presheaves (resp. hypersheaves) of anima over $\mc{C}$.
We let $\Sh(\ccal,\mathrm{Sets})$ be the full subcategory of sheaves of sets. 
If $R$ is a ring in $\Sh(\mc{C},\mathrm{Sets})$, we let $\Sh(\mc{C},R)$ be the category of $R$-modules. When dealing with presheaves on a site with values in a $\infty$-category, we will always use hypersheaves and never sheaves.
For the sake of readability, we will sometime drop the "$\infty$" and simply refer to $\infty$-categories as categories (and the same for their condensed analogue defined in \Cref{part:condensed}).
Finally, all schemes we consider are qcqs.

\section{Solid sheaves on schemes.}\label{sec:solid}
In this section, we adapt \cite[Sections VII.1 \& VII.2]{Fargues-Scholze} to schemes. 
We begin by recalling a few preliminaries on pro-\'etale sheaves in \Cref{sec:proet_rec}.
In \Cref{sec:abelianproet} we define the abelian category of solid pro-\'etale sheaves and prove its main properties.
Then, in \Cref{sec:derivedsolid} we study the derived category of solid sheaves on a scheme.
Both of these sections closely follow the material in \cite[\S VII]{Fargues-Scholze}.
Finally in \Cref{sec:solidcats} we use ideas from condensed category theory to provide a different perspective on solid sheaves.
This approach will be used in \Cref{sec:rigidity} to solidify the categories of pro-\'etale motives constructed in the \Cref{sec:proetmotives}.

\subsection{Recollections on the pro-étale topos}\label{sec:proet_rec}
If $X$ is a scheme, we let $X_\proet$ denote the small pro-étale site of $X$ consisting of weakly \'etale $X$-schemes and by $X_\proet^{\aff}$ the subcategory made of cofiltered limits with affine transition maps of étale $X$-schemes which are also affine as schemes. Recall from \cite[\S 2]{zbMATH06479630} that an affine scheme is said to be \emph{$w$-contractible} if any pro-étale cover of it admits a section. It is said to be \emph{$w$-local} if any open cover splits and the subset of its closed points is closed. We often use \cite[Theorem 1.8]{zbMATH06479630}: an affine scheme is $w$-contractible if and only if it is $w$-local, all of its local rings at
closed points are strictly henselian, and its $\pi_0$ is an extremally disconnected profinite set.
Finally, a morphism of $w$-local affine schemes is \emph{$w$-local} if it maps closed points to closed points.

\begin{emptypar}
	\label{notation-tilde}
	Let $X = \Spec(R)$ be an affine scheme and $i \colon Z = \Spec(R/I) \rightarrow X$ be a closed immersion.
	Recall that by \cite[Lemma 2.2.12]{zbMATH06479630} the base change functor
	\[
		-\otimes_R R/I \colon \In(R_{\et}) \rightarrow \In((R/I)_{\et})
	\]
	admits a fully faithful right adjoint
	\[
		 \operatorname{Hens}_R(-) \colon  \In((R/I)_{\et}) \rightarrow \In(R_{\et}).
	\]
	As in loc. cit. we often simplify notation a bit and simply write
	\begin{align*}
		\widetilde{(-)} \colon Z_{\proet}^{\operatorname{aff}} &\rightarrow X_{\proet}^{\operatorname{aff}} \\
		V & \mapsto \widetilde{V}
	\end{align*}
	for the corresponding functor on schemes (that we also denote by $\mathrm{Hens}_X$).
	This functor is fully faithful and left adjoint to the base-change functor.
	It follows that for a pro-\'etale sheaf $F$ on $X$, the sheaf $i^*F$ may be described as the sheaf associated to the presheaf
	\[
		V \mapsto F(\widetilde{V}).
	\]
\end{emptypar}

\begin{prop}
    \label{prop:W-contractible_gives_hens_pair}
	Let $ W $ be a $ w $-contractible affine scheme and let $ i \colon W^c \to W $ be the inclusion of its subscheme of closed points.
	Then $ \mathrm{Hens}_{W}(W^c) = W $.
\end{prop}
\begin{proof}
	Let us write $ \widetilde{W^c} = \mathrm{Hens}_{W}(W^c) $. First recall that $W^c$ is $w$-contractible: it is $w$-local by \cite[Lemma 2.1.3]{zbMATH06479630}, its local rings are algebraically closed field and its $\pi_0$ is the same as that of $W$.
	By full faithfulness of the henselisation functor, we have $\widetilde{W^c}\times_W W^c=W^c$, whence a map $W^c \to \widetilde{W^c}$. Given a pro-étale cover $W'\to \widetilde{W^c}$,
	its pullback $W'\times_{\widetilde{W^c}} W^c\to W^c$ splits. As $W'\times_{\widetilde{W^c}} W^c=W'\times_{\widetilde{W^c}} (\widetilde{W^c}\times_W W^c)=W'\times_W W^c$, the splitting yields a map $\widetilde{W^c}\to W'$ by adjunction. It is a splitting of our pro-étale cover (by adjunction again).
	Thus $ \widetilde{W^c} $ is also $ w $-contractible.

	Furthermore the canonical map $ W^c \to \widetilde{W^c} $ is a closed immersion as it is the pullback of $W^c\to W$ along the map $\widetilde{W^c}\to W$ and $ W^c \to \widetilde{W^c} $ is a henselian pair (indeed, this follows from the universal property of $\mathrm{Hens}_W$ together with \cite[\href{https://stacks.math.columbia.edu/tag/09XI}{Tag~09XI (2)}]{stacks-project}).
	We claim that $ W^c \to \widetilde{W^c} $ is exactly the subscheme of closed points.
	Indeed, by \cite[Corollary 4.4]{MR4609461}, it induces an isomorphism on $ \pi_0 $ and the claim follows because any connected component of $ \widetilde{W^c} $ has a unique closed point since $ \widetilde{W^c} $ is $ w $-local.
	Since the canonical map
	\[
		\widetilde{W^c} \to W
	\]
	is pro-\'etale, $ w $-local and induces an isomorphism on closed points, it follows from the proof of \cite[lemma 2.3.8]{zbMATH06479630} that it is an isomorphism.
\end{proof}


\begin{rem}[Size issues]\label{rem:set_theory_continued}
    Since the category of weakly étale $ X $-schemes is not small this induces some set-theoretic issues.
    In the end, one can always circumvent these issues and they do not have any serious effect on our results.
    For the more cautious reader, we suggest one of the following two ways of reading this paper:
    \begin{enumerate}
        \item Fix once and for all two strongly inaccessible cardinals $ \delta < \epsilon $.
        All schemes are then assumed to be $\delta$-small and all categorical constructions, such as taking sheaves on a site, are taken with respect to the larger universe determined by $\epsilon$.
        In particular $ \Sh(X_{\proet}) $ always means hypersheaves of $\epsilon$-small anima on $\delta$-small weakly étale $ X $-schemes, and similarly for the $\infty$-category of condensed anima $\Cond(\Ani)$.
        
        \item If the reader does not  want to work with universes, they may proceed as follows.
        For a scheme $ X $, choose a strong limit cardinal $ \kappa $ such that $ X $ is $ \kappa $-small.
        Write $X_{\proet,\kappa}$ for the category of $ \kappa $-small weakly étale $ X $-schemes.
        The assumption that $ \kappa $ is a strong limit cardinal guarantees that there are enough $w$-contractible schemes in $X_{\proet,\kappa} $.
        We then define
        \begin{equation*}
           \Sh(X_{\proet}) \coloneqq  \colim_\kappa  \Sh(X_{\proet,\kappa}) 
        \end{equation*}
        and similarly for the category of condensed anima.
        This is the approach taken by Clausen and Scholze in \cite{condensedpdf}.
    \end{enumerate}
\end{rem}

\begin{lem}
	\label{lem:strictly_profinite_cover}
	Let $X$ be a quasi-compact scheme.
	Then there is a surjective morphism of schemes $ f \colon E \to X$ where $ E $ is reduced, $0$-dimensional, affine and all its local rings are separably closed fields.
	Furthermore for a such a scheme $E$ we have a canonical equivalence
	\[
		\Sh(E_{\proet}) \simeq \Cond(\Ani)_{/\pi_0(E)}.
	\]
\end{lem}
\begin{proof}
	Since $X$ is quasi-compact, we can assume $X$ to be affine and reduced. By \cite[Remark 2.1.11]{zbMATH06479630} there exist maps 
	\[(X^Z)^c\xrightarrow{a} X^Z\xrightarrow{b} X\] with $b$ a pro-open covering and $a$ a  closed immersion, with $X^Z$ a $w$-local scheme and $(X^Z)^c$ its subscheme of closed points. In particular by \cite[Lemma 2.2.3]{zbMATH06479630} the scheme $(X^Z)^c$ is absolutely flat (\emph{i.e.} reduced and $0$-dimensional).
	
	The existence of $E$ then follows from \cite[Lemma~2.2.7]{zbMATH06479630}. Now, the pro-étale site of $E$ is the same as the pro-Zariski site of $E$ by \cite[Lemma~2.3.8]{zbMATH06479630} and as $E$ is $0$-dimensional, the map $E\to \pi_0(E)$ is a homeomorphism so that the lemma follows.
\end{proof}

\begin{cor}
    \label{cor:global_sections_on_wcontr}
    Let $ W $ be a $ w $-contractible scheme and let $ i \colon W^c \to W $ be the inclusion of its subscheme of closed points.
    Then for any $F \in \Sh(W_\proet)$ the canonical map
    \[
        \eta \colon i^*F(W^c) \to F(W)
    \]
    is an equivalence.
\end{cor}
\begin{proof}
    By \Cref{prop:W-contractible_gives_hens_pair}, the map $\eta$ is an equivalence if $F$ is representable by a pro-affine-étale $W$-scheme.
    Since $i^*$, $\Gamma(W^c,-)$ and $\Gamma(W,-)$ preserves sifted colimits, and an arbitrary pro-étale sheaf is a sifted colimit of pro-affine-étale representables, the claim follows.
\end{proof}


\begin{lem}
    \label{lem:conservativity-of-pullbacks}
    Let $f \colon X \to Y$ be a surjective morphism of schemes.
    Then the pullback functor $f^{*} \colon \Sh(Y_{\proet}) \to \Sh(X_{\proet})$ is conservative.
\end{lem}
\begin{proof}
	We may immediately reduce to the case where both $X$ and $Y$ are affine.
	Since $w$-contractible schemes form a basis for the topology the collection of functors given by evaluating at any $w$-contractible $W$ is conservative.
	Using \Cref{cor:global_sections_on_wcontr} it follows that the collection of functors given by pulling back along the  maps
	\[
		W^c \to W \to X
	\]
	for any $w$-contractible $W$ is conservative.
	Hence, we may assume that $W^c = Y$ to reduce to the case where $Y$ is reduced of Krull dimension 0 with separably closed residue fields.
	By \Cref{lem:strictly_profinite_cover}, we may choose another such scheme $E$ with a surjection $E \to X$ to reduce to the case $X = E$.
	Again by \Cref{lem:strictly_profinite_cover} the functor $f^*$ identifies with the functor
	\[
		\Cond(\Ani)_{/\pi_0(X)} \to \Cond(\Ani)_{/\pi_0(Y)}
	\]
	induced by pulling back along $\pi_0(f) \colon \pi_0(X) \to \pi_0(Y)$.
	But by construction the map $\pi_0(f)$ is surjective and thus an effective epimorphim in condensed anima and so pulling back along it is conservative.
\end{proof}

\begin{lem}
    \label{lem:pullbackslimits}
    Let $f \colon X \to Y$ be a morphism of schemes.
    Then the pullback functor $f^{*} \colon \Sh(Y_{\proet}) \to \Sh(X_{\proet})$ is compatible with limits.
\end{lem}
\begin{proof}
	This is \cite[Corollary~4.18]{zbMATH07671238}.
\end{proof}

\subsection{Abelian categories.}\label{sec:abelianproet}
Let $X$ be a qcqs scheme. We fix a set $\mathcal{P}$ of prime numbers in this section. Recall the ring $\ZhatP$ from the introduction. 
The ring $\ZhatP$ is a topological ring in the obvious way and thus we may consider it as a condensed ring.
Via the canonical morphism of topoi
\[
	\Cond(\rm{Set}) \to \Sh(X_{\proet},\rm{Set}),
\]
see for example \cite[Remark 3.5]{bastietal}, we may therefore view $\ZhatP$ as a ring object in $\Sh(X_{\proet},\rm{Set})$.
We write $\Sh(X_{\proet},\ZhatP) = \Mod_{\ZhatP}(\Sh(X_{\proet},\rm{Ab}))$ and $\D(X_\proet,\ZhatP)$ for the derived category of this abelian category.
\begin{rec}
	\label{rec:proet-VS-pro-aff-et}
Recall that by \cite[Exposé VII, Théorème 5.7]{sga4} we have a fully faithful map:
\[\mathrm{Pro}(X_{\et}^{\mathrm{aff},{\fp}})\to X_\proet\]
where the left-hand side denotes the full subcategory of $\mathrm{Pro}(X_{\et})$ spanned by cofiltered limits of finite presentation étale $X$-schemes which are affine. Moreover, this induces an equivalence between the associated topoi by \cite[Theorem 2.3.4]{zbMATH06479630}.
\end{rec}

If $j\colon U\to X$ is a pro-étale map, the pullback $j^*$ on pro-étale sheaves has a left adjoint that we denote by $j_\sharp$.

\begin{defi}
	\label{defi:solid-generator}
	If $j\colon U\to X$ is a pro-étale map that can be written as a cofiltered limit of étale maps $j_\alpha \colon U_\alpha \to X$ with each $U_\alpha$ affine, we set \[j_\sharp^\blacksquare\ZhatP\coloneqq \lim (j_\alpha)_\sharp \ZhatP \in \Sh_{\proet}(X,\ZhatP).\]
	We have a map $j_\sharp \ZhatP\to j_\sharp^\blacksquare \ZhatP$.
\end{defi}

\begin{defi}
	\label{defi:solidSheaves}
	\label{defi:solid}
	Let $F$ be an abelian pro-étale sheaf of $\ZhatP$-modules. We say that $F$ is \emph{solid}, if for every $j$ as in \Cref{defi:solid-generator}, the natural map \[\Hom_{\Sh(X_\proet,\ZhatP)}(j_\sharp^\blacksquare\ZhatP,F)\to \Hom_{\Sh(X_\proet,\ZhatP)}(j_\sharp\ZhatP,F)=F(U)\] is an equivalence.
	We denote by $\Sh(X,\ZhatP)^{\blacksquare}$ the full subcategory of $\Sh(X_\proet,\ZhatP)$ spanned by solid objects.

	Finally, we say that a complex of pro-étale sheaves of $\ZhatP$-modules is \emph{solid} if its cohomology sheaves are solid. We denote by $\D(X,\ZhatP)^{\blacksquare}$ the full subcategory of $\D(X_\proet,\ZhatP)$ spanned by the solid objects.
\end{defi}

\begin{rem}
Equivalently $\Sh(X,\ZhatP)^{\blacksquare}$ is the reflective localization at the maps $j_\sharp\ZhatP \to  j_\sharp^\blacksquare\ZhatP$.
	In particular the inclusion $ \Sh(X,\ZhatP)^{\blacksquare} \to \Sh(X_\proet,\ZhatP)$ admits a left adjoint that we denote by $(-)^\blacksquare$ and call the \emph{solidification functor}.
\end{rem}

We now want to understand the category of solid sheaves. Our main result is Theorem~\ref{thm:solid} and to that end, we first need a few technical lemmas.

\begin{emptypar}
	Recall that an étale sheaf of sets $F$ is \emph{constructible} if there is a stratification of $X$ by quasi-compact locally closed subschemes such that $F$ is locally constant with finite stalks over each stratum; we let $X_{\et}^\mathrm{cons}$ be the full subcategory of $\Sh(X_\et,\mathrm{Sets})$ made of constructible étale sheaves. We denote by $\nu \colon \Sh(X_\proet,\mathrm{Sets})\to \Sh(X_\et,\mathrm{Sets})$ the morphism of sites induced by the inclusion and by $\ZhatP[-]$ the left adjoint to the forgetful functor $\Sh(X_\proet,\ZhatP)\to \Sh(X_\proet,\mathrm{Sets})$.
\end{emptypar}
In many reductions, we need the following result on schemes.

\begin{lem}
    \label{lem:ZhatconstrIslimTors}
    Let $F$ be a constructible étale sheaf of sets. Then in $\Sh(X_\proet,\ZhatP)$ the map
    \[\ZhatP[\nu^*F]\to \lim_{n\in \Pi\mc{P}}\Z/n\Z[\nu^*F],\] where $\Pi\mc{P}$ is the set of positive integers, ordered by divisibility, that can be written as products of primes in $\mc{P}$, is an equivalence.
\end{lem}
\begin{proof}
	We can assume that $F$ is étale-locally constant: indeed by \cite[Lemma 6.1.11]{zbMATH06479630}, given a finite stratification of $X$ with quasi-compact strata, the family of restrictions to the strata is jointly conservative and these restrictions commute with limits by Lemma~\ref{lem:pullbackslimits}.

    Take a surjective map $g\colon E\to X$ with $E$ affine reduced of Krull dimension $0$ with separably closed residue fields using \Cref{lem:strictly_profinite_cover}.
    Since $g^*$ is conservative and commutes with limits by Lemma~\ref{lem:conservativity-of-pullbacks} and Lemma~\ref{lem:pullbackslimits}, we may replace $X$ by $E$.

    In that case, the Zariski and the étale site coincide by \cite[Corollary~2.5]{Schroer} and thus $F$ is Zariski-locally constant with finite stalks.
    Thus we may assume it is actually constant, in which case the claim is obvious, since cofiltered limits commute with finite sums.
\end{proof}

An étale sheaf of abelian groups is said to be \emph{$\mc{P}$-torsion constructible} if it is $\mc{P}$-torsion and constructible as an étale sheaf of sets. We let $\Sh_\mc{P}^\mathrm{cons}(X_\et)$ be the full subcategory of $\Sh(X_\et,\Z)$ made of $\mc{P}$-torsion constructible étale sheaves.

We have the following analogue of \cite[Proposition VII.1.6]{Fargues-Scholze} in our setting:
\begin{lem}
\label{limIsRlim}
	Let $(F_i)$ be a cofiltered inverse system of
$\mc{P}$-torsion constructible étale sheaves. Then for all $q>0$, the $R^q\lim F_i$ vanishes when taken in the category $\Sh(X_\proet,\ZhatP)$. 
\end{lem}
\begin{proof}
    It suffices to prove that for any $w$-contractible weakly étale $X$-scheme $W$ and any $q$, the abelian group $[R^q\lim F_i](W)$ vanishes.
    By \Cref{cor:global_sections_on_wcontr} and \Cref{lem:pullbackslimits}, we may replace $X$ by the set of closed points $W^c$ of $W$ and therefore assume that $X$ is affine reduced of Krull dimension $0$ with separably closed residue fields.
    By \cite[Corollary~2.5]{Schroer}, we get that $\Sh(X_\et) = \Sh(|X|)$ where $|X|$ the underlying topological space of $X$ which is extremally disconnected.
    Now the same proof as in the end of \cite[Proposition VII.1.6]{Fargues-Scholze} applies.
\end{proof}
\begin{thm} \label{thm:solid} Keep the above notations.
	\begin{enumerate}
		\item The subcategory $\Sh(X,\ZhatP)^{\blacksquare}$ of $\Sh(X_\proet,\ZhatP)$ is weak Serre and closed under small limits and colimits. For $j\colon U \to X$  as above, the object $j_\sharp^\blacksquare \ZhatP$  is the solidification of $j_\sharp\ZhatP$. The collection of the $j_\sharp^\blacksquare \ZhatP$ generates the abelian category $\Sh(X,\ZhatP)^{\blacksquare}$.
		\item Let $F$ be an object of $\Sh(X_\proet,\ZhatP)$. The following are equivalent
		\begin{enumerate}
			\item The sheaf $F$ is compact in $\Sh(X_\proet,\ZhatP)$ and solid.
			\item The sheaf $F$ is solid and compact in $\Sh(X,\ZhatP)^{\blacksquare}$.
			\item The sheaf $F$ can be written as a cofiltered limit of torsion constructible étale sheaves.
		\end{enumerate}
		\item 
		The subcategory of $\Sh(X_\proet,\ZhatP)$ spanned by objects satisfying the above equivalent conditions is weak Serre. As an abelian category it is equivalent to the pro-category $\mathrm{Pro}(\Sh_\mc{P}^\mathrm{cons}(X_\et))$. 
		\item The category $\Sh(X,\ZhatP)^{\blacksquare}$ is equivalent to $\In (\mathrm{Pro}(\Sh_\mc{P}^\mathrm{cons}(X_\et)))$.
		\item Let $K$ be a complex of pro-étale sheaves of $\ZhatP$-modules.
		Then, it is solid if and only if the natural maps \[\underline{\Hom}_{\D(X_\proet,\ZhatP)}(j_\sharp^\blacksquare\ZhatP,K)\to j_*j^*K\] are equivalences for any $(j \colon U = \lim_\alpha U_\alpha\to X) \in \mathrm{Pro}(X_{\et}^{\mathrm{aff},{\fp}})$.
	\end{enumerate}
\end{thm}


\begin{proof}
We mostly follow the proof of \cite[Theorem VII.1.3]{Fargues-Scholze}.
The category $\mathrm{Pro}(\Sh_\mc{P}^\mathrm{cons}(X_\et))$ is abelian by \cite[Proposition 8.6.7]{MR2182076}. Furthermore, \Cref{limIsRlim} ensures that the functor \begin{equation}\label{eq:ProCons-in-Proet}\mathrm{Pro}(\Sh_\mc{P}^\mathrm{cons}(X_\et)) \to \Sh(X_\proet,\ZhatP)\end{equation} is exact. We want to show that it is fully faithful. Let $(F_i)$ be a cofiltered inverse system of
$\mc{P}$-torsion constructible étale sheaves and let $G$ be a $\mc{P}$-torsion constructible étale sheaf, we have to prove that the natural map:
\[\colim\Hom_{\Sh(X_\et,\Z)}(F_i,G)\to \Hom_{\Sh(X_\proet,\ZhatP)}(\lim \nu^*F_i,\nu^*G)\] is an equivalence.

 Recall the category $X_{\et}^\mathrm{cons}$ of set-valued constructible étale sheaves.
 We consider the category $\Sh(\mathrm{Pro}(X_{\et}^\mathrm{cons}),\mathrm{Sets})$ of sheaves on $\mathrm{Pro}(X_{\et}^\mathrm{cons})$ for the coherent topology, \cite[Remark~7.1.1, Definition B.5.3]{ultracategories}.
 By \cite[Example 7.1.7]{ultracategories} the natural map \[\mathrm{Pro}(X_{\et}^\mathrm{cons})\to \Sh(X_\proet,\mathrm{Sets})\] induces an equivalence $\Sh(\mathrm{Pro}(X_{\et}^\mathrm{cons}),\mathrm{Sets})\to \Sh(X_\proet,\mathrm{Sets})$ by Kan extension.
 In particular, it is fully faithful and therefore, the map \[\colim\Hom_{\Sh(X_\et,\mathrm{Sets})}(F_i,G)\to \Hom_{\Sh(X_\proet,\mathrm{Sets})}(\lim \nu^*F_i,\nu^*G)\] is an equivalence, since the coherent topology is subcanonical  \cite[Corollary B.5.6]{ultracategories}.
 Enforcing compatibility with addition on the right hand side amounts to saying that two different maps $\lim \nu^* F_i \times \lim \nu^* F_i\to G$ coincide (and similarly on the left hand side); as before, we have \[\colim\Hom_{\Sh(X_\et,\mathrm{Sets})}(F_i \times F_i,G)\xrightarrow{\sim} \Hom_{\Sh(X_\proet,\mathrm{Sets})}(\lim \nu^*F_i\times \lim \nu^*F_i,\nu^*G)\]
 and taking equalizers, we get \[\colim\Hom_{\Sh(X_\et,\Z)}(F_i,G)\xrightarrow{\sim} \Hom_{\Sh(X_\proet,\Z)}(\lim \nu^*F_i,\nu^*G).\]
 Furthermore, any map between cofiltered limits of constructible $\mc{P}$-torsion sheaves that is $\Z$-linear is $\ZhatP$-linear.
Hence, the category $\mathrm{Pro}(\Sh_\mc{P}^\mathrm{cons}(X_\et))$ is equivalent to an abelian subcategory of $\Sh(X_\proet,\ZhatP)$ that we denote by $\mathcal{C}(X)$.

By definition, the category $\mc{C}(X)$ is spanned by the objects that satisfy condition (c). We now show that its objects are solid. It suffices to prove that $\mc{P}$-torsion constructible sheaves are solid since solid objects are closed under limits. Let $F$ be a $\mc{P}$-torsion constructible étale sheaf and let $(U_\alpha)$ be a cofiltered system of étale $X$-schemes which are affine. Denote by $j\colon\lim_\alpha U_\alpha\to X$ the canonical pro-étale map. In the commutative diagram
\[\begin{tikzcd}
	{\Hom_{\Sh(X_\proet,\ZhatP)}(j_\sharp^\blacksquare\ZhatP,\nu^*F)} & {\Hom_{\Sh(X_\proet,\ZhatP)}(j_\sharp\ZhatP,\nu^*F)} \\
	{\colim_\alpha F(U_\alpha)} & {(\nu^*F)(\lim_\alpha U_\alpha)}
	\arrow[from=1-1, to=1-2]
	\arrow[from=1-1, to=2-1]
	\arrow[from=1-2, to=2-2]
	\arrow[from=2-1, to=2-2]
\end{tikzcd},\] the right vertical arrow is invertible by the Yoneda lemma,
the bottom horizontal arrow is invertible because of \cite[Lemma 5.2]{zbMATH06479630},
and one sees that the left vertical arrow is invertible by using \Cref{lem:ZhatconstrIslimTors} to write $j_\sharp^\blacksquare\ZhatP$ as an object of $\mc{C}(X)$, and the fact that \Cref{eq:ProCons-in-Proet} is fully faithful. Thus, the top horizontal arrow is an equivalence and this proves that $\nu^*F$ is solid. Hence $\ccal(X)$ is contained in $\Sh(X,\ZhatP)^{\blacksquare}$.

We now prove that the objects of $\ccal(X)$ are also compact in $\Sh(X_\proet,\ZhatP)$. First note that the objects of $\mathrm{Pro}(X_{\et}^\mathrm{cons})$ are compact in $\Sh(\mathrm{Pro}(X_{\et}^\mathrm{cons}),\mathrm{Sets})\simeq \Sh(X_\proet,\mathrm{Sets})$; this is because filtered colimits of sheaves are computed in the category of set-valued presheaves over $\mathrm{Pro}(X_{\et}^\mathrm{cons})$ by \cite[Remark 7.1.1]{ultracategories}. Let $F$ be in $\ccal(X)$ and let $G$ be a pro-étale sheaf of $\ZhatP$-modules.
As above, one can describe $\Hom_{\Sh(X_\proet,\ZhatP)}(F, G)$ as the subset of $\Hom_{\Sh(X_\proet)}(F, G)$ made of maps satisfying additivity and $\ZhatP$-linearity, \textit{i.e} certain maps $F\times F\to G$ and $F\times \ZhatP \to G$ agree; as the functor $\Hom_{\Sh(X_\proet)}(H,-)$ is compatible with filtered colimits for any $H$ in $\mathcal{C}(X)$, this description also commutes with filtered colimits and therefore $F$ is compact in $\Sh(X_\proet,\ZhatP)$. Hence, we proved that condition (c) implies condition (a).

We can also deduce that solid objects are closed under filtered colimits: indeed, if $j\colon U\to X$ is a pro-étale map that can be written as a cofiltered limit of étale maps from affine schemes, the sheaf $j_\sharp^\blacksquare \ZhatP$ is compact in $\Sh(X_\proet,\ZhatP)$ as we already know that it belongs to $\ccal(X)$. This also proves that condition (a) implies condition (b).

We now prove 4. \textit{i.e.}, we describe the category of solid sheaves.
Since solid objects are closed under filtered colimits, we have a functor \[\In \ccal(X)\to \Sh(X,\ZhatP)^{\blacksquare}.\]
As all objects of $\ccal(X)$ are compact, it is fully faithful.
It is also an exact functor of abelian categories by \cite[Proposition 8.6.7]{MR2182076}.

As $j_\sharp^\blacksquare \ZhatP$ is in $\ccal(X)$, it is solid. 
Thus, the Yoneda lemma and the definition of solid objects yield
\[
      (j_\sharp\ZhatP)^\blacksquare\xrightarrow{\sim}j_\sharp^\blacksquare \ZhatP.
\]
Thus any solid sheaf $F$ admits a surjection from a direct sum of $j_\sharp^\blacksquare \ZhatP$ for various $j$ as before.
As the kernel is also solid, we get a presentation \[\bigoplus (j_\beta)_\sharp^\blacksquare \ZhatP\to \bigoplus (j_\gamma)_\sharp^\blacksquare \ZhatP\to F \to 0\] which implies that $F$ belongs to $\In \ccal(X)$, since the latter is closed under cofibers.
From this, we also deduce that condition (b) implies condition (c) because $\ccal(X)$ is idempotent complete and using that taking compact objects of the Ind-category is the same as taking the idempotent-completion by \cite[Lemma~5.4.2.4]{MR2522659}. Thus we proved 2.

So far we proved 2. and 4., and everything in points 1. and 3. except for the stabilities under extensions.
This follows from 5., which we now prove.
We begin by showing that any solid complex $K$ satisfies the condition of 5.
By dévissage along the Postnikov tower (note that the formula commutes with limits in $K$ as $j^*$ commutes with limits by \Cref{lem:pullbackslimits}) and a spectral sequence argument, we can reduce to $K=F[0]$ being in the heart. 
The strategy is as follows: we prove that the equation of 5. commutes with filtered colimits so that we may assume that $F$ is a constructible torsion étale sheaf, and then we prove the claim for those specific sheaves. Both steps use Breen-Deligne resolutions:

By \cite[Lemma 5.3]{MR4609461} the right-hand side of 5. commutes with filtered colimits of objects in the heart.
We claim that the left-hand side $\underline{\Hom}_{\D(X_\proet,\ZhatP)}(j_\sharp^\blacksquare\ZhatP,-)$ also commutes with filtered colimits in $\Sh(X,\ZhatP)^{\blacksquare}$.
To prove this claim, note that the $p_\sharp \ZhatP$ for $p\colon V\to X$ weakly étale and $V$ $w$-contractible are compact generators of $\D(X_\proet,\ZhatP)$ (compactness follows from \cite[Lemma 4.1]{MR4609461}) and therefore, it suffices to show that the $\Hom_{\D(X_\proet,\ZhatP)}(p_\sharp\ZhatP,\underline{\Hom}_{\D(X_\proet,\ZhatP)}(j_\sharp^\blacksquare\ZhatP,-))$ are compatible with filtered colimits.
The latter coincides with $\Hom_{\D(X_\proet,\ZhatP)}(p_\sharp\ZhatP\otimes j_\sharp^\blacksquare\ZhatP,-)$.
Note that $p_\sharp\ZhatP\otimes j_\sharp^\blacksquare\ZhatP$ is the non-derived tensor product because $p_\sharp\ZhatP$ is projective.
We now use a Breen-Deligne resolution of $j_\sharp^\blacksquare\ZhatP$ which implies as in \cite[Corollary 4.8]{condensedpdf} that we have a functorial converging spectral sequence $E_*^{p,q}(j_\sharp^\blacksquare\ZhatP^{r_{p,i}}\times V,-)$ whose first page is of the form
\[
E_1^{p,q}=\prod_{i=1}^{n_p}\Ext^q_{\Sh(X_\proet,\ZhatP)}(\ZhatP[j_\sharp^\blacksquare\ZhatP^{r_{p,i}}\times V],-)\Rightarrow \HH^{p+q}\Hom_{\D(X_\proet,\ZhatP)}(p_\sharp\ZhatP\otimes j_\sharp^\blacksquare\ZhatP,-),
\] with $n_p$ and $r_{p,i}$ some universal constants.
Note that the sheaf $j_\sharp^\blacksquare\ZhatP^{r_{p,i}}\times V$ is a cofiltered limit of affine étale $X$-schemes by \Cref{lem:ZhatconstrIslimTors}.
Thus, the left-hand side commutes with filtered colimits by \cite[\href{https://stacks.math.columbia.edu/tag/0739}{Tag 0739}]{stacks-project} and thus the abutment also commutes with filtered colimits, which proves the claim.
Hence, we can assume that $F$ is compact.
We can then further reduce to $F$ being a $\mc{P}$-torsion constructible étale sheaf as both sides of 5. commute with limits (which are derived in our case by \Cref{limIsRlim}).

We now will prove that 5. is an equivalence for $K=F[0]$ and $F$ a $\mc{P}$-torsion constructible étale sheaf. Again it suffices to check that the map is an equivalence after applying $\mathrm{Hom}_{\D(X_\proet,\ZhatP)}(p_\sharp \ZhatP,-)$ for all $p\colon V\to X$ pro-étale affine with $V$ $w$-contractible. Note that the right hand-side of 5. then becomes $\mathrm{Hom}_{\D(X_\proet,\ZhatP)}(p_\sharp\ZhatP \otimes j_\sharp\ZhatP,F[0])$ by using adjunctions and the projection formula for $j_\sharp$. Using again Breen-Deligne resolutions to get simultaneous resolutions of $j_\sharp^\blacksquare\ZhatP$ and $(j_\alpha)_\sharp\ZhatP$ we obtain as above a map 
$$\colim_\alpha E_*^{p,q}((j_\alpha)_\sharp\ZhatP^{r_{p,i}}\times V,F)\to E_*^{p,q}(j_\sharp^\blacksquare\ZhatP^{r_{p,i}}\times V,F)$$
of spectral sequences which on the $E_1$-page is of the form

\[\colim_\alpha\prod_{i=1}^{n_p}\Ext^q_{\Sh(X_\proet,\ZhatP)}(\ZhatP[(j_\alpha)_\sharp\ZhatP^{r_{p,i}}\times V],F)\to \prod_{i=1}^{n_p}\Ext^q_{\Sh(X_\proet,\ZhatP)}(\ZhatP[j_\sharp^\blacksquare\ZhatP^{r_{p,i}}\times V],F).\] 
This map is an equivalence for all $p,q\in\Z$. Indeed, since $j_\sharp^\blacksquare\ZhatP^{r_q}\times V = \lim_{\alpha} (j_\alpha)_\sharp\ZhatP^{r_q}\times V$, this follows from \cite[Corollary~5.1.6]{zbMATH06479630} because the $j_\sharp^\blacksquare\ZhatP^{r_q}\times V$ and $(j_\alpha)_\sharp\ZhatP^{r_q}\times V$ are all limits of affine schemes by \Cref{lem:ZhatconstrIslimTors}. As the two spectral sequences are equivalent, so are their abutments, thus the map 
\[\colim_\alpha \HH^n\Hom_{\D(X_\proet,\ZhatP)}(p_\sharp\ZhatP\otimes (j_\alpha)_\sharp\ZhatP,F) \to \HH^n\Hom_{\D(X_\proet,\ZhatP)}(p_\sharp\ZhatP\otimes j_\sharp^\blacksquare\ZhatP,F)\] is an equivalence. 
Using the projection formula for $(j_\alpha)_\sharp$, the left-hand side can be written 
\[\colim_\alpha \HH^n\Hom_{\D(X_\proet,\ZhatP)}((u_\alpha)_\sharp\ZhatP,F) \simeq \colim_\alpha \HH^n_\et(W_\alpha,F_{\mid W_\alpha})\] with the notation of the following cartesian diagram: \[\begin{tikzcd}
	{W_\alpha} & V \\
	{U_{\alpha}} & X
	\arrow[from=1-1, to=1-2]
	\arrow[from=1-1, to=2-1]
	\arrow["{u_\alpha}"{description}, from=1-1, to=2-2]
	\arrow["p", from=1-2, to=2-2]
	\arrow["{j_\alpha}"', from=2-1, to=2-2]
\end{tikzcd}.\] By \cite[\href{https://stacks.math.columbia.edu/tag/09YQ}{Theorem 09YQ}]{stacks-project} we have 
\[\colim_\alpha \HH^n_\et(W_\alpha,F_{\mid W_\alpha})\simeq \HH^n_\et(W,F_{\mid W})\] with $W$ the base change of $p$ along $j$, thus doing the previous operations in reverse we find 
\[\HH^n\Hom_{\D(X_\proet,\ZhatP)}(p_\sharp\ZhatP\otimes j_\sharp^\blacksquare\ZhatP,F) \xrightarrow{\sim} \HH^n\Hom_{\D(X_\proet,\ZhatP)}(p_\sharp\ZhatP\otimes j_\sharp\ZhatP,F)\]
which is what we wanted to prove.

We now prove the converse of 5.
Let us write $\D(X,\ZhatP)^{L\blacksquare}$ for the full subcategory spanned by those complexes of pro-étale sheaves of $\ZhatP$-modules $K$ such that for every $j$ as before, the map
\[\Hom_{\D(X_\proet,\ZhatP)}(j_\sharp^\blacksquare\ZhatP,K)\to \Gamma(U,K) \]
is an equivalence.
Clearly all complexes satisfying the conclusion of 5. are in $\D(X,\ZhatP)^{L\blacksquare}$.
In particular we have $\D(X,\ZhatP)^{\blacksquare} \subseteq \D(X,\ZhatP)^{L\blacksquare}$ by what we have just shown.
To show that they agree we observe that the inclusion $\D(X,\ZhatP)^{L\blacksquare} \hookrightarrow \D(X_\proet,\ZhatP)$ admits a left adjoint $L$.
Note that $D(X,\ZhatP)^{\blacksquare}$ is closed under colimits in $\D(X_\proet,\ZhatP)$.
Thus we are done if we prove that the natural map $L(j_\sharp \ZhatP) \to  j_\sharp^\blacksquare \ZhatP$ is an equivalence, since the latter is contained in $D(X,\ZhatP)^{\blacksquare}$.
Since $j_\sharp^\blacksquare \ZhatP$ is in $\D(X,\ZhatP)^{L\blacksquare}$ it suffices to show that for any $ K \in \D(X,\ZhatP)^{L\blacksquare}$ the canonical map
\[
    \Hom_{\D(X,\ZhatP)^{L\blacksquare}}(j_\sharp^\blacksquare \ZhatP,K) \to \Hom_{\D(X,\ZhatP)}(j_\sharp \ZhatP,K)
\]
is an equivalence, which is true by definition.
\end{proof}

\begin{cor}\label{torsion are solid}
	If $F$ is an abelian étale sheaf of $\Z/n\Z$-modules on $X_{\et}$ with $n\in \Pi\mathcal{P}$ (see \Cref{lem:ZhatconstrIslimTors} for this notation), the sheaf $\nu^*F$ is solid.
\end{cor}
\begin{proof}
\Cref{thm:solid} shows that this is true if $F$ is constructible. As solid sheaves are closed under colimits and $\nu^*$ is compatible with colimits, the result is true because any sheaf of $\Z/n\Z$-modules is a colimit of constructible sheaves by \cite[\href{https://stacks.math.columbia.edu/tag/03SA}{Lemma 03SA}]{stacks-project}.
\end{proof}

\begin{rem}
	\label{rem:diff-generators}
	In \Cref{defi:solid-generator} we have chosen to work with inverse limits of affine \'etale $X$-schemes.
	We could have also used other generators of the pro-\'etale site instead.
	To see this we need to make some preliminary observations. We consider the functor
	\[
		\ZhatP[-] \colon X^{\rm{cons}}_{\et} \to \mathrm{Pro}(\Sh_\mc{P}^\mathrm{cons}(X_\et))
	\]
	sending a constructible sheaf of sets $\mathcal{F}$ to the free $\ZhatP$-module on it.
	Note that this functor does indeed land in the full subcategory $ \mathrm{Pro}(\Sh_\mc{P}^\mathrm{cons}(X_\et)) \subseteq \Sh(X,\ZhatP)^{\blacksquare}$ by \Cref{lem:ZhatconstrIslimTors}.
	Furthermore it preserves finite colimits.
	Since cofiltered limits commute with finite colimits in $\mathrm{Pro}(\Sh_\mc{P}^\mathrm{cons}(X_\et))$, it follows that the extension to pro-objects
	\[
		\rm{Pro}(\ZhatP[-]) \colon \rm{Pro}(X^{\rm{cons}}_{\et}) \to \mathrm{Pro}(\Sh_\mc{P}^\mathrm{cons}(X_\et))
	\]
	also commutes with finite colimits.
	Now by \Cref{limIsRlim} the obvious functor \begin{equation*}\mathrm{Pro}(\Sh_\mc{P}^\mathrm{cons}(X_\et)) \to \Sh(X_\proet,\ZhatP)\end{equation*} is exact so in particular it also preserves finite colimits.
	It follows that the composite
	\[
		\rm{Pro}(X^{\rm{cons}}_{\et}) \xrightarrow{\rm{Pro}(\ZhatP[-])} \mathrm{Pro}(\Sh_\mc{P}^\mathrm{cons}(X_\et)) \to \Sh(X_\proet,\ZhatP)
	\]
	is exact as well.

	Now observe that if $\mathcal{F} = \lim_i \mathcal{F}_i \in \rm{Pro}(X^{\rm{cons}}_{\et}) $ this functor, by definition sends $\mathcal{F}$ to $\lim_i  \ZhatP[\mathcal{F}_i]$.
	In particular if $\mathcal{F}$ is represented by some $j \colon U \to  X$ in $\mathrm{Pro}(X_{\et}^{\mathrm{aff},{\fp}})$, it agrees with $j_\sharp^\blacksquare \ZhatP$.
	It is a consequence of \cite[Remark 7.1.7]{ultracategories} that all objects in $\rm{Pro}(X^{\rm{cons}}_{\et})$ are qcqs objects of the topos $\Sh(X_{\proet},\rm{Set})$.
	In particular we may write them as a finite colimit of representables $j \colon U \to X$ in $\mathrm{Pro}(X_{\et}^{\mathrm{aff},{\fp}})$.
	This shows that the canonical comparison map
	\[
		\ZhatP[\mathcal{F}]^\blacksquare \to \lim_i  \ZhatP[\mathcal{F}_i] \in \Sh(X,\ZhatP)^\blacksquare
	\]
	writes as a finite colimits of maps of the form $ (j_\sharp \ZhatP)^\blacksquare \to j_\sharp^\blacksquare \ZhatP $ and is thus an equivalence.
	This shows that the full subcategory $\Sh(X,\ZhatP)^\blacksquare \subseteq \Sh(X_{\proet},\ZhatP) $ is equivalently the reflective localization at the larger class of maps of the form \[\ZhatP[\fcal]\to \ZhatP[\fcal]^\blacksquare\simeq \lim_i\ZhatP[\fcal_i].\]

	The above considerations also show that we could have worked with a smaller class of maps instead.
	Since $w$-contractible schemes which are in  $\mathrm{Pro}(X_{\et}^{\aff,\fp})$ also generate the topos $\Sh(X_{\proet},\rm{Set})$ under colimits, the same arguments as above show that $\Sh(X,\ZhatP)^\blacksquare$ is also equivalent to the reflective localization at maps of the form
	\[
		  j_\sharp \ZhatP \to j_\sharp^\blacksquare \ZhatP 
	\]
	for $j \colon U \to X$ in $\mathrm{Pro}(X_{\et}^{\aff,\fp})$ and $w$-contractible.
\end{rem}

\begin{prop}
	\label{lem:pullback-preserves-solid}
	\label{lem:f*-commutes-with-iL}
	\label{prop:pullback_of_solid_remains_solid}

	Let $f\colon Y\to X$ be a map of qcqs schemes. Then \[f^*\colon \Sh(X_\proet,\ZhatP)\to\Sh(Y_\proet,\ZhatP)\] preserves solid sheaves. In fact, if $j\colon U\to X$ is affine pro-étale and if $j_Y\colon V=U\times_X Y\to Y$ is its base change along $f$, we have \[f^*(j_\sharp^\blacksquare\ZhatP)\simeq (j_Y)_\sharp^\blacksquare\ZhatP.\] As a consequence, if we let $L^\blacksquare$ be the composition \[\Sh(X_\proet,\ZhatP)\xrightarrow{(-)^\blacksquare} \Sh^{\blacksquare}(X_\proet,\ZhatP)\xhookrightarrow{i} \Sh(X_\proet,\ZhatP),\] then the natural transformation \[L^\blacksquare\circ f^*\to f^*\circ L^\blacksquare\] is an equivalence.
\end{prop}
\begin{proof}
	The inclusion $\Sh^{\blacksquare}(Y_\proet,\ZhatP)\subset \Sh(Y_\proet,\ZhatP)$ preserves colimits by \Cref{thm:solid}, and the same is true over $X$; using that $\Sh^{\blacksquare}(X_\proet,\ZhatP)$ is generated under colimits by the $j_\sharp^\blacksquare\ZhatP$ for $j\colon U\to X$ affine pro-étale, to prove that $f^*$ preserves solid sheaves, it suffices to check that each $f^*(j_\sharp^\blacksquare\ZhatP)$ is solid. Thus the first part of the lemma follows form the second.

	By commutation of the functor $f^*$ with limits (\Cref{lem:pullbackslimits}), étale base change and \Cref{rem:diff-generators} we have $f^*(j_\sharp^\blacksquare\ZhatP)\simeq (j_Y)_\sharp^\blacksquare\ZhatP$ : we have proven the second part.

	Now, as we have just proven that $f^*$ preserves solid sheaves, there is a natural transformation $\mathrm{Ex}^{\blacksquare,*}\colon L^\blacksquare\circ f^*\to f^*\circ L^\blacksquare$. As noted above, the functor $i$ preserves colimits, thus the functor $L^\blacksquare$ also commutes with colimits. This implies that it suffices to check that $\mathrm{Ex}^{\blacksquare,*}$ is an equivalence when evaluated at the $j_\sharp \ZhatP$ for $j\colon U\to X$ affine pro-étale (\Cref{rec:proet-VS-pro-aff-et}). This then follows from the second part already proven, and the first point of \Cref{thm:solid}.
\end{proof}

\begin{cor}
	\label{prop:pullack-detects-solid}
Let $f\colon X\to Y$ be a surjective morphism of qcqs schemes, and let $F$ be a pro-étale sheaf of $\ZhatP$-modules on $X$ such that $f^*F$ is solid, then $F$ is solid.
\end{cor}
\begin{proof}
	By \Cref{lem:pullbackslimits} the map $F\to L^\blacksquare F$ is sent to an isomorphism by $f^*$, thus is an equivalence because $f^*$ is conservative by \Cref{lem:conservativity-of-pullbacks}.
\end{proof}

\begin{rem}\label{rem:quotients_of_Zhat}
	It is clear from the definition of solid sheaves that for a $\ZhatP$-module, being solid as a $\ZhatP$-module is the same as being solid as a $\widehat{\Z}$-module. We can in fact define solid $R$-modules for $R$ any quotient of $\widehat{\Z}$ in the same fashion as above. All of the above results hold in that setting (in particular, they hold for $R=\Z/n\Z$).
\end{rem}
The above remark motivates the following definition:
\begin{defi}
	\label{defi:SolidAnyCoeff}
	Let $\Lambda$ be a solid $\ZhatP$-algebra and let $X$ be a scheme. We define the derived category of solid schemes on $X$ with $\Lambda$ coefficients to be 
	\[\D(X,\Lambda)^{\blacksquare}:=\mathrm{Mod}_\Lambda(\D(X,\ZhatP)^{\blacksquare}).\] 
\end{defi}

\subsection{Derived categories}\label{sec:derivedsolid}
We keep $X$ and $\mathcal{P}$ from the above section. Recall that by 5. of \Cref{thm:solid}. a complex $K$ in $\D(X_\proet,\ZhatP)$ belongs to $\D(X,\ZhatP)^{\blacksquare}$ if and only if for each $j\colon U\to X$ in $\mathrm{Pro}(X_{\et}^{\mathrm{aff},\mathrm{fp}})$, the natural map
\[\underline{\Hom}_{\D(X_\proet,\ZhatP)}(j_\sharp^\blacksquare\ZhatP,K)\to j_*j^*K\] is an equivalence.
\begin{prop}
	\label{prop:derivedSolid}
We have the following:
\begin{enumerate}
    \item The inclusion
\[\D(X,\ZhatP)^{\blacksquare}\subseteq \D(X_\proet,\ZhatP)\]
admits a left adjoint $(-)^\blacksquare$.
    \item $\D(X,\ZhatP)^{\blacksquare}$ identifies with the derived category of solid sheaves of $\ZhatP$-modules and $(-)^\blacksquare$ with the derived functor of the solidification.
\item The kernel of $K\mapsto K^\blacksquare$ is a tensor ideal. In particular, there is a unique symmetric monoidal structure $\otimes^\blacksquare$ on $\D(X,\ZhatP)^{\blacksquare}$ making the functor $(-)^\blacksquare$ symmetric monoidal. It is the left derived functor of the induced symmetric monoidal structure the abelian category of solid sheaves. This symmetric monoidal structure commutes with small colimits in each variable and pullbacks.
\item The above symmetric monoidal structure is closed and its internal $\Hom$ are induced by the ones on $\D(X_\proet,\ZhatP)$.
\item The t-structure on $\D(X,\ZhatP)^{\blacksquare}$ is compatible with colimits, left complete and right complete. Hence the inclusion $\D(X,\ZhatP)^{\blacksquare}\subseteq \D(X_\proet,\ZhatP)$ is compatible with colimits.
\end{enumerate}
\end{prop}
\begin{proof}
	By  \Cref{thm:solid}~5 and \Cref{rem:diff-generators}, we may apply \cite[Lemma 5.9]{condensedpdf} to our situation (with compact generators of $\Sh(X_\proet,\ZhatP)$ the $\ZhatP[W]$ for $W\to X$ a pro-étale affine map with $W$ $w$-contractible) and this proves the first two points. 
	The third point is proven following the same proof as \cite[Theorem 6.2]{condensedpdf}. The fourth point amounts to proving that $\underline{\Hom}$ sends pairs of $\mathrm{FS}$-solid objects to $\mathrm{FS}$-solid objects; this is true by adjunction.
	The compatibility with colimits and right completeness of the last point follow from the second point and the general theory of derived categories from \cite[Remark C.5.4.11]{SAG}, the left completeness follows from the fact that the inclusion into $\D(X_\proet,\ZhatP)$ is t-exact, compatible with limits, and the fact that the latter has a left-complete t-structure by \cite[Proposition 3.3.3, Proposition 4.2.8 and Proposition 3.2.3]{zbMATH06479630}.

\end{proof}

For the rest of this section, we fix a solid $\ZhatP$-algebra $\Lambda$.
\begin{cor}
	\label{prop:derivedSolid_Lambda} 
We have the following:
\begin{enumerate}
        \item The inclusion
\[\D(X,\Lambda)^{\blacksquare}\subseteq \D(X_\proet,\Lambda)\]
admits a left adjoint $(-)^\blacksquare$ and commutes with colimits.
\item The formation of $K\mapsto K^\blacksquare$, for $K$ in $\D(X_\proet,\Lambda)$ commutes
with any base change of qcqs schemes.
\item The kernel of $K\mapsto K^\blacksquare$ is a tensor ideal. In particular, there is a unique symmetric monoidal structure $\otimes^\blacksquare$ on $\D(X,\Lambda)^{\blacksquare}$ making the functor $(-)^\blacksquare$ symmetric monoidal. This symmetric monoidal structure commutes with small colimits in each variable and pullbacks, is closed and its internal $\Hom$ are induced by the ones on $\D(X_\proet,\Lambda)$.
\end{enumerate}
\end{cor}
\begin{proof}
		We begin by noting that by \cite[Theorem 4.8.4.6]{lurieHigherAlgebra2022} we have that \begin{equation}\label{eq:solidLabdaTensor}\D(X,\Lambda)^{\blacksquare} \simeq \D(X,\ZhatP)^{\blacksquare}\otimes_{\D(*_\proet,\ZhatP)}\D(*_\proet,\Lambda).\end{equation} The same is true for $\D(X_\proet,\Lambda)$. Thus we may use \cite[Corollary 1.46]{RamziDualizable} to conclude that the functor $$(-)^\blacksquare\otimes \Lambda\colon \D(X_\proet,\Lambda)\to \D(X,\Lambda)^{\blacksquare}$$ is a localization functor and that there is a canonical inclusion\[\D(X,\Lambda)^{\blacksquare}\subseteq \D(X_\proet,\Lambda).\] The commutation of latter inclusion with colimits follows from the fact in the commutative square of right adjoints \[\begin{tikzcd}
	{\D(X,\Lambda)^{\blacksquare}} & {\D(X_\proet,\Lambda)} \\
	{\D(X,\ZhatP)^{\blacksquare}} & {\D(X_\proet,\ZhatP)}
	\arrow[from=1-1, to=1-2]
	\arrow[from=1-1, to=2-1]
	\arrow[from=1-2, to=2-2]
	\arrow[from=2-1, to=2-2]
\end{tikzcd}\] the vertical functors are conservative and compatible with colimits by \cite[Corollary 4.2.3.5 and Corollary 4.2.3.2]{lurieHigherAlgebra2022}, and the bottom horizontal map is compatible with colimits by \Cref{prop:derivedSolid}. The second and third points are direct applications of formula \ref{eq:solidLabdaTensor}.

\end{proof}
The following lemma is well-known and can be assembled from pieces in Martini-Wolf' work on internal category theory.
 \begin{lem}
	\label{lem:derived-cat-of-ring-topoi}
	Let $(\mathfrak{X},\mathscr{O})$ be a ringed $\infty$-topos. 
		 There exists a functor \[\mathrm{Mod}_{\mathscr{O}}\colon \mathfrak{X}^\op\to \mathrm{CAlg}(\PrL)\] that sends $U\in\mathfrak{X}^\op$ to $\mathrm{Mod}_{\mathscr{O}_{\mid U}}(\Sh(\mathfrak{X}_{/U},\Sp))$ and a map in $\mathfrak{X}^\op$ to the restriction functor.
\end{lem}
\begin{proof}
	First, as a consequence of \cite[Proposition 2.6.3.7]{MWPresentable}, the functor 
	$U\mapsto \mathfrak{X}_{/U}$ can be endowed with the structure of the initial presentable category internal to the topos $\mathfrak{X}$, that it denoted by $\Omega^\otimes$. The objet $\mathscr{O}\in\mathrm{CAlg}(\mathrm{Shv}(\mathfrak{X},\Sp))$ defines (see the paragraph just after \cite[Definition 2.5.2.1]{MWPresentable}) a commutative $\mathfrak{X}$-spectrum (\cite[Definition 2.5.3.13]{MWPresentable}, this is a commutative ring objet in the $\mathfrak{X}$-category $\mathrm{Sp}(\Omega)$). In particular, we can take modules internal to $\mathfrak{X}$ and by \cite[Proposition 2.5.3.8]{MWPresentable}, the $\mathfrak{X}$-category 
	$\mathrm{Mod}_\oscr(\mathrm{Sp}(\Omega))$ is a presentably symmetric monoidal $\mathfrak{X}$-category. Unravelling the definitions, this provides a functor 
	\[\mathfrak{X}^\op \to\mathrm{CAlg}(\PrL)\] sending $U$ to  $\mathrm{Mod}_{\mathscr{O}_{\mid U}}(\Sh(\mathfrak{X}_{/U},\Sp))$, which is what we wanted.
	The preservation of limits can be checked after composition with the forgetful functor to $\PrL$ where we can use \cite[Remark 2.1.0.5]{SAG}.
\end{proof}
\begin{lem}
	\label{lem:some-operations-on-Dsolid}
	Let $f\colon Y\to X$ be a map of qcqs schemes.

	\begin{enumerate}
	\item The restriction of the functor $f^*\colon \D(X_\proet,\Lambda)\to \D(Y_\proet,\Lambda)$ to $\D(X,\Lambda)^{\blacksquare}$ lands in $\D(X,\Lambda)^{\blacksquare}$. When $f$ varies, these functors assemble to give a functor
	\[\D(-,\Lambda)^{\blacksquare}\colon \Sch^\op\to \CAlg(\PrL).\]
	\item The restriction of the functor $f_*\colon \D(Y_\proet,\Lambda)\to \D(X_\proet,\Lambda)$ to $\D(Y,\Lambda)^{\blacksquare}$ lands in $\D(Y,\Lambda)^{\blacksquare}$. In particular, this shows that right adjoint $f_*^\blacksquare$ to the functor $f^*$ of 1. is simply $f_*$.
	\item If $f$ is weakly étale, then there exists a functor $f_\sharp^\blacksquare\colon \D(Y,\Lambda)^{\blacksquare}\to \D(X,\Lambda)^{\blacksquare}$, left adjoint to $f^*$.
	\end{enumerate}
	Moreover, if $F$ is $f^*$ or $f_\sharp$, and $F^\blacksquare$ its solid counterpart, the natural transformation $(-)^\blacksquare\circ F\to F^\blacksquare\circ (-)^\blacksquare$ is an equivalence.
\end{lem}
\begin{proof}
	We begin by proving the first point.
	That $f^*$ restrict to a functor between the solid categories is immediate from \Cref{prop:pullback_of_solid_remains_solid} in the case $\Lambda = \ZhatP$.
	For a general $\Lambda$, note that the commutative square
\[\begin{tikzcd}
	{ \D(Y,\Lambda)^{\blacksquare}} & { \D(Y,\Lambda)^{\blacksquare}} \\
	{ \D(Y,\ZhatP)^{\blacksquare}} & { \D(Y,\ZhatP)^{\blacksquare}}
	\arrow["{f^*}", from=1-1, to=1-2]
	\arrow["{-\otimes_{\ZhatP} \Lambda}", from=2-1, to=1-1]
	\arrow["{f^*}", from=2-1, to=2-2]
	\arrow["{-\otimes_{\ZhatP} \Lambda}"', from=2-2, to=1-2]
\end{tikzcd}\]
	is vertically right adjointable.
	This follows immediately from the formula $\D(X,\Lambda)^{\blacksquare} \simeq \D(X,\ZhatP)^{\blacksquare}\otimes_{\D(*,\ZhatP)^\blacksquare}\D(*,\Lambda)^\blacksquare$ and the bifunctoriality of the relative tensor product, since the forgetful functor $\D(Y,\Lambda)^{\blacksquare} \to \D(Y,\ZhatP)^{\blacksquare}$ is ${\D(*,\ZhatP)^\blacksquare}$-linear and colimit preserving.
	Thus the case for a general $\Lambda$ follows since the forgetful functor is conservative and preserve limits.

	The highly coherent functor $\D(-,\Lambda)^{\blacksquare}$ is constructed as follows: by \Cref{lem:derived-cat-of-ring-topoi}, we have such a functor for $X\mapsto \Sh((\Sch/X)_\proet,\Sp)$. 
	Note that for each scheme $X$, the functor $\Sh(X_\proet,\Sp)\to \Sh((\mathrm{Sch}/X)_\proet,\Sp)$ is fully faithful, compatible with tensor products and pullbacks. Thus by \cite[Proposition 2.5.1.13 and Proposition 2.5.2.8]{MWPresentable}, we have a functor $\D((-)_\proet,\Lambda)$ on schemes, with values in $\mathrm{CAlg}(\PrL)$. Now we may apply \cite[Proposition 2.2.1.9]{lurieHigherAlgebra2022} to obtain both the functor $\D(-,\Lambda)^{\blacksquare}$ valued in $\mathrm{CAlg}(\PrL)$ and the compatibility with the solidification functor.

	Because the functor $f^*$ preserves solid sheaves, the existence of $f_\sharp^{\blacksquare}$ for weakly étale maps, together with its compatibility with solidification is formal.
	We now prove that $f_*$ preserves solid objects.
	For this note that

	We may assume that $\Lambda=\ZhatP$ because $f_*$ is induced by the functor $f_*$ with $\ZhatP$ coefficients.
	Let $K\in \D(Y,\ZhatP)^{\blacksquare}$. By using left-completeness (\Cref{prop:derivedSolid} 5.) and the fact that the inclusion functor of solid complexes in all complexes is t-exact and commutes with limits, we may assume that $K$ is bounded below. Next, by a spectral sequence argument, we may assume that $K$ is concentrated in degree $0$. In that case using that $f_*$ preserves colimits of uniformly bounded below complexes (\cite[Lemma~5.3]{MR4609461}), that the t-structure on solid sheaves is compatible with colimits, and that the abelian category of solid sheaves is closed under colimits in the abelian category of pro-étale sheaves, we may assume that $K$ is compact in $\Sh(X,\ZhatP)^{\blacksquare}$. In that case, $K$ is a cofiltered limit of torsion étale sheaves thus we may assume that $K$ is a constructible torsion étale sheaf. By \cite[Lemma 5.4.3]{zbMATH06479630} we conclude as then $f_*K$ has cohomology sheaves being torsion étale sheaves which are solid by \Cref{torsion are solid}. 
\end{proof}

\begin{lem}
	\label{lem:solid-pullback-commute-with limits}
	Let $ f \colon X \to Y $ be a morphism of schemes.	
	Then the functor $ f^* \colon \D(Y,\Lambda)^{\blacksquare} \to \D(X,\Lambda)^{\blacksquare} $ commutes with all limits.
\end{lem}
\begin{proof}
	By the first point of \Cref{lem:some-operations-on-Dsolid} we have a commutative square
	\[\begin{tikzcd}
		{\D(Y_\proet,\Lambda)} & {\D(X_\proet,\Lambda)} \\
		{\D(Y,\Lambda)^\blacksquare} & {\D(X,\Lambda)^\blacksquare}
		\arrow["{f^*}", from=1-1, to=1-2]
		\arrow[hook', from=2-1, to=1-1]
		\arrow["{f^*}", from=2-1, to=2-2]
		\arrow[hook', from=2-2, to=1-2]
	\end{tikzcd}\]
	where the upper horizontal functor preserves limits by \Cref{lem:pullbackslimits} and the vertical functors are fully faithful and preserve limits as they are right adjoint functors.
	Thus so does the lower horizontal functor.
\end{proof}

\begin{prop}\label{prop:solid_BC}
	Consider a cartesian square of schemes
	\[\begin{tikzcd}
		{X'} & X \\
		{Y'} & Y
		\arrow["g", from=1-1, to=1-2]
		\arrow["q"', from=1-1, to=2-1]
		\arrow["p", from=1-2, to=2-2]
		\arrow["f"', from=2-1, to=2-2]
	\end{tikzcd}.\] 
	The canonical map $f^* p_* \to q_* g^* $ of functors $ \D(X,\Lambda)^{\blacksquare} \to  \D(Y',\Lambda)^{\blacksquare} $ is an equivalence if we are in one of the following cases:
	\begin{enumerate}
		\item The morphism $ p $ is proper.
		\item The morphism $ f $ is smooth and all primes in $\mc{P}$ are invertible on $Y$.
	\end{enumerate}
\end{prop}
\begin{proof}
	Because the forgetful functor $\D(X,\Lambda)^{\blacksquare}\to\D(X,\ZhatP)^{\blacksquare}$ is conservative and commutes with all $f^*$ and $p_*$'s, we may assume that $\Lambda = \ZhatP$. In that case, since all of the involved functors commute with limits by \Cref{lem:solid-pullback-commute-with limits} and the t-structure on solid complexes is left complete by the last point of \Cref{prop:derivedSolid}, we may replace some $ K \in \D(X,\ZhatP)^\blacksquare $ by its truncation $ \tau_{\leq n} K $ to assume that $ K $ is bounded below.
	Since $ p_*, q_* $ commute with filtered colimits of uniformly bounded below complexes (\cite[Lemma~5.3]{MR4609461}), we may reduce the case where $ K $ is bounded by writing $ \colim_n \tau_{\geq n} K \simeq K $.
	If $ K $ is bounded we may do dévissage to assume that $ K $ is concentrated in degree $ 0 $.
	Using commutation with filtered colimits again, \Cref{thm:solid} allows us to reduce to the case where $ K $ is a torsion constructible étale sheaf in which case the claim follows from ordinary proper/smooth base change in étale cohomology which is \cite[Théorème~5.1]{sga4} and \cite[\href{https://stacks.math.columbia.edu/tag/0EYU}{Tag~0EYU}]{stacks-project}.
\end{proof}

\begin{rem}
	As a consequence of \Cref{lem:solid-pullback-commute-with limits}, the functor $f^* \colon \D(Y,\Lambda)^{\blacksquare} \to \D(X,\Lambda)^{\blacksquare}$ has a left adjoint $f_\sharp^\blacksquare$, which we can think of as relative solid $\Lambda$-homology.
	Note that this left adjoint also exists for the non-solid categories by \Cref{lem:pullbackslimits}.
	The non-solid version however is not that well behaved.
	For example, it will not be homotopy variant, even with torsion coefficients prime to the characteristic, see \cite[\S 7.1]{bastietal}.
	Relative solid $\Lambda$-homology however is $\A^1$-invariant by the proof of \Cref{cor:A1invrho}.
	Furthermore \Cref{prop:solid_BC} shows that for smooth $f$, it is compatible with base change.
\end{rem}

For later use we also add the following small observation:

\begin{lem}
	\label{lem:muinfty_invertible}Assume that all the primes in $\mc{P}$ are invertible on $X$, then the sheaf $ \mu_{\infty,\mathcal{P}} \coloneqq\lim_{n\in \Pi\mathcal{P}} \mu_n$ is invertible in $\D(X,\ZhatP)^\blacksquare$ (see \Cref{lem:ZhatconstrIslimTors} for the notation $\Pi\mathcal{P}$).
\end{lem}
\begin{proof}
	Since all primes in $\mathcal{P}$ are invertible on $X$, the map of schemes $\lim_{n\in \Pi\mathcal{P}} \mu_n \to X$ is a pro-\'etale cover.
	Thus we may assume that a primitive $n$-th root of unity exists for every $n$ that is a product of primes in $\mathcal{P}$.
	Picking a compatible system of such roots then defines an isomorphism $\widehat{\Z}_\mathcal{P}\xrightarrow{\cong} \mu_{\infty,\mathcal{P}} $, which shows the claim.
\end{proof}

We end this section with a comparison with usual $\ell$-adic sheaves and complexes of torsion étale sheaves.
\begin{prop}
	\label{prop:torsion etale is solid derived}
	Let $X$ be a qcqs scheme and assume that $\Lambda = \Z/n\Z$. Then the functor 
	\[\D(X_\et,\Z/n\Z)\to \D(X_\proet,\Z/n\Z)\] lands in $\D(X,\Z/n\Z)^{\blacksquare}$ and gives a symmetric monoidal functor 
	\[\D(X_\et,\Z/n\Z)\to \D(X,\Z/n\Z)^{\blacksquare}.\] 
	Moreover, the latter induces a fully faithful functor 
	\[\D(X_\et,\Z/n\Z)^{\mathrm{t}\wedge}\to \D(X,\Z/n\Z)^{\blacksquare},\] where $\D(X_\et,\Z/n\Z)^{\mathrm{t}\wedge}$ is the left completion of $\D(X_\et,\Z/n\Z)$.
\end{prop}
\begin{proof}
	To prove the first two parts of the proposition, as being solid can be checked on cohomology objects, and the functor 
	$\D(X_\et,\Z/n\Z)\to \D(X_\proet,\Z/n\Z)$ is t-exact, it suffices to prove that if $F\in\mathrm{Sh}(X_\et,\Z/n\Z)$, then $\nu^*(F)$ is solid. This is \Cref{torsion are solid}. The last part of the proposition follows from the fact that the t-structure on $\D(X,\Z/n\Z)^{\blacksquare}$ is left-complete, and that $\D(X_\et,\Z/n\Z)\to \D(X_\proet,\Z/n\Z)$ is fully faithful after left-completion by \cite[Proposition 5.3.2]{zbMATH06479630}.
\end{proof}

Recall the following definition:
\begin{defi}[Hemo-Richarz-Scholbach]
	\label{defi:HRS}
	Let $X$ be a qcqs scheme. A complex $K\in\D(X_\proet,\Lambda)$ is called constructible if 
	there exists a finite stratification of $X$ by quasi-compact locally closed subschemes
	such that $K$ is dualizable over each stratum. 
	Denote by $\D_\mathrm{cons}(X,\Lambda)$ the subcategory of such $K$'s.

	Recall also that by \cite[Theorem 7.7]{MR4609461}, there is a canonical equivalence
	\[\D_\mathrm{cons}(X,\Z_\ell)\simeq \D^b_c(X_{\et},\Z_\ell)\] with the classical category of $\ell$-adic constructible complexes and a the same is true with $\Q_\ell$-coefficients. 
\end{defi}
\begin{prop}\label{prop:cons_are_solid}
	 Let $X$ be a qcqs scheme. Let $K\in\D_\mathrm{cons}(X,\ZhatP)$. Then $K$ is solid: we have $K\in\D(X,\ZhatP)^{\blacksquare}$. 
\end{prop}
\begin{proof}
	If $K$ is dualizable, the functor $-\otimes K$ commutes with limits, thus using that $\ZhatP = \lim_n\ZhatP/n\ZhatP$, we see that $K\simeq \lim_n K/n$ and we conclude using \cite[Proposition 7.1]{MR4609461} that implies that each $K/n$ is a bounded complex of torsion étale sheaves.

In the general case, $K$ is dualizable over the strata of some stratification of $X$. Taking the disjoint union of all strata yields a surjective map $g\colon Y\to X$ such that $g^*K$ is dualizable, whence solid. We conclude by \Cref{prop:pullack-detects-solid}.
\end{proof}
\begin{cor}\label{prop:cons_are_solid_rational}
	 Let $X$ be a qcqs scheme. Then every object of $\D^b_c(X_\et,\Q_\ell)\subset \D(X_\proet,\Q_\ell)$ lies in the full subcategory  $\D(X,\Q_\ell)^{\blacksquare}$.
\end{cor}
\begin{proof}
	Given $K\in\D_\mathrm{cons}(X,\Q_\ell)$, there exists $L\in\D_\mathrm{cons}(X,\Z_\ell)$ such that $L\otimes_{\mathbb{Z}_l} \Q_l \simeq K$, by \cite[Corollary 2.4]{zbMATH07751006}. 
	Thus the claim follows from \Cref{prop:cons_are_solid}.
\end{proof}

\subsection{Solid sheaves through condensed category theory}\label{sec:solidcats}

We now provide an alternative way of passing to subcategories of solid objects using the machinery of condensed $\infty$-categories.
Condensed $\infty$-categories can be viewed as $ \infty $-categories internal to the $ \infty $-topos $ \Cond(\Ani) $ of condensed anima.
A general theory of $ \infty $-categories internal to an arbitrary $ \infty $-topos has been developed in \cite{MYoneda,MWColimits,MWPresentable}.
We recall some of the basic ideas here.

\begin{defi}
	\label{part:condensed}
	A \emph{condensed $\infty$-category} $\mathcal{C}$ is a hypersheaf of $\infty$-categories on the site $\ProFin=*_\proet$ of profinite sets $\mathcal{C} \colon \ProFin \to \catinfty$.
	A functor of condensed  $\infty$-categories is simply a natural transformation.
	Similarly a symmetric monoidal condensed  $\infty$-category is simply a hypersheaf on profinite sets with values $ \CAlg(\catinfty) $.

	If $\mathcal{C}$ is a condensed category and $s \colon S \to T$ is a map of profinite sets, we often denote the associated transition functor $\mathcal{C}(T) \to \mathcal{C}(S)$ by $s^*$.

\end{defi}

Here are the main examples that we use in this paper:

\begin{ex}
	\label{ex:condensed-cats}
	\begin{enumerate}
		\item Unstraightening the codomain fibration and restricting along $\ProFin^{\op} \to \Cond(\Ani)^{\op}$ gives a functor \[ \ProFin^{\op} \to \catinfty; \;\;\; S \mapsto \Cond(\Ani)_{/S}.\] We denote this condensed $\infty$-category by $\underline{\Cond}(\Ani)$.
		Its role in condensed category theory is analogous to the role of $\infty$-category of $\Ani$ in usual $\infty$-category theory.
		Taking the cartesian monoidal structure at every $S$, $\underline{\Cond}(\Ani)$ canonically refines to a symmetric monoidal condensed category, see \cite[Remark 2.6.2.7]{MWPresentable}.
		\item For any presentable $\infty$-category $\mathcal{C}$ there is a version of the above example with coefficients in $\ccal$ given by \[ \ProFin^{\op} \to \catinfty; \;\;\; S \mapsto \Cond(S,\mathcal{C}) = \Cond(\Ani)_{/S} \otimes \mathcal{C}, \] where $\otimes$ denotes the Lurie tensor product. We denote this condensed $\infty$-category by $\underline{\Cond}(\mathcal{C})$. If furthermore $\mathcal{C} \in \CAlg(\PrL)$, then $\underline{\Cond}(\mathcal{C})$ canonically refines to a symmetric monoidal condensed $\infty$-category.
		\item Let  $\Lambda$ be a condensed $E_\infty$-ring, i.e. a commutative algebra in $\Cond(\Sp)$. For any profinite set $ S $ we still denote by $\Lambda$ the base change of $\Lambda$ to $\Cond(S,\Sp)$.
			Then there is a symmetric monoidal condensed category of $\Lambda$-modules
			\[
				\underline{\Mod}_\Lambda \colon \ProFin^{\op} \to \CAlg(\catinfty); \;\;\; S \mapsto \Mod_{\Lambda}(\Cond(S,\Sp)),
			\]
			see \Cref{lem:derived-cat-of-ring-topoi}.
		\item If $\Lambda$ is a solid $\widehat{\Z}_\mathcal{P}$-algebra as before, we get also get a symmetric monoidal condensed category of solid $\Lambda$-modules 
				\[
					\underline{\mathrm{Solid}}_\Lambda \colon \ProFin^{\op} \to \CAlg(\catinfty); \;\;\; S \mapsto \D(S,\Lambda)^{\blacksquare}.
				\]
		\item Let $X$ be a qcqs scheme, we have a condensed category \[
					\underline{\Sh}(X_\proet) \colon \ProFin^{\op} \to \CAlg(\catinfty); \;\;\; S \mapsto \Sh((X\times S)_\proet).
				\]
			As before, we also have a spectral version $\underline{\Sh}(X_\proet,\mathrm{Sp})$ and for any condensed ring spectrum $\Lambda$, we have a $\Lambda$-linear version $\underline{\D}(X_\proet,\Lambda)\colon S \mapsto \Mod_{\Lambda}(\Sh((X\times S)_\proet,\mathrm{Sp}))$.
	\end{enumerate}
\end{ex}

It turns out that all the examples of condensed categories listed above enjoy some special features. That is, they are \emph{presentable} in the following sense.

\begin{defi}[{\cite[\S 2.4]{MWPresentable}}]\label{def:cond_presentable}
	A condensed $\infty$-category $\mathcal{C}$ is called presentable if the following two conditions are satisfied:
	\begin{enumerate}
		\item The sheaf $\mathcal{C} \colon \ProFin^\op \to \catinfty$ factors through the (non-full) subcategory $\PrL \subseteq \catinfty$.
		\item For any $s \colon S \to T $, the left adjoint $s_!$ to the functor $s^* \colon \mathcal{C}(T) \to \mathcal{C}(S)$ is such that for any other map $ \alpha \colon T' \to T $ the commutative square
\[\begin{tikzcd}
	{\mathcal{C}(T)} & {\mathcal{C}(S)} \\
	{\mathcal{C}(T')} & {\mathcal{C}(S \times_T T')}
	\arrow["{s^*}", from=1-1, to=1-2]
	\arrow["{\alpha^*}"', from=1-1, to=2-1]
	\arrow["{(\alpha \times_T S)^*}", from=1-2, to=2-2]
	\arrow["{(s \times_T T')^*}"', from=2-1, to=2-2]
\end{tikzcd}\]
		is horizontally left adjointable.
	\end{enumerate}
\end{defi}

\begin{rem}
	It is easy to check that all examples appearing in \Cref{ex:condensed-cats} are indeed presentable.
\end{rem}

The $\infty$-category of condensed $\infty$-categories is cartesian closed \cite[Proposition 3.2.11]{MYoneda}.
For two condensed $\infty$-categories $\mathcal{C},\mathcal{D}$ we denote the internal hom by
\[
	\underline{\Fun}^{\cond}(\mathcal{C},\mathcal{D}) \in \Cond(\catinfty)
\]
and call it the condensed category of functors from $\mathcal{C}$ to $\mathcal{D}$.
Taking global sections we obtain an $\infty$-category
\[
	\Fun^{\cond}(\mathcal{C},\mathcal{D}) = \underline{\Fun}^{\cond}(\mathcal{C},\mathcal{D})(\ast)
\]
that we call the $\infty$-category of condensed functors from $\mathcal{C}$ to $\mathcal{D}$.
In particular $\Cond(\catinfty)$ is naturally enriched over $\catinfty$ and thus defines an $(\infty,2)$-category (see \cite[Remark 3.4.3]{MYoneda}).
We define an \emph{adjunction} between condensed categories to simply be an adjunction in the $(\infty,2)$-category $\Cond(\catinfty)$ \cite[Definition 3.1.1]{MWColimits}.
Adjoint functors between condensed categories also admit the following very explicit description:

\begin{prop}[{\cite[Proposition 3.2.9]{MYoneda}}]\label{prop:cond_left_adj}
	\label{prop:charac_of_left_adjoints}
	Let $f \colon \mathcal{C} \to \mathcal{D}$ be a functor of condensed categories.
	Then $f$ admits (resp. is) a left adjoint if and only if for any $S \in \ProFin$ the functor $f(S)$ admits a left (resp. right) adjoint and for any map $s \colon S \to T$ the square
\[\begin{tikzcd}
	{\mathcal{C}(T)} & {\mathcal{D}(T)} \\
	{\mathcal{C}(S)} & {\mathcal{D}(S)}
	\arrow["{f(T)}", from=1-1, to=1-2]
	\arrow["{s^*}"', from=1-1, to=2-1]
	\arrow["{s^*}", from=1-2, to=2-2]
	\arrow["{f(S)}"', from=2-1, to=2-2]
\end{tikzcd}\]
	is horizontally left (resp. right) adjointable.
\end{prop}

\begin{defi}
	We define the $\infty$-category $\PrL_{\cond}$ of presentable condensed categories to be the subcategory of $\Cond(\catinfty)$ spanned by the presentable condensed categories and left adjoint functors between them.
\end{defi}

Evaluating at the point defines a functor $\Gamma \colon \PrL_{\cond} \to \PrL$. For later use we recall the following:

\begin{lem}
	\label{lem:colimits_in_Prlcond}
	The functor $\Gamma \colon \PrL_{\cond} \to \PrL$ preserves colimits.
\end{lem}
\begin{proof}
	See the proof of \cite[Proposition 2.6.3.2]{MWPresentable}.
\end{proof}

By \cite[\S 2.6]{MWPresentable} there is a symmetric monoidal structure $ - \otimes^{\cond} -$ on $\PrL_{\cond}$ that can be viewed as a condensed variant of the Lurie tensor product of presentable $\infty$-categories.

\begin{defi}
	A presentably symmetric monoidal condensed category is a commutative algebra $\mathcal{C} \in \CAlg(\PrL_{\cond})$.
\end{defi}

This may sound scary at first sight, but it admits the following more explicit description:

\begin{lem}
	A symmetric monoidal condensed category $\mathcal{C}$ is presentably symmetric monoidal if and only if the following three conditions are satisfied:
	\begin{enumerate}
		\item The underlying condensed category of $\mathcal{C}$ is presentable.
		\item For any $ S \in \ProFin$ the tensor product $ -\otimes - \colon \mathcal{C}(S) \times \mathcal{C}(S) \to \mathcal{C}(S)$ preserves colimits in each variable.
		\item For any map $s \colon S \to T$ of profinite sets, the left adjoint $s_! \colon \mathcal{C}(S) \to \mathcal{C}(T)$ of $s^*$ satisfies the projection formula, that is the canonical map
			\[
				s_!(a \otimes s^* b) \to s_!(a) \otimes b
			\]
			is an equivalence for any $a \in \mathcal{C}(S)$ and $b \in \mathcal{C}(T)$.
	\end{enumerate}
\end{lem}
\begin{proof}
	This follows from unwinding the definitions and \Cref{prop:charac_of_left_adjoints}, see e.g. \cite[after Definition 2.15]{arXiv:2303.00736}.
\end{proof}

\begin{rem}
	Using the above Lemma it is easy to check that all examples in \Cref{ex:condensed-cats} are presentably symmetric monoidal.
\end{rem}

\begin{rem}
	\label{rem:2-func-of-tensor}
Note that the symmetric monoidal $\infty$-category $\mathcal{C}$ is cartesian closed and for two presentable condensed categories, the internal hom is given by the condensed subcategory
\[
	\underline{\Fun}^{\cond,\mathrm{L}}(\mathcal{C},\mathcal{D}) \subseteq \underline{\Fun}(\mathcal{C},\mathcal{D}) 
\]
spanned by the left adjoint functors (this essentially follows from the definition of the $-\otimes^{\cond} -$ as the universal bilinear functor, see \cite[\S 2.6]{MWPresentable}).
For a fixed presentable condensed category $\mathcal{T}$, we thus obtain a map
\[
	\underline{\Fun}^{\cond,\mathrm{L}}(\mathcal{C},\mathcal{D}) \to \underline{\Fun}^{\cond,\mathrm{L}}(\mathcal{C} \otimes^{\cond} \mathcal{T},\mathcal{D} \otimes^{\cond} \mathcal{T})
\]
that is natural in $\mathcal{C}$ and $\mathcal{D}$ as the transpose of the natural map
\[
	\underline{\Fun}^{\cond,\mathrm{L}}(\mathcal{C},\mathcal{D}) \otimes^{\cond} \mathcal{C} \otimes^{\cond} \mathcal{T} \xrightarrow{\mathrm{ev} \otimes \id} \mathcal{D} \otimes^{\cond} \mathcal{T}.
\]
Evaluating at the point we thus obtain a natural map of $\infty$-categories
\[\Fun^{\cond,\mathrm{L}}(\mathcal{C},\mathcal{D}) \to \Fun^{\cond,\mathrm{L}}(\mathcal{C} \otimes^{\cond} \mathcal{T},\mathcal{D} \otimes^{\cond} \mathcal{T})\]
and one furthermore checks that after passing to maximal subgroupoids, this functor recovers the action of the functor $ - \otimes^{\cond} \mathcal{T} \colon \PrL_{\cond} \to \PrL_{\cond}$ on mapping spaces.

In more coherent fashion, one can show that $ - \otimes^{\cond} -$ upgrades to an $(\infty,2)$-functor. Since we do not actually need the full coherence in this paper we do not spell this out here.
\end{rem}

\begin{lem}
	\label{lem:tensoring_adjoints}
	Let $f \colon \mathcal{C} \to \mathcal{D}$ be a map in $\PrL_{\cond}$ that admits a further left adjoint $g \colon \mathcal{D} \to \mathcal{C}$.
	Then for any $\mathcal{T} \in \PrL_{\cond}$ the functor $g \otimes^{\cond} \mathcal{T}$ is left adjoint to $f \otimes^{\cond} \mathcal{T}$.
\end{lem}
\begin{proof}
	This is immediate from \Cref{rem:2-func-of-tensor}.
\end{proof}

\begin{defi}
	We call a functor $f \colon \mathcal{C} \to \mathcal{D}$ of condensed categories fully faithful if for every $S \in \ProFin$ the functor $f(S)$ is fully faithful.
\end{defi}
\begin{lem}
	\label{lem:tensoring_adjoints_fully_faithful}
	Let $f \colon \mathcal{C} \to \mathcal{D}$ be a map in $\PrL_{\cond}$.
	Then for any $\mathcal{T} \in \PrL_{\cond}$, if the right adjoint $g$ of $f$ is fully faithful (\textit{i.e.} $f$ is a \emph{condensed localization}), then so is the right adjoint $g \otimes^{\cond} \mathcal{T}$ of $f \otimes^{\cond} \mathcal{T}$.
\end{lem}
\begin{proof}From \cite[Proposition~2.6.2.11]{MWPresentable}, we have identifications:
	\[\ccal \otimes^{\cond} \tcal \simeq \underline{\Fun}^{\cond,\mathrm{R}}(\tcal^\op,\ccal) \hspace{1cm}\text{and}\hspace{1cm}\dcal \otimes^{\cond} \tcal \simeq \underline{\Fun}^{\cond,\mathrm{R}}(\tcal^\op,\dcal)\]
	where $\underline{\Fun}^{\cond,\mathrm{R}}(\acal,\bcal)$ is the condensed subcategory of $\underline{\Fun}^{\cond}(\acal,\bcal)$ spanned by condensed right adjoints for $\acal$ and $\bcal$ two condensed $\infty$-categories.
	The condensed functor $g \otimes^{\cond} \mathcal{T}$ identifies with composition with $g$ through the above equivalences. The result then follows from \cite[Proposition~3.8.4]{MYoneda}.
\end{proof}


\begin{lem}
	\label{lem:f-sharp-is-condensed}
	Let $f\colon Y\to X$ be a map of schemes. Then the functor $f_\sharp \colon \Sh(Y_\proet)\to \Sh(X_\proet)$ naturally enhances to a condensed left adjoint functor
	$f_\sharp \colon \underline{\Sh}(Y_\proet)\to\underline{\Sh}(X_\proet)$. 
\end{lem}
\begin{proof}
	To see that we have a condensed functor, we need to prove that if $g\colon K\to S$ is a map of profinite sets, the square
\[\begin{tikzcd}
	{\Sh((Y\times K)_\proet)} & {\Sh((X\times K)_\proet)} \\
	{\Sh((Y\times S)_\proet)} & {\Sh((X\times S)_\proet)}
	\arrow["{f_\sharp^K}", from=1-1, to=1-2]
	\arrow["{g^*_Y}"', from=1-1, to=2-1]
	\arrow["{g_X^*}", from=1-2, to=2-2]
	\arrow["{f_\sharp^S}"', from=2-1, to=2-2]
\end{tikzcd}\] commutes. Passing to right adjoints and noting that  $g_X$ is integral, this follows from integral base change (\cite[Proposition 6.18]{bastietal}). As the above square is horizontally right adjointable, $f_\sharp$ is indeed a condensed left adjoint. 
\end{proof}

From this point on, we let $\Lambda$ be a solid $\widehat{\Z}_\mathcal{P}$-algebra where $\mc{P}$ is as before.
\begin{defi}
	\label{defi:SolidCat}
    Let $\ccal\in \Mod_{\underline{\Mod}_\Lambda}(\PrL_\mathrm{cond})$, the \emph{solidification} of $\ccal$ is the tensor product \[\ccal^\blacksquare \coloneqq  \ccal\otimes^\mathrm{cond}_{\underline{\Mod}_\Lambda}\underline{\mathrm{Solid}}_\Lambda.\]\end{defi}

We now have a second way of defining a category of solid pro-\'etale sheaves, that is we may consider the underlying category of the solidification $\underline{\D}(X_\proet,\Lambda)^\blacksquare $ of $\underline{\D}(X_\proet,\Lambda) $ in the above sense.
The next proposition shows that this agrees with the category of \Cref{defi:SolidAnyCoeff}.

\begin{prop}\label{LikeFS}
    Let $X$ be a qcqs scheme. Then \Cref{defi:SolidCat} does not create a clash of notations: 
	the $\infty$-category $\D(X,\Lambda)^{\blacksquare}$ of \Cref{defi:SolidAnyCoeff} is naturally equivalent to $\underline{\D}(X_\proet,\Lambda)^\blacksquare(*)$.
\end{prop}
\begin{proof}
	We may assume that $\Lambda = \widehat{\Z}_\mathcal{P}$.
    Note that both categories in discussion are reflective subcategories of $ \D(X_\proet,\widehat{\Z}_\mathcal{P}) $.
	For the second one this follows using \Cref{lem:tensoring_adjoints_fully_faithful} with the observation that $\underline{\D}(X_\proet,\Lambda)^\blacksquare \simeq \underline{\Sh}(X_{\proet}) \otimes^{\cond} \underline{\mathrm{Solid}}_\Lambda$.
	We have two solidification functors
    \[
    	(-)^{\blacksquare} \colon \D(X_\proet,\widehat{\Z}_\mathcal{P}) \to \D(X,\widehat{\Z}_\mathcal{P})^\blacksquare
    \]
    and
    \[
    	(-)^{\blacksquare'} \colon \D(X_\proet,\widehat{\Z}_\mathcal{P}) \to \underline{\D}(X_\proet,\widehat{\Z}_\mathcal{P})^{\blacksquare}(\ast).
    \]
    Our goal is to show that these two reflective subcategories agrees.
	At first suppose that we are given a map of schemes $ f \colon Y \to X $. We claim that the commutative square
	\[\begin{tikzcd}
		{\underline{\D}(X_\proet,\widehat{\Z}_\mathcal{P})} & {\underline{\D}(Y_\proet,\widehat{\Z}_\mathcal{P})} \\
		{\underline{\D}(X_\proet,\widehat{\Z}_\mathcal{P})^\blacksquare} & {\underline{\D}(Y_\proet,\widehat{\Z}_\mathcal{P})^\blacksquare}
		\arrow["{f^*}", from=1-1, to=1-2]
		\arrow["{(-)^{\blacksquare'}}"', from=1-1, to=2-1]
		\arrow["{(-)^{\blacksquare'}}", from=1-2, to=2-2]
		\arrow["{f^*}"', from=2-1, to=2-2]
	\end{tikzcd}\]
	(the fact that it commutes is just bifunctoriality of $ - \otimes^{\cond} - $) is vertically right adjointable.
	This follows because the square that we get after passing to right adjoints vertically is the square that we get from passing to right adjoints (horizontally and vertically) in
	\[\begin{tikzcd}[column sep=large]
		{\underline{\D}(X_\proet,\widehat{\Z}_\mathcal{P})} & {\underline{\D}(Y_\proet,\widehat{\Z}_\mathcal{P})} \\
		{\underline{\D}(X_\proet,\widehat{\Z}_\mathcal{P})^\blacksquare} & {\underline{\D}(Y_\proet,\widehat{\Z}_\mathcal{P})^\blacksquare}
		\arrow["{(-)^{\blacksquare'}}"', from=1-1, to=2-1]
		\arrow["{f_{\sharp}}"', from=1-2, to=1-1]
		\arrow["{(-)^{\blacksquare'}}", from=1-2, to=2-2]
		\arrow["{f_{\sharp} \otimes^{\cond} \underline{\mathrm{Solid}}_{\widehat{\Z}_\mathcal{P}}}", from=2-2, to=2-1]
	\end{tikzcd}\]
	where the bottom line is obtained by tensoring the condensed functor from \Cref{lem:f-sharp-is-condensed} with $\underline{\mathrm{Solid}}_{\widehat{\Z}_\mathcal{P}}$.
	Since this square commutes by bifunctoriality of $ - \otimes^{\cond} - $ we have shown the claim.

	Now pick a surjective morphism $f \colon E \to X $ of schemes where $ E $ is $ 0 $-dimensional and all local rings are separably closed fields using \Cref{lem:strictly_profinite_cover}.
	A complex $K$ in $\D(X_\proet,{\widehat{\Z}_\mathcal{P}})$ is contained in $ \underline{\D}(X_\proet,{\widehat{\Z}_\mathcal{P}})^{\blacksquare}(\ast) $ if and only if $ f^*K$ is contained in $ \underline{\D}(E_\proet,{\widehat{\Z}_\mathcal{P}})^{\blacksquare}(\ast) $: indeed $ f^* $ is conservative by \Cref{lem:conservativity-of-pullbacks} and applying $ f^* $ to the canonical map $ K\to K^{\blacksquare'}$ gives the canonical map $ f^* K \to (f^*K)^{\blacksquare'} $, by the adjointability that we have just shown.
	Since the same is true in $ \D(X,{\widehat{\Z}_\mathcal{P}})^{\blacksquare} $ by \Cref{prop:pullack-detects-solid}, we have reduced the Proposition to $ X =E $. 
	But in this case
	\[
		\D(E_\proet,{\widehat{\Z}_\mathcal{P}})^\blacksquare(\ast) = \underline{\mathrm{Solid}}_{{\widehat{\Z}_\mathcal{P}}}(\pi_0(E)) =  \D(E,{\widehat{\Z}_\mathcal{P}})^{\blacksquare}
	\]
	as full subcategories of $ \D(E_\proet,{\widehat{\Z}_\mathcal{P}}) $.
	Here the first equality follows from \cite[Proposition 2.6.2.11]{MWPresentable}, \cite[Theorem 7.1.1]{MWColimits} and \Cref{lem:strictly_profinite_cover}.
	The second one is true by definition of $ \underline{\mathrm{Solid}}_{{\widehat{\Z}_\mathcal{P}}}(\pi_0(E)) $.
	This completes the proof.
\end{proof}


\section{Pro-étale motives.}\label{sec:proetmotives}
In this section, we construct the category of pro-étale motivic spectra over a base scheme and we prove that they form a coefficient system, affording a full six functors formalism\footnote{The only caveat is that the exceptional functors will only be defined for finitely presented maps and not for arbitrary maps of finite type which is the usual assumption.}. Our construction closely follows classical material from \cite{MR1813224,MR2423375,AyoubPartII,MR3971240,khan,MR3570135}.
We begin in \Cref{sec:weaklysmooth} by defining the desired category of pro-étale motivic spectra based on the notion of a \emph{weakly smooth} scheme, \Cref{def:weaklysmooth}.
In \Cref{sec:functoriality} we study the basic functoriality properties of this construction and especially investigate the push-forward along a closed immersion \Cref{Proposition: i_* pre colimits}.
This is one of the ingredients for the proof of the \emph{localization theorem} in this setting, which we prove in \Cref{sec:localization}.
Given the localization theorem, we can follow the general construction principles of motivic six functor formalisms.
The only difference is that in our setting is the localization theorem only holds for quasi-compact open immersions.
In \Cref{sec:ambidexterity} and \Cref{sec:sixfunctors} we verify that the key steps in the construction of a motivic six functor formalism still go through in that case, if one only aims to to construct the exceptional operations for finitely presented morphisms.
Finally, in \Cref{sec:embeddingetalemotives} we show that under suitable cohomological finiteness assumptions, \'etale motives embed into pro-\'etale motives.

\subsection{Basic definitions and properties}\label{sec:weaklysmooth}
\begin{defi}
	Let $X$ be a scheme.
	We define the \textit{large pro-\'etale site of $S$-schemes} $\Sch_X^{\proet}$ to be the category of all schemes over $X$, equipped with the Grothendieck-topology generated by collections of weakly \'etale morphisms that are coverings in the fpqc topology.
	In particular we may consider the \textit{large pro-\'etale $\infty$-topos} of hypercomplete sheaves of anima on this site.
	We denote this $\infty$-category by $\Sh_{\proet}(X).$
\end{defi}


\begin{defi} \label{def:BigSH} Let $X$ be a scheme.
	\begin{enumerate}
		\item We define $ L_{\A^1}\Sh_{\proet}(X) $ to be the reflective subcategory of $ \Sh_{\proet}(X) $ spanned by $\mb{A}^1$-invariant sheaves. We denote by $L_{\mathbb{A}^1}$ the localization functor.
		\item We define $ L_{\A^1}\Sh_{\proet,\bullet}(X) $ to be the category of pointed objects in $ L_{\A^1}\Sh_{\proet}(X) $.
		 \item We define $L_{\A^1}\Sh_{\proet}^{\mathrm{S}^1}(X)$ to be the stabilization of $ L_{\A^1}\Sh_{\proet,\bullet}(X) $. There is an infinite suspension functor 
		 \[\Sigma^\infty \colon L_{\A^1}\Sh_{\proet,\bullet}(X)\to L_{\A^1}\Sh_{\proet}^{\mathrm{S}^1}(X).\]
	\end{enumerate}
\end{defi}


	By composing left adjoints we get a \emph{motivic localization functor}
	\[
		L_{\mot} \colon \PSh(\Sch_X) \rightarrow L_{\A^1}\Sh_{\proet}(X)
	\]
	which can be described more explicitly following the ideas of Morel and Voevodsky.
	To that end, recall the algebraic $ n$-simplex $$\Delta_X^n\coloneqq \Spec(\Z[t_0,\dots,t_n]/(\sum_{i=0}^n t_i -1))\times_{\Spec(\Z)} X$$
	 over $ X $. When $n$ varies, these assemble to a simplicial $X$-scheme
	\[
		\Delta^\bullet_X \colon \Delta^\op \rightarrow \Sch_X.
	\]
	For any presheaf $ F $ on $ \Sch_X $ we define
	\[
		\Sing(F) = \colim_{n \in \Delta} F(- \times \Delta^n_X).
	\]
	This assembles to a functor $ \Sing \colon \PSh(\Sch_X) \rightarrow \PSh(\Sch_X) $.
	We write $ \PSh^{\mathbb{A}^1}(\Sch_X) $ for the full subcategory spanned by the $ \A^1 $-invariant presheaves.
	Then the same arguments as in \cite[Theorem~4.25]{zbMATH06863810} show:

\begin{lem}
	Let $X$ be a scheme. The functor $ \Sing(-) $ lands in the full subcategory $ \PSh^{\mathbb{A}^1}(\Sch_{X}) $ and defines a left adjoint to the inclusion functor.
\end{lem}

\begin{prop}\label{prop:motiviclocpreprod}
	Let $X$ be a scheme. The motivic localization functor $ L_{\mot} \colon \PSh(\Sch_{X}) \rightarrow L_{\A^1}\Sh_{\proet}(X) $ commutes with finite products.
\end{prop}
\begin{proof}
	We pick a regular cardinal $ \kappa $ such that the inclusion $ \Sh_{\proet}(X) \hookrightarrow \PSh(\Sch_{X}) $ commutes with $ \kappa $-filtered colimits.
	Note that we have an endofunctor
	\[\PSh(\Sch_{X}) \xrightarrow{L_{\mathbb{A}^1} \circ L_{\proet}} \PSh(\Sch_X)\]
	together with a canonical natural transformation $ \id \rightarrow L_{\mathbb{A}^1} \circ L_{\proet} $.
	We may therefore consider the endofunctor $ L_M = \colim_{\tau <\kappa} (L_{\mathbb{A}^1} \circ L_{\proet})^\tau $.
	Then for any presheaf $ F $, the presheaf $ L_M(F) $ is $ \kappa $-filtered colimit of $ \mathbb{A}^1 $-local presheaves and therefore $ \mathbb{A}^1$-local.
	Furthermore the usual cofinality arguments (see the proof of \cite[Lemma 6.3.3.4]{MR2522659}) show that we also have an equivalence
	\[
		L_M \simeq \colim_{\tau < \kappa} (L_{\proet}\circ L_{\mathbb{A}^1})^\tau
	\]
	and thus $ L_M(F) $ is also a pro-\'etale sheaf for any presheaf $F$.
	Thus $ L_M $ defines a functor
	\[
		\PSh(\Sch_X) \rightarrow L_{\A^1}\Sh_{\proet}(X)
	\]
	and it is easy to see that it is a left adjoint to the inclusion and therefore equivalent to $ L_{\mot} $.
	This explicit description proves the claim since sifted colimits commute with finite products.
\end{proof}

	We fix notations for later use:
	\begin{enumerate} \item The cartesian product defines a symmetric monoidal structure on $ L_{\A^1}\Sh_{\proet}(X) $ and by Proposition~\ref{prop:motiviclocpreprod}, the motivic localization functor is symmetric monoidal. 
	Moreover, the Yoneda embedding, composed with the motivic localization functor, induces a functor 
	\[\mathrm{Sch}_X\to L_{\A^1}\Sh_{\proet}(X).\] If $Y\in\mathrm{Sch}_X$, we also denote by $Y$ its image in $L_{\A^1}\Sh_{\proet}(X)$. \item There is also a functor 
	\[(-)_+\colon L_{\A^1}\Sh_{\proet}(X)\to L_{\A^1}\Sh_{\proet,\bullet}(X)\] that freely adds a base point; by \cite[Corollary 5.4]{MR3281141}, there is a unique symmetric monoidal structure on $L_{\A^1}\Sh_{\proet,\bullet}(X)$ such that $(-)_+$ enhances to a symmetric monoidal functor that is the universal functor from $L_{\A^1}\Sh_{\proet}(X)$ to a pointed category. Hence, we obtain a symmetric monoidal functor 
	\[\mathrm{Sch}_X\to L_{\A^1}\Sh_{\proet,\bullet}(X)\] sending a scheme $X$ to $X_+$. Note that by \cite[Proposition 4.10]{MR3281141}, the functor $\Sigma^\infty$ also naturally enhances to a symmetric monoidal functor which is universal among the functors from $L_{\A^1}\Sh_{\proet,\bullet}(X)$ to a stable category.
	\item We also have a symmetric monoidal functor 
	\[\mathrm{Sch}_{X,\bullet}\to L_{\A^1}\Sh_{\proet,\bullet}(X)\] that sends a pointed $X$-scheme $(Y,s)$, where $s\colon X\to Y$ is a section of $Y\to X$, to an object that we also denote by $(Y,s)$. 
	\end{enumerate}

\begin{defi}
	Let $X$ be a scheme and let $p,q\in \N$ be nonnegative integers with $p\geqslant q$. The \emph{motivic sphere} $\mathrm{S}^{p,q}\in L_{\A^1}\Sh_{\proet,\bullet}(X)$ is the 
	objet \[\mathrm{S}^{p,q}\coloneqq  \Sigma^{p-q}(\Gm,1)^{\otimes q}\] obtained by suspending $p-q$ times the $X$-scheme $\Gm\times X$ pointed at $1$. If $u\geqslant v$ we have $\mathrm{S}^{p,q}\otimes \mathrm{S}^{u,v}\simeq \mathrm{S}^{p+u,q+v}$. Note that as in \cite[Lemma 2.15 and Corollary 2.18]{MR1813224} we have that $\mathrm{S}^{2,1}\simeq (\mathbb{P}^1_X,*)$ is the projective space pointed at any point.

	We define the category $L_{\A^1}\Sh^{\mathbb{P}^1}_{\proet}(X)$, to be the stable presentably monoidal category obtained by $\otimes$-inverting the sphere $\mathrm{S}^{2,1}$ in the sense of \cite[Definition~4.8]{MR3281141}. As $\mathrm{S}^{1,0}\simeq \mathrm{S}^1\otimes \un_{L_{\A^1}\Sh_{\proet,\bullet}(X)}$ is inverted, we obtain a stable $\infty$-category. We denote by $\mathbb{S}(q)[p]$ the image of $\mathrm{S}^{p,q}$ in this category.
	We write
	\[
		\mathcal{F}(n) \simeq \mathcal{F} \otimes \mathbb{S}(1)^n
 	\]
 	for any $ \mathcal{F} \in L_{\A^1}\Sh^{\mathbb{P}^1}_{\proet}(X)$ and $ n \in \mathbb{Z} $. Finally we write \[\mathbb{S}[-]\colon \Sch_X\to L_{\mb{A}^1}\Sh_{\proet,\bullet}(X)\to L_{\A^1}\Sh^{\mathbb{P}^1}_{\proet}(X)\] the obvious functor.
\end{defi}


We now define the analogue in our pro-étale setting, of the smooth-étale site.

\begin{defi}\label{def:weaklysmooth}
	We say that a map $f\colon Y\to X$ of schemes is weakly smooth if it can be written as a finite composition of weakly étale maps and smooth maps.
\end{defi}

    \begin{lem}
    \label{lem:locWL}
    Let $f:Y\to X$ be a weakly smooth map. Then Zariski-locally on $Y$ we can write $f$ as a composition $$Y\to \A^n_X\to X$$ with $Y\to \A^n_X$ weakly étale and $\A^n_X\to X$ the canonical projection.
    \end{lem}
    \begin{proof}
        Assume that $f$ is a composition
        $$ Y=X_{N+1}\xrightarrow{f_N} X_N\to \cdots \to X_1 \xrightarrow{f_0} X_0 = X$$ with each $f_i$ either weakly étale or smooth.  We prove the proposition by induction on $N$.
		If $N=0$, we are looking at a single morphism which is either weakly étale, in which case we are done, or smooth in which case it follows from the structure theorem for smooth morphisms (\cite[\href{https://stacks.math.columbia.edu/tag/039Q}{Theorem 039Q}]{stacks-project}). 

		Let $N\geqslant 1$ and assume that any morphism $f$ that can be written as above with $N-1$ maps is Zariski-locally on the source in the form stated in the lemma. Let $f\colon Y\to X$ that can be written as the composition of $N$ maps as above. We can thus write $f = f_0\circ g$ with $g\colon Y\to Z$ a composition of at most $N-1$ maps as above.
        By our induction hypothesis, locally on $Y$ the map $g$ is of the form $Y\to \A^n_Z\to Z$. Up to replacing $Y$ by one of its Zariski opens we can therefore assume that we are in a situation
        $$Y\to \A^n_Z\to Z \xrightarrow{f_0} X$$ with $Y\to \A^n_Z$ weakly étale. There are two cases. If $f_0$ is weakly étale, then have a commutative diagram \[\begin{tikzcd}
            Y & {\A^n_Z} & Z & X \\
            && {\A^n_X}
            \arrow[from=1-1, to=1-2]
            \arrow[from=1-2, to=1-3]
            \arrow[from=1-3, to=1-4]
            \arrow[from=2-3, to=1-4]
            \arrow[from=1-2, to=2-3]
            \arrow["u"', from=1-1, to=2-3]
        \end{tikzcd}\]
        with $u$ weakly étale as the composition of a weakly étale map and the product with $\A^n$ of a weakly étale map and we are done.
        If $f_0$ is smooth, then locally on $Z$ it is of the form $Z\to \A^m_X\to X$ with $Z\to \A^m_X$ étale. Let $V$ be an open subset of $Z$ on which $f_0$ has this property and let $U$ be its inverse image on $Y$. We have a commutative diagram
        \[\begin{tikzcd}
            U & {\A^n_V} & V & X \\
            & {\A^n_{\A^m_X}\simeq \A^{n+m}_X} & {\A^m_X}
            \arrow["{\mathrm{w-\acute{e}tale}}", from=1-1, to=1-2]
            \arrow[from=1-2, to=1-3]
            \arrow[from=1-3, to=1-4]
            \arrow["{\mathrm{\acute{e}tale}}"', from=1-3, to=2-3]
            \arrow[from=2-3, to=1-4]
            \arrow["{\mathrm{\acute{e}tale}}"', from=1-2, to=2-2]
            \arrow[from=2-2, to=2-3]
        \end{tikzcd}\]
        which proves that we can write $U\to X$ as
        $$U\to \A^{n+m}_X\to X$$ with $U\to \A^{n+m}_X$ weakly étale. This finishes the proof.
    \end{proof}

\begin{defi} \label{def:SmallSH} Let $X$ be a scheme.
	\begin{enumerate}
	\item We define the \emph{pro-\'etale unstable (resp. unstable pointed) motivic homotopy category} $ \HH_{\proet}(X) $ (resp. $ \HH_{\proet,\bullet}(X) $) \emph{of $X$} to be the smallest full subcategory of $ L_{\A^1}\Sh_{\proet}(X) $ (resp. $ L_{\A^1}\Sh_{\proet,\bullet}(X) $) that is closed under small colimits and contains all objects of the form $ U $ (resp. $ U_+$) for any weakly smooth $X$-scheme $U$.

	\item The \emph{$\mathrm{S}^1$-stable pro-\'etale motivic homotopy category} of $ X $ is the smallest stable full subcategory $ \SH^{\mathrm{S}^1}_{\proet}(X) $ of $ L_{\A^1}\Sh^{\mathrm{S}^1}_{\proet}(X) $ that is closed under small colimits and contains all objects of the form $ \Sigma^\infty U_+$ for any weakly smooth $X$-scheme $U$.

	\item The \emph{stable pro-\'etale motivic homotopy category} of $ X $ is the smallest stable full subcategory $ \SH_{\proet}(X) $ of $ L_{\A^1}\Sh^{\mathbb{P}^1}_{\proet}(X) $ that is closed under small colimits and contains all objects of the form $ \mathbb{S}[U](n) $ for any weakly smooth $X$-scheme $U$ and $n\in\Z$.
	\end{enumerate}
\end{defi}

\begin{rem}Let $X$ be a scheme. 
	It is easy to see that $ \HH_{\proet,\bullet}(X) $ is canonically equivalent to the category of pointed objects in $ \HH_{\proet}(X) $.
	In the same way, the category $ \SH_{\proet}(X) $ is given by the $ (\mathbb{P}^1,\infty)$-stabilization of  $ \HH_{\proet,\bullet}(X) $. We denote by \[\Sigma^\infty_{\mathbb{P}^1}\colon \HH_{\proet,\bullet}(X)\to\SH_\proet(X)\] the induced functor.
\end{rem}

\begin{rem}
	Let $X$ be a scheme. The pro-étale topology induces a topology on $\WSm_X$. The map $\WSm_X\to \Sch_X$ therefore induces a map \[L_{\A^1}\Sh_\proet(\WSm_X)\to L_{\A^1}\Sh_{\proet}(X)\] from the category of $\A^1$-invariant pro-étale sheaves on $\WSm_X$. The same proof as \cite[Proposition~2]{zbMATH07799808} (replacing the Nisnevich topology with the pro-étale topology) shows that this functor is fully faithful and that its essential image is $\HH_\proet(X)$. As in \cite[Proposition~4]{zbMATH07799808}, $\SH_\proet(X)$ is then equivalent to the $ \mathbb{P}^1$-stabilization of the subcategory of $\A^1$-invariant objects of $\Sh_{\proet,\bullet}(\WSm_X)$. 
\end{rem}

	Note that since the tensor products on the categories defined in \Cref{def:BigSH} commute with colimits in both variables, we get an induced symmetric monoidal structure on the categories defined in \Cref{def:SmallSH}.

	\subsection{Functoriality}\label{sec:functoriality}

Let us now discuss the basic functoriality of all these constructions:
\begin{emptypar}[Basic Functoriality]
	Let $ f \colon X \rightarrow Y $ be any morphism of schemes.
	There is an induced morphism $ f^* \colon \Sh_{\proet}(Y) \rightarrow \Sh_{\proet}(X) $ of $ \infty $-topoi.

These functor give rise to an adjunction
\[
f^*\colon {\SH_\proet(Y)}\leftrightarrows{\SH_\proet(X)}\colon {f_*}
\]
where the left adjoint is canonically symmetric monoidal and this is also true for all other of the above categories.
\end{emptypar}

\begin{emptypar}[Addtional functorialty]\label{empty:additionaladjoints}
	For any morphism of schemes $ f \colon X \rightarrow Y $, the forgetful functor
	\[
		\Sch_{X} \rightarrow \Sch_{Y}
	\]
	gives rise to a further left adjoint
	\[
		f_\sharp \colon \Sh_{\proet}(X) \rightarrow \Sh_{\proet}(Y)
	\] to $f^*$.
	The functor $ f_\sharp $  descends to a functor
	\[
		f_\sharp \colon L_{\A^1}\Sh^{\mathbb{P}^1}_{\proet}(X) \rightarrow  L_{\A^1}\Sh^{\mathbb{P}^1}_{\proet}(Y)
	\]
	that is left adjoint to $ f^* $.
	Furthermore for any $ \mathcal{F} \in {L_{\A^1}\Sh^{\mathbb{P}^1}_{\proet}(X)} $ and $ \mathcal{G} \in  {L_{\A^1}\Sh^{\mathbb{P}^1}_{\proet}(Y)}$ the unit map $ \mathcal{F} \rightarrow f^* f_\sharp \mathcal{F} $ induces a morphism
	\[
		f_\sharp(\mathcal{F} \otimes f^*\mathcal{G}) \rightarrow f_\sharp(\mathcal{F}) \otimes \mathcal{G.}
	\]
	This map is an equivalence because it reduces to representables where it is clear.
	This shows that the functor $ f_\sharp $ is in fact a morphism of $ L_{\A^1}\Sh^{\mathbb{P}^1}_{\proet}(Y) $-modules.
	Furthermore we note that for \emph{any} cartesian square of schemes
	\[\begin{tikzcd}
		{X'} & X \\
		{Y'} & Y
		\arrow["f"', from=1-2, to=2-2]
		\arrow["q", from=1-1, to=1-2]
		\arrow["g", from=2-1, to=2-2]
		\arrow["p"', from=1-1, to=2-1]
		\arrow["\lrcorner"{anchor=center, pos=0.125}, draw=none, from=1-1, to=2-2]
	\end{tikzcd}\]
	the canonical morphism $ p_\sharp q^* \rightarrow g^*  f_\sharp $ is an isomorphism.
	One can again easily check this on representables.
	All of this discussion is also valid for the unstable and pointed unstable categories.
\end{emptypar}
\begin{emptypar}[Additional smooth functoriality]
	Let $f \colon X \to  Y$ be weakly smooth.
	Then the functor $ f^* \colon \SH_{\proet}(Y) \rightarrow \SH_{\proet}(X) $ admits a left adjoint $ f_{\sharp} $ such that
		\begin{enumerate}
			\item The weakly smooth projection formula is satisfied, i.e. the canonical map
			\[
				f_\sharp (\mathcal{F} \otimes f^*\mathcal{G}) \rightarrow (f_\sharp\mathcal{F}) \otimes \mathcal{G}
			\]
			is an equivalence for any $\fcal \in \SH_{\proet}(X)$ and $\gcal \in \SH_{\proet}(Y)$.
			\item For any cartesian square
			\[\begin{tikzcd}
				{T \times_Y U} & U \\
				T & Y
				\arrow["f", from=1-2, to=2-2]
				\arrow["p"', from=1-1, to=2-1]
				\arrow["g"', from=2-1, to=2-2]
				\arrow["\lrcorner"{anchor=center, pos=0.125}, draw=none, from=1-1, to=2-2]
				\arrow["q", from=1-1, to=1-2]
			\end{tikzcd}\]
			of schemes, the canonical natural transformation $ p_\sharp q^* \rightarrow g^* f_\sharp $ is an equivalence.
		\end{enumerate}
		Indeed since weakly smooth maps are stable under composition $f_\sharp$ is the restriction of the functor constructed above.
		This also works for the unstable and pointed unstable versions.
\end{emptypar}

We now show that there is some additional functoriality in the case of a closed immersion.
This is a first step towards the localization theorem in our context.
The first goal of this section is to show the following proposition.

\begin{prop}\label{Proposition: i_* pre colimits}
	Let $i \colon Z \rightarrow X$ be a closed immersion.
	Then the induced functor
	\[
		i_* \colon \HH_{\proet}(Z) \rightarrow \HH_{\proet}(X)
	\]
	preserves weakly contractible colimits (\textit{i.e.} colimits indexed by weakly contractible simplicial sets) and the induced functor
	\[
		i_* \colon \HH_{\proet,\bullet}(Z) \rightarrow \HH_{\proet,\bullet}(X)
	\]
	preserves all small colimits.
\end{prop}


\begin{emptypar}
	For this we recall some terminology first.
	If $ \mathcal{C} $ is any category with an initial object $\emptyset$, we denote by $ \PSh_{\emptyset}(\mathcal{C}) $ the reflective subcategory of $ \PSh(\mathcal{C}) $ spanned by those presheaves $ \mathcal{F} $ such that $ \mathcal{F}(\emptyset) \simeq * $.

	If $ X $ is a scheme we call a morphism $\varphi$ in $ \PSh(\Sch_X) $ a \emph{motivic equivalence} if $L_{\mot}(\varphi)$ is an equivalence.
	Analogously we call a morphism in $ \PSh_\emptyset(\Sch_X) $ a motivic equivalence if it is a motivic equivalence when seen as a map in $ \PSh(\Sch_X) $.
\end{emptypar}

\begin{lem}\label{lem:i_*moteq}
	The functor $i_* \colon \PSh_\emptyset(\Sch_Z) \rightarrow \PSh_\emptyset(\Sch_X)$ preserves motivic equivalences.
\end{lem}
\begin{proof}
	As in \cite[Théorème~4.5.35]{MR2423375} the functor $i_*$ preserves $\A^1$-local equivalences, so it suffices to see that it preserves pro-\'etale local equivalences.
	Let $f$ be a pro-\'etale local equivalence in $ \PSh_\emptyset(\Sch_Z)$, i.e. a morphism that becomes an equivalence after pro-\'etale hypersheafification.
	Note that the inclusion $ \PSh_\emptyset(\Sch_Z) \to  \PSh(\Sch_Z)$ commutes with $\tau_{\leq n}$, i.e. truncations in $\PSh_\emptyset(\Sch_Z)$ are computed pointwise.
	It follows that $i_* \colon \PSh_\emptyset(\Sch_Z) \rightarrow \PSh_\emptyset(\Sch_X)$ also commutes with truncations.
	To see that $i_* f$ is a pro-\'etale local equivalence it suffices, that $ \tau_{\leq n} i_*f $ becomes an equivalence after pro-\'etale sheafification.
	By the preceding discussion, we may thus assume that $f$ is a map of $n$-truncated presheaves.
	In this case, there is no difference between pro-\'etale sheafification and hypersheafification, thus it suffices to see that $i_*$ preserves maps that become equivalences after pro-\'etale sheafification.
	
	For this the usual argument applies. More precisely, the claim follows as in \cite[Lemme 4.5.33]{MR2423375} if for any $T \in \Sch_X$ and any pro-\'etale covering $\ucal = \{U_i \to T_Z\}$ of the base change $T_Z$ of $T$ to $Z$ there is a covering $ \mc{V} = \{V_j \to T \} $ such that the pullback of $\mc{V}$ to $Z$ refines $ \ucal $.
	For this we may assume that $T$ is affine and all $U_i$ are pro-\'etale affines over $T_Z$ by \cite[Theorem~2.3.4]{zbMATH06479630}.
	Then we may simply take the covering $\{\tilde{U_i} \to T, T_{X\setminus Z } \to T\}$ where we use the notations of \Cref{notation-tilde}.
\end{proof}

\begin{proof}[Proof of \Cref{Proposition: i_* pre colimits}]
    The first statement is a formal consequence of the above Lemma, see \cite[\S~7.2]{khan}. The second statement is now formal, see \cite[Proposition~5.1]{MR3570135} for more details.
\end{proof}
We also need the following lemma.
\begin{lem}\label{lem:essimagei_*}
    Let $i \colon Z \to X$ be a closed immersion.
    Then the image of
    \[
    i^* \colon \SH_\proet(X) \to \SH_\proet(Z)
    \]
    generates $\SH_\proet(Z)$ under colimits.
    The analogous claim holds for $ \HH_{\proet}$ and $ \HH_{\proet,\bullet}$.
\end{lem}
\begin{proof}
    We can assume $X$ to be affine. Using \Cref{lem:locWL}, 
	and the fact that  $i^*$ commutes with $ \mathbb{P}^1$-shifts it suffices to show that the $\mathbb{S}[Y]$ belong to its image, where $ f \colon Y \to Z$ factors as
    \[
        Y \xrightarrow{g} \mathbb{A}^n_Z \to Z
    \]
    with $g$ pro-\'etale and $Y$ affine.
    But the henselization of $ g $ along the closed immersion $ \mathbb{A}^n_Z  \to \mathbb{A}^n_X  $ restricts to $Y$ over $Z$ and the claim follows.
    The proofs for $ \HH_{\proet}$ and $ \HH_{\proet,\bullet}$ are the same.
\end{proof}

\subsection{Localization}\label{sec:localization}

Let $X$ be a qcqs scheme. The goal of this subsection is to prove the following theorem in our setting:

\begin{thm}[Localization]\label{thm:loc}
	Let $i \colon Z \rightarrow X$ be a closed immersion and let $j \colon U \rightarrow X$ be the complementary open immersion. Assume that $U$ is quasi-compact and quasi-separated.
	For any $\mathcal{F} \in \HH_{\proet}(X)$ the canonical square
	\[
	\begin{tikzcd}
j_\sharp j^* \mathcal{F} \arrow[r] \arrow[d] & \mathcal{F} \arrow[d] \\
U \arrow[r]                                  & i_*i^*\mathcal{F}
\end{tikzcd}
	\]
	is cocartesian.
\end{thm}

The upshot of our proof is that the localization theorem for the small pro-\'etale site allows us to reduce to the original proof due to Morel and Voevodsky \cite[\S 3 Theorem 2.21]{MR1813224}.

\begin{rem}
	
	As we assume that the scheme $X$ is qcqs, by \cite[\href{https://stacks.math.columbia.edu/tag/01PH}{Lemma 01PH} and \href{https://stacks.math.columbia.edu/tag/01TV}{Lemma 01TV}]{stacks-project} the open subset $U$ is qcqs if and only if there is a closed subscheme $Z'$ of $X$ with  same underlying topological subset as $Z$, such that $Z'\to X$ is finitely presented.
\end{rem}
\begin{rem}
	The assumptions that the open $U$ is quasi-compact is necessary for \Cref{thm:loc} to hold.
	In fact for non-quasi-compact opens the localization theorem already fails on the level of sheaves on the small pro-\'etale site, see \Cref{rem:counterex_to_loc} below.
\end{rem}
 We begin with a similar statement for the small pro-étale site. 
\begin{prop}
	\label{prop:loc-proet-small-site}
	The functors $i^* \colon \Sh(X_{\proet}) \rightarrow \Sh(Z_{\proet})$ and $j^* \colon \Sh(X_{\proet}) \rightarrow \Sh(U_{\proet})$ are jointly conservative.
\end{prop}
\begin{proof}
	Using \Cref{lem:strictly_profinite_cover} and \Cref{lem:conservativity-of-pullbacks}, we may reduce to the case where $X$ is reduced, $0$-dimensional, affine and all of its local rings are separably closed fields.
	In this case, $i^*$ and $j^*$ identify with the pullback functors $\Cond(\Ani)_{/\pi_0(X)} \to \Cond(\Ani)_{/\pi_0(Z)}$ and $\Cond(\Ani)_{/\pi_0(X)} \to \Cond(\Ani)_{/\pi_0(U)}$.
	Since a quasi-compact open of a profinite set is clopen, we have an isomorphism of profinite sets $\pi_0(X) \simeq \pi_0(Z) \amalg \pi_0(U)$.
	Thus the two above functors are clearly jointly conservative, as desired.
\end{proof}
\begin{rem}
	\label{rem:counterex_to_loc}
	The conclusion of \Cref{prop:loc-proet-small-site} does not hold if the open subset $U$ is not quasi-compact.
	Indeed, consider the profinite set $\mathbb{N} \cup \{ \infty \}$ and the open closed decomposition $ \mathbb{N} \hookrightarrow \mathbb{N} \cup \{ \infty \} \hookleftarrow \{\infty\}$.
	Then in $\Cond(\Ani)_{/ \mathbb{N} \cup \{ \infty \}}$ we have a canonical map
	\[
		\mathbb{N} \amalg \{ \infty \} \to \mathbb{N} \cup \{ \infty \}.
	\]
	This map is clearly not an isomorphism (the domain is equivalent to $\mathbb{N} \cup \{ \infty \}$ equipped with the discrete topology).
	However it also clearly becomes an isomorphism after restricting to $\mathbb{N} $ and $ \{\infty\}$.
\end{rem}
\begin{emptypar}\label{par:i**computation}

	Observe that the diagram
	\[
	\begin{tikzcd}
\PSh_\emptyset(X_{\proet}^{\aff}) \arrow[r,"i^*"] \arrow[d,"L_\proet",swap] & \PSh_\emptyset(Z_{\proet}^{\aff}) \arrow[r,"i_*"] \arrow[d, "L_{\proet}",swap] & \PSh_\emptyset(X_{\proet}^{\aff}) \arrow[d, "L_{\proet}"] \\
\Sh(X_{\proet}) \arrow[r,"i^*"]&\Sh(Z_{\proet}) \arrow[r, "i_*"]                 & \Sh(X_{\proet})
\end{tikzcd}
	\]
	commutes: the left-hand square commutes by definition; in the right-hand square, all functors commute with weakly contractible colimits and $\PSh_\emptyset(Z_{\proet}^{\aff})$ is generated by representables under weakly contractible colimits; as the case of representables follows from the fact that the pro-étale topology is sub-canonical \cite[\href{https://stacks.math.columbia.edu/tag/098Z}{Tag~098Z}]{stacks-project}, the right-hand square indeed commutes. 
	
	We now use notations introduced in \Cref{notation-tilde}. We have a left adjoint $\widetilde{(-)}$ to the pullback $X^{\aff}_{\proet} \to Z^{\aff}_{\proet}$, so that at the level of presheaves, $i^*$ is given by 
	\[P \mapsto (U\mapsto P(\widetilde{U})).\]
	Since $\tilde{\emptyset} = \emptyset$, this preserves presheaves that send $\emptyset$ to $*$. Let $\fcal$ be in $\Sh(X_\proet)$; using the above commutative diagram, the sheaf $i_*i^* \mathcal{F}$ is therefore given by the sheafification of the assertion
	\[
		V \mapsto \fcal(\widetilde{\restr V Z}).
	\]
\end{emptypar}

Armed with this we can now roughly proceed as in the proof of the localization theorem for ordinary motivic spaces. If $Y\in\mathrm{Sch}_X$ is a $X$-scheme, we also denote by $Y$ its image in $\Sh_\proet(X)$.

\begin{defi}
	For any scheme $Y$ over $X$ and any section $t \colon Z \rightarrow Y_Z=Y\times_X Z$ we define a presheaf of sets $\Phi_X(Y,t)$ on the category of schemes over $X$ by setting
	\[
  \Phi_X(Y,t)(Y')=
  \begin{cases}
    \Hom_X(Y',Y) \times_{\Hom_Z(Y'_Z,Y_Z)} *  & \text{if } Y_Z \neq \emptyset\\
    * & \text{else}
  \end{cases}
\]
where the map $* \to \Hom_Z(Y'_Z,Y_Z)$ is given by the map $Y'_Z\to Z \xrightarrow{t} Y_Z$. In other words,
\[\Phi_X(Y,t)=(Y \bigsqcup_{Y_U}U) \times_{i_*Y_Z}X\] with the map $X \to i_* Y_Z$ given by $t$ and adjunction.
\end{defi}

\begin{lem}
	\label{Lemma:Phiinvariantunderproetale}
	Let $f \colon V \rightarrow Y$ be a pro-\'etale morphism of affine schemes and let $t' \colon Z \rightarrow V_Z$ and $t \colon Z \rightarrow Y_Z$ be sections over $Z$ such that $ f_Z \circ t' = t $.
	Then the induced map
		\[
			\psi \colon \Phi_X(V,t') \rightarrow \Phi_X(Y,t)
		\]
	becomes an isomorphism after applying the sheafification functor
	\[
		L_{\proet} \colon \PSh(\Sch_X) \rightarrow \Sh_{\proet}(X).
	\]
\end{lem}
\begin{proof}
	Let $T$ be an affine scheme over $X$. As pro-étale sheafification commutes with restriction to small sites by \cite[Proposition~7.1]{MR4296353}, it suffices to prove that $L_\proet\psi_{\mid T^\mathrm{aff}_{\proet}}$ is an equivalence. By \Cref{prop:loc-proet-small-site} it suffices to prove that $L_\proet\psi_{\mid (T_U)^\mathrm{aff}_{\proet}}$ and $L_\proet\psi_{\mid (T_Z)^\mathrm{aff}_\proet}$ are equivalences.
	As for any $Y\in (T_U)^\mathrm{aff}_{\proet}$, we have $Y_Z = \emptyset$, it is clear that $\psi_{\mid (T_U)^\mathrm{aff}_\proet}$ is an equivalence. We remark that if $i_T\colon T_Z\to T$ is the closed immersion obtained by pullback of $i$ to $T$, the functor $(i_T)_*$ is conservative because it is fully faithful as its left adjoint $i_T^*$ has a further left adjoint given by henselization (see \Cref{notation-tilde}) and the latter is fully faithful.
	Thus it suffices to prove that 
	$$(i_T)_*L_{\proet}\psi_{\mid (T_Z)_{\proet}} = (i_T)_*(i_T)^*\psi_{\mid T_\proet}$$ is an equivalence.
	By \Cref{par:i**computation}, we have that the map of sheaves $(i_T)_*(i_T)^*\psi_{\mid T_\proet}$ is obtained as the sheafification of 
	\[W\in T_\proet^\mathrm{aff}\mapsto \psi(\widetilde{W_{\mid T_Z}})\] so that in particular, it suffices to show (up to replacing $T$ by an affine pro-étale scheme over $T$) that $\psi(\widetilde{T_Z})$ is an equivalence.

	To see this we check that the fiber of $ \psi(\widetilde{T_Z}) $ over any morphism $ q \colon \widetilde{T_Z} \rightarrow Y $ over $ X $ that gives rise to an element in $ 	\Phi_X(Y,t)(\widetilde{T_Z}) $ is a point.
	The latter fiber is isomorphic to
	\[
		\Phi_{\widetilde{T_Z}}(V \times_Y \widetilde{T_Z}, t \times_Z T_Z) (\widetilde{T_Z})
	\]
	so it suffices to see that this set is a point.
	But an element in there is given by a map of schemes $ s \colon \widetilde{T_Z} \rightarrow \widetilde{T_Z} \times_X V $ that makes the diagram
	\[\begin{tikzcd}
		& {T_Z \times_X V} \\
		{T_Z} & {T_Z} & {\widetilde{T_Z} \times_XV} \\
		& {\widetilde{T_Z}} & {\widetilde{T_Z} } \\
		\arrow["\id", from=3-2, to=3-3]
		\arrow[from=2-3, to=3-3]
		\arrow[hook, from=1-2, to=2-3]
		\arrow[from=2-1, to=2-2]
		\arrow[from=1-2, to=2-2]
		\arrow[hook, from=2-1, to=3-2]
		\arrow["{\,}"{description}, hook, from=2-2, to=3-3]
		\arrow["s"{pos=0.3}, from=3-2, to=2-3]
		\arrow["{t \times_Z T_Z}"{description}, from=2-1, to=1-2]
	\end{tikzcd}\]
	commute.
	It follows that there is a unique such map $ s $ by the universal property of the henselization.
	This proves the claim.
\end{proof}

Now we are able to prove the localization theorem in the pro-\'etale setting:

\begin{proof}(of \Cref{thm:loc})
	We may reduce to the case where $X$ is affine.
	Since all functors in discussion commute with sifted colimits, we may reduce to the case that $\mathcal{F}$ is given by $L_{\mot}(Y)$ for some $Y$ that is affine and pro-\'etale over $\A^n_X$ for some $n$ by \Cref{lem:locWL} and \cite[Theorem~2.3.4]{zbMATH06479630}.
	Furthermore all functors in discussion preserve motivic equivalences (the only non-trivial case is $i_*$ but this is \Cref{lem:i_*moteq}).
	This means that it suffices to see that the induced map
	\[
		Y \amalg_{Y_U} U \rightarrow i_* Y_Z
	\]
	is a motivic equivalence in $\PSh(\Sch_X)$.
	For this it suffices to see that for any affine $X$-scheme $f\colon T\to X$ with a morphism $T \rightarrow i_* Y_Z$ corresponding to a morphism $t_0 \colon T_Z \to Y_Z $, the induced morphism
	\[
		(Y \amalg_{Y_U} U \rightarrow i_* Y_Z) \times_{i_* Y_Z} T
	\]
	is a motivic equivalence.

	But using the projection formula in $\PSh(\Sch_T)$ (see \cite[Proposition 6.3.5.11]{MR2522659}) this map is given by applying the functor $f_\sharp \colon \PSh(\Sch_T) \rightarrow \PSh(\Sch_X)$ to the morphism
	\[
		\Phi_T(Y_T,t) \rightarrow * =T.
	\]
	with $t\colon T_Z \to Y_T \times_{T} T_Z=Y_Z \times_X T$ induced by $t_0$ on the first factor and the inclusion $T_Z\to T$ on the second factor. 
	Hence, we have reduced to showing that the presheaf $\Phi_T(Y_T,t)$ is motivically contractible.
	Recall that since $Y$ was affine pro-\'etale over $\A^n_X$, it follows that the structure map $Y_T \rightarrow T$ factors as
	\[
		Y_T \xrightarrow{f_T} \A^n_T \rightarrow T
	\]
	where all schemes are affine and $f_T$ is pro-\'etale.
	By Lemma~\ref{Lemma:Phiinvariantunderproetale} we may reduce to showing that $\Phi_T(\A^n_T,(f_T \times_T T_Z)\circ t)$ is motivically contractible in $\PSh(\Sch_T)$, which is well-known as now $\A^n_T \rightarrow T$ is smooth (see the proof of \cite[Theorem~4.18]{MR3570135}).
\end{proof}

This also implies the localization theorem for $ \SH_{\proet}(X) $:

\begin{cor}[Pointed localization]
\label{cor:pointedloc}
    Let $i\colon Z\to X$ be a closed immersion and let $j\colon U\to X$ be the open complement. Assume that $j$ is a qcqs immersion. For every $M\in \HH_{\proet,\bullet}(X)$ The sequence
    \[
    j_\sharp j^*M\to M\to i_*i^* M
    \]
    is a cofiber sequence.
	In particular $i^*$ and $j^*$ are jointly conservative and $i_*$ is fully faithful.
\end{cor}
\begin{proof}
    The first statement is a formal consequence of \Cref{thm:loc}, see \cite[Proposition~5.2]{MR3570135}.
	That $i^*$ and $j^*$ are jointly conservative is an obvious consequence.
	Plugging $i_* (-)$ into the above cofiber sequence and using $j^* i_* \simeq \ast$ we see that $i_* \to i_* i^* i_*$ is an equivalence.
	Thus the claim follows since $i_*$ is conservative by \Cref{lem:essimagei_*}.
\end{proof}
\begin{cor}[Closed projection formula]
	\label{cor:closed-proj}
    Let $i\colon Z\to X$ be a closed immersion. Assume that its complement is a qcqs open immersion.
    Let $M\in \HH_{\proet,\bullet}(X)$ and let $N\in \HH_{\proet,\bullet}(Z)$. The canonical map
    \[
    M\otimes i_* N\to i_*(i^*M\otimes N)
    \]
is an equivalence.
\end{cor}
\begin{proof}
    This follows immediately from \Cref{cor:pointedloc} using that $j^* i_* \simeq \ast$.
\end{proof}

\begin{emptypar}
    Let $i\colon Z\to X$ be a closed immersion.
    Then canonical functor
    \[\HH_{\proet,\bullet}(X)\otimes_{\HH_{\proet,\bullet}(Z)}\SH_\proet(Z)\to\SH_\proet(X)\] is an equivalence by \cite[Lemma 1.3.11]{arXiv:2204.03434}. 
\end{emptypar}

\begin{cor}
    \label{cor:i_*_stable_pre_colimits}
    Let $i\colon Z\to X$ be a closed immersion. Assume that its complement is a qcqs open immersion.
    Then the functor
    \[i_*\colon \SH_\proet(Z)\to\SH_\proet(X)\] commutes with colimits and the closed projection formula holds in $\SH_\proet$.
\end{cor}
\begin{proof}
    The closed projection formula for $\HH_{\proet,\bullet}$ (\Cref{cor:closed-proj}) together with \Cref{Proposition: i_* pre colimits} implies that the functor $i_*\colon \HH_{\proet,\bullet}(Z)\to\HH_{\proet,\bullet}(X)$ is a map of $\HH_{\proet,\bullet}(X)$-modules in $\PrL$ (see \cite[Remark~7.3.2.9]{lurieHigherAlgebra2022}).
    Using \cite[Section~4.4]{MR3607274}, as the adjunction $(i^*,i_*)$ is internal to the $(\infty,2)$-category denoted by $\mathbf{Pr}^\mathrm{L}(\HH_{\proet,\bullet}(X))$, whose objects are the presentable $\HH_{\proet,\bullet}(X)$-modules, its image by the $(\infty,2)$-functor
    \[-\otimes_{\HH_{\proet,\bullet}(Z)}\SH_\proet(Z)\colon \mathbf{Pr}^\mathrm{L}(\HH_{\proet,\bullet}(X))\to \mathbf{Pr}^\mathrm{L}(\SH_\proet(X))\] is again an internal adjunction.
    In particular, the right adjoint $i_*$ of $i^*\colon \SH_\proet(X)\to\SH_\proet(Z)$ preserves colimits and the projection formula holds.
\end{proof}

\begin{cor}\label{cor:loc}
    Let $i\colon Z\to X$ be a closed immersion and let $j\colon U\to X$ be the open complement. Assume that $j$ is a qcqs immersion.
    For every $M\in \SH_{\proet}(X)$ the sequence
    \[
    j_\sharp j^*M\to M\to i_*i^* M
    \]
    is a cofiber sequence. In particular the pair $(i^*,j^*)$ is conservative.
\end{cor}
\begin{proof}
    By \Cref{cor:i_*_stable_pre_colimits} all functors in question commute with colimits and $\mathbb{P}^1$-stabilization.
    Since $\Sigma^\infty_{\mathbb{P}^1}$ commutes with colimits and its image generates $\SH_{\proet}(X)$ under colimits and $ \mathbb{P}^1$-shifts, the claim follows from pointed localization (\Cref{cor:pointedloc}).
\end{proof}

\begin{cor}\label{cor:i_*_is_ff}
    Let $i \colon Z \to X$ be a closed immersion.
    Assume that its complement is a qcqs open immersion. Then the functor
    \[i_*\colon \SH_\proet(Z)\to\SH_\proet(X)\]
    is fully faithful. 
    The same holds for $ \HH_{\proet}(Z) $ and $ \HH_{\proet}(X) $.
\end{cor}
\begin{proof}
    The proof is the same as in \Cref{cor:pointedloc}.
\end{proof}

\subsection{Ambidexterity}\label{sec:ambidexterity}
In this section, we verify that no step of the proof of ambidexterity for smooth projective finitely presented morphism of qcqs schemes use a localization triangle with a non-qcqs open immersion. We follow \cite[Section 5]{MR3570135}.

\begin{defi}
	\label{def:unstableqcqssys}
	Let $\ccal\colon \mathrm{Sch}^\op\to \mathrm{CAlg}(\PrL_\bullet)$ be a functor with values in pointed symmetric monoidal presentable $\infty$-categories. For $f$ a map of schemes, we denote by $f^*$ the functor $\ccal(f)$ and by $f_*$ its right adjoint.
	We say that $\ccal$ is an \emph{unstable qcqs motivic coefficient system} if the following conditions hold:
	\begin{enumerate}
		\item For every smooth map $f$, the functor $f^*$ has a left adjoint $f_\sharp$ satisfying smooth base change and the smooth projection formula.
		\item \emph{$\A^1$-invariance}: the functor $\pi^*$ is fully faithful if $\pi\colon\A^1_X\to X$ is the canonical projection with $X$ any qcqs scheme.
		\item \emph{Nisnevich descent}: the functor $\ccal$ is a Nisnevich sheaf.
		\item \emph{Qcqs localization}: the analogue of \Cref{cor:pointedloc} holds in $\ccal$.
	\end{enumerate}
\end{defi}

\begin{thm}[Drew-Gallauer]
	Let $\mathrm{H}_\bullet(-)$ denote the unstable qcqs motivic coefficient system constructed by Morel and Voevodsky (it is the Nisnevich version, see \cite{MR1813224}). There exists a natural transformation 
	\[\rho_\ccal\colon\mathrm{H}_\bullet(-)\to \ccal(-)\] that also commutes with $f_\sharp$ for a smooth morphism $f$.
\end{thm}
\begin{proof}
	Indeed, this is a direct application of \cite[Proposition 6.8 and Theorem 7.3]{MR4560376} (see Remark 7.11 of \emph{loc. cit.}).
\end{proof}

\begin{lem}
	\label{lem:univ-rho-commutes-i_*}
	Let $i\colon Z\to X$ be a closed immersion with qcqs open complement and with $X$ qcqs. Then the natural transformation
	\[ \rho_\ccal\circ i_* \to i_*\circ\rho_\ccal\] is an equivalence.
\end{lem}
\begin{proof}
	By localization, this follows directly from the commutation of $\rho_\ccal$ with $j_\sharp$ with $j\colon U\to X$ the open complement.
\end{proof}

\begin{emptypar}\label{def:Sigma^M} Let $S$ be a qcqs scheme and let $\mathcal{M}$ be a locally free $\ocal_S$-module. We denote by 
$p\colon \mathbb{V}(\mathcal{M})\to S$ the associated vector bundle, with zero section $s\colon S\to \mathbb{V}(\mathcal{M})$.
Recall that one denotes by 
\[\Sigma^\mcal \colon \ccal(S) \to \ccal(S) \] 
the functor $p_\sharp s_*$. We also denote by $\Sr^\mcal\coloneq \Sigma^\mcal\un_S\in \ccal(S)$. Note that the morphism $\mathbb{V}(\mcal)\to S$ is affine thus $\mathbb{V}(\mcal)$ is qcqs, and as the zero section is given by a finite type ideal (namely, $\mcal^\vee$), the open immersion $\mathbb{V}(\mcal)\setminus S\to\mathbb{V}(\mcal)$ is a qcqs open immersion. By the closed projection formula (the analog of \Cref{cor:closed-proj} which is follows formally from qcqs localization and the smooth projection formula), we then have that \begin{equation}\label{SigmaM}\Sigma^\mcal\simeq \mathrm{S}^\mcal\otimes(-).\end{equation} Moreover, by \Cref{lem:univ-rho-commutes-i_*} we have $\rho_\ccal\circ\Sigma^\mcal\simeq\Sigma^\mcal\circ\rho_\ccal$. 
	If $f\colon S\to T$ is smooth, then we denote by $$\mathrm{Th}_S(\mcal)\coloneqq f_\sharp(\mathrm{S}^\mcal)\in\ccal(T)$$ the Thom space of $\mcal$. Again by \Cref{lem:univ-rho-commutes-i_*}, we have that $\rho_\ccal(\mathrm{Th}_S(\mcal))\simeq \mathrm{Th}_S(\mcal)$. 
\end{emptypar}
\begin{thm}[Purity]
	\label{purity}
	Let \[\begin{tikzcd}
	Z & X \\
	& S
	\arrow["s", from=1-1, to=1-2]
	\arrow["q"', from=1-1, to=2-2]
	\arrow["p", from=1-2, to=2-2]
\end{tikzcd}\] be a commutative diagram of qcqs schemes with $p$ and $q$ smooth and $s$ a finitely presented closed immersion. Let $\ncal_s$ be the normal bundle of $s$. Then there is a canonical isomorphism 
\[p_\sharp s_*\simeq q_\sharp \Sigma^{\ncal_s}\] of functors $\ccal(Z)\to \ccal(S)$. 
\end{thm}
\begin{proof}
	The morphism is obtained as a span given by deformation to the normal cone: there is a smooth morphism $\hat{p}\colon D_Z X\to \A^1_S$ with $D_Z X$ the deformation space and a finitely presented closed immersion $\hat{s}\colon\A^1_Z\to D_Z X$ obtained by pulling back the finitely presented closed immersion $Z\to X$ along the structure map $D_Z X\to X$, such that the fiber at $1$ is the triangle of the theorem and the fiber at $0$ is the triangle \[\begin{tikzcd}
	Z & {\mathbb{V}(\ncal_s)} \\
	& S
	\arrow["{s_0}", from=1-1, to=1-2]
	\arrow[from=1-1, to=2-2]
	\arrow["{p_0}", from=1-2, to=2-2]
\end{tikzcd}\] with $s_0$ the zero section of the normal bundle and $p_0$ the canonical map. This induces a span 
$$p_\sharp s_* \to  r_\sharp \hat{p}_\sharp\hat{s}_* r^*\leftarrow q_\sharp \Sigma^{\ncal_s}$$ of functors $\ccal(Z)\to\ccal(S)$, with $r$ the projection $\A^1_{(-)}\to (-)$ (see the discussion above \cite[Proposition~5.7]{MR3570135} for more details). 
Using the closed projection formula (\Cref{cor:closed-proj} which holds for $\ccal$ as before), if $M\in\ccal(S)$, then the above span evaluated at $s^*M$ can be written 
$$p_\sharp s_*(\un_Z)\otimes M \to r_\sharp \hat{p}_\sharp\hat{s}_* r^*(\un_Z)\otimes M\leftarrow q_\sharp \Sigma^{\ncal_s}(\un_Z)\otimes M.$$ As $s_*$ is fully faithful by \Cref{cor:i_*_is_ff}, in fact all objects of $\ccal(Z)$ are of the form $s^*M$ and it suffices to prove that the two morphism in the span are equivalences when evaluated at $\un_Z$. Using \Cref{lem:univ-rho-commutes-i_*}, it suffices to prove it in $\mathrm{H}_\bullet(S)$, where it is \cite[Proposition 5.7]{MR3570135}.
\end{proof}
\begin{constr}
In \cite[Section 5.3]{MR3570135}, Hoyois constructs the unit and counit of the ambidexterity theorem. We need the following. Let $f\colon X\to S$ be a finitely presented projective smooth morphism of qcqs schemes. Then there exists (depending of some choice of an embedding of $X$ in a projective space over $S$) a map (called the Pontryagin–Thom collapse)
\[\eta\colon \mathrm{S}^\mcal \to\mathrm{Th}_X(\ncal)\] in $\mathrm{H}_\bullet(S)$; here $\mcal$ and $\ncal$ are vector bundles on $S$ and $X$ that satisfy $\mathrm{S}^{\Omega_f}\otimes \mathrm{S}^\ncal\simeq \mathrm{S}^{f^*(\mcal)}$ with $\Omega_f$ the locally free sheaf of Kähler differentials. By \Cref{SigmaM} and the smooth projection formula, this defines a natural transformation
\[\eta\colon \Sigma^\mcal \simeq \mathrm{S}^\mcal\otimes (-)\xrightarrow{\rho_\ccal(\eta)\otimes (-)} \mathrm{Th}_X(\ncal)\otimes(-)\simeq f_\sharp\Sigma^{\ncal}f^*\] of functors $\ccal(S)\to\ccal(S)$.

He also defines a map (\cite[(5.20)]{MR3570135}) $\varepsilon\colon f^*f_\sharp\to \Sigma^{\Omega_f}$ in $\mathrm{H}_\bullet(X)$ and we can copy his definition to get a map in $\ccal(X)$ as follows:
$$\varepsilon\colon f^*f_\sharp \simeq (p_2)_\sharp p_1^* \to (p_2)_\sharp\delta_*\delta^*p_1^*\simeq (p_2)_\sharp\delta_*\overset{\text{\ref{purity}}}{\simeq} \Sigma^{\ncal_\delta}\overset{\nu}{\simeq} \Sigma^{\Omega_f},$$ with $p_1,p_2\colon X\times_S X\to X$ the projections, $\delta\colon X\to X\times_S X$ the diagonal, $\nu\colon \ncal_\delta\to\Omega_f$ the isomorphism sending $x\otimes 1 - 1\otimes x$ to $dx$, and the one before last isomorphism is purity \Cref{purity}, that we may apply because the morphism $X\to S$ is of finite type so that the diagonal $\delta$ is finitely presented by  \cite[\href{https://stacks.math.columbia.edu/tag/0818}{Tag~0818}]{stacks-project}.
\end{constr}
\begin{thm}[Ambidexterity]
	\label{ambidextry}
	Let $f\colon X\to S$ be a finitely presented projective smooth morphism of qcqs schemes. Then the compositions 
	\[f^*\Sigma^{\mcal}\xrightarrow{\eta} f^*f_\sharp\Sigma^{\ncal}f^* \xrightarrow{\varepsilon} \Sigma^{\Omega_f} \Sigma^{\ncal} f^*\simeq \Sigma^{f^*(\mcal)}f^*\simeq f^*\Sigma^{\mcal}\] and
	\[\Sigma^{\mcal}f_\sharp \xrightarrow\eta f_\sharp\Sigma^{\ncal}f^* f_\sharp \xrightarrow\varepsilon f_\sharp\Sigma^{\ncal}\Sigma^{\Omega_f}\simeq f_\sharp \Sigma^{f^*(\mcal)} \simeq \Sigma^{\mcal}f_\sharp\]
	are the identity in $\ccal$.

	In particular, if $\Sigma^\mathcal{V}$  is an equivalence in $\ccal(S)$ for any vector bundle $\mathcal{V}$, then the map $\varepsilon$ induces a map $f^*f_\sharp \Sigma^{-\Omega_f}\to\mathrm{Id}$ that exhibits $f^*$ as the left adjoint of $f_\sharp\Sigma^{-\Omega_f}$ and yields a natural equivalence $f_\sharp\Sigma^{-\Omega_f}\xrightarrow{\sim} f_*$.
\end{thm}
\begin{proof}
	As in the proof of \cite[Theorem 5.20]{MR3570135}, if $\iota$ is the first map evaluated at the unit, then the two maps are $\iota\otimes f^*(-)$ and $f_\sharp(\iota\otimes (-))$ thus it suffices to prove that $\iota$ is the identity. By construction we have $\iota = \rho_\ccal(\iota)$ and by \cite[Theorem 5.20]{MR3570135}, $\iota$ is the identity in $\mathrm{H}_\bullet(S)$. Thus the proof is finished as the second point is a direct consequence (see \cite[Theorem 6.9]{MR3570135}).
\end{proof}

\begin{defi}
	\label{def:qcqsmotiviccoeffsys}
	Let $\ccal$ be an unstable qcqs motivic coefficient system \Cref{def:unstableqcqssys}. We say that $\ccal$ is a \emph{qcqs motivic coefficient system} if 
	\begin{enumerate}
	\item It takes values in stable presentably symmetric monoidal $\infty$-categories.
	\item It is \emph{Tate-stable} meaning that the Tate motive $T_{\Spec(\Z)}\coloneqq \mathrm{S}^{\ocal_{\Spec(\Z)}}$ is $\otimes$-invertible.
	\end{enumerate}
\end{defi}
\begin{rem}
	Let $\ccal$ be a qcqs motivic coefficient system. Using Tate-stability, the proof of \cite[Theorem~7.14]{MR4560376} can easily be adapted to prove that there is a map 
	\[\SH\to \ccal\] where $\SH$ is the coefficient system of Nisnevich motivic spectra. Let now $\mcal$ be a locally free $\ocal_X$-module; since in $\SH$ the sphere $\mathrm{S}^\mcal$ is $\otimes$-invertible by \cite[Theorème~1.5.7]{MR2423375}, so is its image in $\ccal$ which is the object we also denoted by $\mathrm{S}^\mcal$. Hence, the map $\Sigma^\mcal$ is an equivalence. 
\end{rem}

\begin{cor}[Ambidexterity]
	\label{ambidextry_stable} Working over a qcqs motivic coefficient system, let $f\colon X\to S$ be a finitely presented projective smooth morphism of qcqs schemes. There is a natural equivalence $f_\sharp\Sigma^{-\Omega_f}\xrightarrow{\sim} f_*$.
\end{cor}

\subsection{The six functors for pro-étale motives}\label{sec:sixfunctors}

Because the localization property is only true for closed immersions with qcqs open immersion, we do not get a 6 functor formalism as good as for $\mathrm{SH}$. Fix a qcqs motivic coefficient system $\ccal$.

\begin{prop}
	\label{lem:exchange-open-proper}
		\label{prop:PBC}
	Let $p\colon Y\to X$ be a finitely presented proper map.
	\begin{enumerate}
		\item Let \[\begin{tikzcd}
	{Y'} & {X'} \\
	Y & X
	\arrow["{p'}", from=1-1, to=1-2]
	\arrow["{j'}"', from=1-1, to=2-1]
	\arrow["j", from=1-2, to=2-2]
	\arrow["p", from=2-1, to=2-2]
\end{tikzcd}\] be a cartesian square of qcqs schemes with $j$ a smooth map. Then the exchange map
\[j_\sharp p'_*\to p_*j'_\sharp\] is invertible.
	\item The functor $p_*$ commutes with colimits.
	\item Let \[\begin{tikzcd}
	{Y'} & Y \\
	{X'} & X
	\arrow["{f'}", from=1-1, to=1-2]
	\arrow["{p'}"', from=1-1, to=2-1]
	\arrow["p", from=1-2, to=2-2]
	\arrow["f", from=2-1, to=2-2]
\end{tikzcd}\] be a cartesian square of qcqs schemes. Then the natural transformation 
\[f^*p_*\to p'_*(f')^*\] is an equivalence.
	\item For any $M$ in $\ccal(X)$ and any $N$ in $\ccal(Y)$, the natural map 
	\[(p_*M)\otimes N\to p_*(M\otimes p^* N)\] is an isomorphism.
	\end{enumerate}
\end{prop}
\begin{proof}
	The proof follows the proofs of \cite[Proposition 4.1.1, Proposition 4.1.2]{MR4466640}.
	We first assume that $p$ is projective. As all claims are local on $X$ we may assume that $p = \pi\circ i$ with $\pi\colon \mathbb{P}^n_X\to X$ the projection and $i\colon Y\to X$ a finitely presented closed immersion. We may treat $i$ and $\pi$ separately. The statements for $i_*$ follow from localization (and \Cref{cor:closed-proj}). Using ambidexterity we have $\pi_*\simeq \pi_\sharp\Sigma^{-\Omega}$, thus the four statements for $\pi_*$ follow from their analogue for $\pi_\sharp$, which are clear.

	We now reduce to the projective case.
	We begin by noting that by Zariski descent, it suffices to prove that each qcqs open immersion $\nu\colon V\to Y$ with base change $\nu'$ to $Y'$, the map
	$j_\sharp p'_*\circ \nu'_\sharp\to p_*j'_\sharp \circ\nu'_\sharp$ is an equivalence (1.), the functor $p_*\circ \nu_\sharp$ preserves colimits (2.), $f^*p_*\circ \nu_\sharp\to p'_*(f')^*\circ \nu_\sharp$ is an equivalence (3.) and $(p_*\circ\nu_\sharp M)\otimes N\to (p_*\circ \nu_\sharp)(M\otimes \nu^*p^*N)$ is an equivalence (4., here we now have $M\in\ccal(V)$). We fix such a $\nu\colon V\to Y$.
	By the refined Chow's lemma \cite[Corollary 2.6]{zbMATH05269409} we may find a blow-up $e\colon Z\to Y$ with center disjoint from $V$ such that $q = p\circ e$ is a projective morphism. 
	We let $w\colon V\to Z$ be an open immersion such that $\nu = e\circ w$ (this exists because $e$ is an isomorphism above $V$). We define $Z' = Z\times_Y Y'$ and $e'\colon Z'\to Y'$, $q'\colon Z'\to X'$ and $w'\colon V'\to Z'$ to be the base changes of $e$ , $q$ and $w$ along $f$. 
	
	We claim that $\nu_\sharp \simeq e_*\circ w_\sharp$ and $\nu'_\sharp\simeq e'_*\circ w'_\sharp$. The proof of the claim is the same in both cases thus we only deal with the first one. 
	Note first that there exists a comparison map $\alpha$ obtained by adjunction from 
	$\mathrm{Id}\to \nu^*e_*w_\sharp \simeq w^*w_\sharp$.  Let $i\colon (Y\setminus V)_\mathrm{red}\to Y$ be the reduced closed immersion. By localization, the pair $(i^*,\nu^*)$ is conservative. We have $\nu^*(\alpha) = \mathrm{Id}_{\mathrm{\nu^*\nu_\sharp}}$ by smooth base change. We have $i^*(\alpha) = \mathrm{Id}_0$ by projective base change (that we proved in the first paragraph). This proves the claim.

	Using the claim, we have: \begin{enumerate}
		\item $(j_\sharp p'_*\circ \nu'_\sharp \to p_*j'_\sharp\circ \nu'_\sharp)=(j_\sharp (p'\circ e')_*\circ w'_\sharp \to (p\circ e)\circ w_\sharp\circ j''_\sharp)$ with $j''$ the base change of $j'$ to $U$. 
		\item $p_*\circ \nu_\sharp \simeq (p\circ e)_*\circ w_\sharp$.
		\item $(f^*p_*\circ \nu_\sharp \to p'_*(f')^*\circ \nu_\sharp)\simeq (f^*(p\circ e)_*\circ w_\sharp \to (p'\circ e')_*\circ w'_\sharp \circ  (f'')^*$) with $f''$ the base change of $f'$ to $V$.
		\item \begin{eqnarray*}  & &[(p_*\circ\nu_\sharp M)\otimes N\to (p_*\circ \nu_\sharp)(M\otimes \nu^*p^*N)]\\ 
			&\simeq& [((p\circ e)_*\circ w_\sharp M)\otimes N \to ((p\circ e)_*\circ w_\sharp)(M\otimes w^*(p\circ e)^*N)] \\
			&\simeq & [(p\circ e)_*(w_\sharp M)\otimes N\to (p\circ e)_*((w_\sharp M)\otimes (p\circ e)^*N)] 
		\end{eqnarray*} using smooth projection formula for $w$. 
	\end{enumerate}
	In particular, we may replace $p$ and $p'$ by $q$ and $q'$ thus reducing to the case where $p$ is a finitely presented projective map. 

\end{proof}


\begin{emptypar}
	\label{par:factoSeparated}
	Let $f\colon X \to S$ be a separated finitely presented morphism of qcqs schemes. Then, there is a finitely presented compactification of $f$ \textit{i.e.} a finitely presented proper morphism $\overline{f}\colon \overline{X}\to S$ and a qcqs open immersion $j\colon X\to \overline{X}$ such that $f=\overline{f}\circ j$. 
	To see that, we may write $S=\lim S_i$ as a cofiltered limit of a system of Noetherian schemes with affine transition maps by \cite[\href{https://stacks.math.columbia.edu/tag/01ZJ}{Tag~01ZJ}]{stacks-project}. We can then choose an index $i$ and a morphism $X_i\to S_i$ of finite presentation whose base change to $S$ is $X\to S$, by \cite[\href{https://stacks.math.columbia.edu/tag/01ZM}{Tag~01ZM}]{stacks-project}. After increasing $i$ we may assume $X_i\to S_i$ is separated, by \cite[\href{https://stacks.math.columbia.edu/tag/01ZQ}{Tag~01ZQ}]{stacks-project}.
	But we can find a compactification of $X_i$ over $S_i$ by \cite[\href{https://stacks.math.columbia.edu/tag/0F41}{Tag~0F41}]{stacks-project}; it is then finitely presented and the base change of this to $S$ is a compactification of $X$. 

	In the above setting, we can set $f_!\coloneqq \overline{f}_* j_\sharp$. Note that this functor preserves colimits: it is clear for $j_\sharp$ and follows from \Cref{ambidextry_stable} for $\overline{f}_*$.
	We can then follow \cite[\S~4.3]{MR4466640} to upgrade $f\mapsto f_!$ into a functor $\ccal_!\colon \Sch_\mathrm{qcqs}^{\mathrm{fp}}\to \PrL$ from the category of qcqs schemes with finitely presented (non-necessarily separated) morphisms. In fact we can do better (see \Cref{thm:motivic_6FF}).
\end{emptypar}
The following extension of \Cref{purity} and \Cref{ambidextry} is classical:
\begin{cor}
	\begin{enumerate}
		\item Let $f\colon Y\to X$ be a smooth map of qcqs schemes. Then there is a canonical isomorphism $f^!\simeq f^*\Sigma^{\Omega_f}$ with $\Omega_f$ the sheaf of differentials of $f$.
		\item Let $f\colon Y\to X$ be a finitely presented smooth and proper map. Then there is a natural equivalence $f_\sharp\Sigma^{-\Omega_f}\xrightarrow{\sim} f_*$.
	\end{enumerate}
\end{cor}
\begin{proof}
	The second point readily follows from the first, using that $f_!\simeq f_*$ for a proper map. The proof of the first is classical, see for example \cite[Theorem 2.44]{khan2021voevodsky} (one checks similarly as before that the proof does not use non qcqs open immersions).
\end{proof}

\begin{thm}\label{thm:motivic_6FF}
     Let $\ccal\colon \Sch\to \CAlg(\PrL)$ be a qcqs motivic coefficient system. Then it is endowed with the six functors. More precisely, this means:
     \begin{enumerate}
    \item There exists a lax-symmetric monoidal functor \[\ccal_!^*\colon \mathrm{Corr}(\mathrm{Sch},\mathrm{fp})^\otimes\to \mathrm{Pr}^{\mathrm{L},\otimes}\] extending $\ccal$, where $\mathrm{Corr}(\mathrm{Sch},\mathrm{fp})$ is the category of correspondences defined in \cite[Definition 2.2.10]{HeyerMann} (see \Cref{rem:6FFexplained}), with backward maps any map, and onward maps finitely presented morphisms of qcqs schemes. Roughly, this means that we have all the higher coherence for the following:\begin{itemize}
		\item For each finitely presented map of qcqs schemes $f$, we have a functor $f_!$, which is compatible with the formation of pullbacks.
		\item  We have a natural isomorphism $$(f_!M)\otimes N \rar f_!(M\otimes f^*N).$$
		\item If $f$ is finitely presented proper, then $f_!$ is the right adjoint of $f^*$, and if $f$ is a qcqs open immersion, then $f_!$ is the left adjoint of $f^*$.
	\end{itemize}
    \item For any $f\colon X\to S$ finitely presented smooth morphism of qcqs schemes, there is a natural equivalence $ f^! \simeq \Sigma^{\Omega_f} \circ f^*$.
    \item For any closed immersion $i\colon Z\rar X$ with qcqs open complement $j\colon X\setminus Z\rar X$, we have exact triangles of natural transformations $$j_!j^!\rar \id \rar i_*i^*$$
    $$i_!i^!\rar \id \rar j_* j^*$$ given by the unit and counit of the appropriate adjunctions.
	\item There is a natural transformation $f_!\to f_*$ which is an equivalence if $f$ is proper of finite presentation.
\end{enumerate}
\end{thm}
\begin{proof}
	Let $\mathrm{Corr}(\mathrm{Sch},\mathrm{fp}^\mathrm{sep})$ be the category of correspondences whose objects are qcqs schemes, backward morphisms are any morphisms, and onward morphisms are finitely presented separated maps (see \cite[Definition 2.2.10]{HeyerMann}). The monoidal structure is induced by the cartesian structure on schemes.
	We first apply \cite[Theorem 4.35]{CLLSpan}, using  \Cref{prop:PBC}, \Cref{lem:exchange-open-proper}  and \Cref{par:factoSeparated}. This provides a lax symmetric monoidal $2$-functor 
	\[\mathbf{\ccal}^*_!\colon \mathrm{SPAN}_2(\Sch,\mathrm{fp}^\mathrm{sep})^\otimes_{\mathrm{prop,open}}\to\mathbf{Pr}^{\mathrm{L},\otimes},\] where $\mathrm{SPAN}_2(\Sch,\mathrm{fp}^\mathrm{sep})^\otimes_{\mathrm{prop,open}}$ is a $(2,2)$-category of correspondences defined in \cite[Construction 4.1]{CLLSpan} (in the notations of \emph{loc. cit.} we take $C$ the category of qcqs schemes, $E$ the class of finitely presented separated maps, $I$ the class of qcqs open immersions and $P$ the class of finitely presented proper maps) and $\mathbf{Pr}^\mathrm{L}$ is the $(\infty,2)$-category of presentable $\infty$-categories, endowed with the Lurie tensor product. The underlying $(2,1)$-category of this $(2,2)$-category is the category $\mathrm{Corr}(\mathrm{Sch},\mathrm{fp}^\mathrm{sep})$ by \cite[Lemma 4.13]{CLLSpan}. In particular, by composing with the canonical lax monoidal inclusion $$\mathrm{Corr}(\mathrm{Sch},\mathrm{fp}^\mathrm{sep})^\otimes\to \mathrm{SPAN}_2(\Sch,\mathrm{fp}^\mathrm{sep})^\otimes_{\mathrm{prop,open}}$$ and forgetting the $2$-morphisms, we obtain a lax symmetric monoidal $\infty$-functor 
	$$\ccal^*_!\colon \mathrm{Corr}(\mathrm{Sch},\mathrm{fp}^\mathrm{sep})^\otimes \to \mathrm{Pr}^{\mathrm{L},\otimes}$$ which is almost all we wanted. The claim was about having for $!$-able maps all finitely presented and not only separated thus we have to do some Zariski descent. Note that for any finitely presented map $f\colon Y\to X$, there exists an hypercovering $Y_\bullet\to Y$ for the Zariski topology such that each composite $Y_n\to Y\to X$ is separated (for example, we may take the $Y_n$ to be affine). We now use the terminology introduced in \cite[Definition 4.5.1]{HeyerMann}. By definition, for any Zariski covering $j\colon U\to S$, we have $j^*\simeq j^!$, where $j^!$ is the right adjoint of $j_!$. Thus by \cite[Lemma 4.5.4]{HeyerMann}, the map $f$ is $\ccal$-suave. Moreover, by Zariski descent, the functor $j^*$ is conservative, thus by \cite[Lemma 4.7.1]{HeyerMann}, the map $j$ is a universal $\ccal^!$-cover. All of this put together allows us to apply \cite[Proposition 3.4.8 (ii)]{HeyerMann} to conclude that the functor $\ccal^*_!$ extends uniquely to a 6-functor formalism on $\mathrm{Corr}(\mathrm{Sch},\mathrm{fp})$ the category of correspondences where now all finitely presented maps are onward maps: this provides a lax symmetric monoidal functor
	\[\ccal^*_!\colon \mathrm{Corr}(\mathrm{Sch},\mathrm{fp})^\otimes\to \mathrm{Pr}^{\mathrm{L},\otimes}\] as promised.
	
	The second point is \Cref{ambidextry_stable}, the third point is an assumption on $\ccal$ and the fourth point follows from the third: given a factorization $f=p\circ j$, applying the localization triangle to $j_*$ and further applying $p_*$, we obtain the natural transformation.
\end{proof}

\begin{rem}
	\label{rem:6FFexplained}
	The category $\mathrm{Corr}(\mathrm{Sch},\mathrm{fp})$ has all qcqs schemes for objects. If $X$ and $Y$ are qcqs, the morphisms $X\to Y$ in the category of or correspondences are  
 	\[\alpha\colon X \xleftarrow{f} Z\xrightarrow{g} Y\] where $f$ is any morphism and $g$ is a finitely presented morphism. Composition is given by the obvious pullback diagram. The functor $\ccal^*_!$ we obtain has the additional property that $\ccal^*_!(\alpha)\simeq g_!f^*$.  If $g$ is furthermore separated, then for any compactification $g=p\circ j$ with $p$ finitely presented proper and $j$ a qcqs open immersion, we have $g_!\simeq p_*j_!$. If $g$ is not separated anymore, one can choose a Zariski hypercovering $Z_\bullet\to Z$ by affines, and by definition, if one denote by $g^n\colon Z_n\to Y$ the composition $Z_n\to Z\to Y$ of $g$ with the morphism $j_n\colon Z_n\to Z$, we have:\[g_!\simeq \colim_{n\in\Delta}g^n_!j_n^*.\]

\end{rem}
\begin{cor}\label{cor:6FF_proet_mot}
    We have the six functors on $\SH_\proet$ in the sense of \Cref{thm:motivic_6FF}.
\end{cor}

\subsection{The embedding theorem from étale motives}\label{sec:embeddingetalemotives}
If $\Lambda$ is a condensed ring spectrum, it can be seen as an object of $\Sh_{\proet}(\WSm_X,\Sp)$ by using the canonical maps of sites
\[*_\proet \to X_\proet \to (\WSm_X)_\proet.\] We let \[
\SH_\proet(X,\Lambda)\coloneqq \Mod_\Lambda(\SH_\proet(X)) .
\]
If the condensed ring spectrum $\Lambda$ is $H\Z$-linear, we denote by $\DM_\proet(X,\Lambda)$ the $\infty$-category $\SH_\proet(X,\Lambda)$ and call it the category of \emph{pro-étale motives with $\Lambda$-coefficients}; in that case, we also denote by $\D_\proet^{\A^1}(\WSm_X,\Lambda)$ the category $\Mod_\Lambda(\mathrm{SH}^{S^1}_\proet(X))$.


\begin{defi}\label{def:etale_bounded}
	Let $X$ be a qcqs scheme.
	\begin{enumerate}
		\item We say that $X$ is étale bounded if it has finite Krull dimension and if  
		\[\sup_{x\in X,p\text{ prime}} \mathrm{cd}_p(k(x))<+\infty\]
	where $\mathrm{cd}_p(k)$ is the mod $p$ Galois cohomological dimension of a field $k$. 
		\item We say that $X$ is locally étale bounded if it is étale-locally étale bounded.
	\end{enumerate}

\end{defi}

\begin{thm}\label{thm:embedding_etale_motives}
	Let $X$ be a locally étale bounded qcqs scheme.
	Then for any ring spectrum $ \Lambda$ (endowed with the discrete topology), the canonical functor
	\[
		\SH_\et(S,\Lambda)\to \SH_\proet(S,\Lambda)
	\]
	is fully faithful. 
\end{thm}
\begin{proof}
	As the $\Lambda$-linear version is obtained by taking $\Lambda$-modules on both sides, we only deal with spectral coefficients. We can reduce to étale bounded schemes because both $\SH_\et$ and $\SH_\proet$ have étale descent.

	We begin by showing that the functor
	\[
		\nu^* \colon \Sh_{\et}(\Sm_{X}) \to \Sh_{\proet}(\WSm_{X})
	\]
	is fully faithful, that is we want to see that the canonical map
	\[
		\Hom_{\Sh_{\et}(\Sm_{X})}(\mathcal{F},\mathcal{G}) \to 	\Hom_{\Sh_{\proet}(\WSm_{X})}(\nu^* \mathcal{F},\nu^*\mathcal{G})
	\]
	is an equivalence.
	We may assume that $ \mathcal{F} = T $ for a smooth $ X $-scheme $ T $.
	Then the map may be rewritten as
	\[
		\Hom_{\Sh_{\et}(\Sm_{T})}(\un,f^* \mathcal{G}) \to \Hom_{\Sh_{\proet}(\WSm_{T})}(\un,f^* \nu^*\mathcal{G})
	\]
	Since the unit is in the essential image of the inclusion functors from the small étale sites we may again identify this as
	\[
			\Hom_{\Sh(T_{\et})}(\un, \theta  f^* \mathcal{G}) \to \Hom_{\Sh({T}_{\proet})}(\un, \theta' \nu^* f^*\mathcal{G})
	\]
	where the functors $ \theta  $ and $ \theta' $ are the right adjoints to the inclusion from the small étale site.
	Furthermore we have a commutative square
	\[\begin{tikzcd}
		{\Sh(T_{\et})} & {\Sh(T_{\proet})} \\
		{\Sh_{\et}(\Sm_{T})} & {\Sh_{\proet}(\WSm_{T})}
		\arrow["{\nu^*}"', swap, from=1-1, to=1-2]
		\arrow["\theta"', swap, from=2-1, to=1-1]
		\arrow["{\nu^*}", from=2-1, to=2-2]
		\arrow["{\theta'}", from=2-2, to=1-2]
	\end{tikzcd}\]
	(all functors preserve colimits because of \cite[Proposition~7.1]{MR4296353} and on representables this is easy).
	So we may reduce to $ \nu^* $ being fully faithful on the small site, which follows from Postnikov-completeness of both sides (which is a consequence of \cite[Proposition 3.2.3]{zbMATH06479630} on the pro-étale site and of our assumption on finite cohomological dimension for the étale site, see \cite[Lemma 2.23]{MattisUnstable}) and the truncated case which is \cite[Proposition 2.38]{bastietal}.

	The stable analogue of $\nu^*$ is thus also fully faithful.
	We now claim that it restricts to a fully faithful functor 
	\[
		\mathrm{SH}^{S^1}_{\et}(X) \to \mathrm{SH}^{S^1}_{\proet}(X).
	\]
	In other words we have to show that if $ \mathcal{F}$ is an $\A^1$-invariant sheaf of spectra on $\Sm_X$, then $ \nu^* \mathcal{F}$ is also $\A^1$-invariant.
	We first claim that $\nu^*$ is t-exact for the standard t-structure: as restricting to small sites is t-exact and commutes with $\nu^*$, we are reduced to showing that $\nu^*\colon \Sh(T_\et,\mathrm{Sp})\to \Sh(T_\proet,\mathrm{Sp})$ is t-exact for any smooth $X$-scheme $T$. The result then follows from \cite[Remark~1.3.2.8]{SAG}.

	So let $ U = \lim_i U_i$ be affine and pro-\'etale over some smooth $X$-scheme $T$.
	We want to show that
	\[
		(\nu^*\mathcal{F})(U) \to (\nu^*\mathcal{F})(U \times \A^1)
	\]
	is an equivalence.
	Since $\nu^*\mathcal{F} \simeq \lim_n \tau_{\leq n} \nu^* \mathcal{F}$ by Postnikov-completeness of $\Sh_\proet(\WSm_X)$ (which is again a consequence of \cite[Proposition 3.2.3]{zbMATH06479630}) and because $\nu^*$ commutes with truncations, we may assume that $\mathcal{F}=\tau_{\leq n} \mathcal{F}$.
	But then the above map rewrites as
	\[
		\colim_i (\nu^*\mathcal{F})(U_i) \to \colim_i (\nu^*\mathcal{F})(U_i \times_X \A^1_X);
	\]
	indeed, the formula for the left-hand side can be checked by restricting to $\Sh(T_\proet)$ and the formula for the right-hand side can be checked by restricting to $\Sh((\A^1_T)_\proet)$ and in both cases, the argument is then the same as in \cite[Corollary 5.1.6]{zbMATH06479630} which is the case of the derived category of sheaves of abelian groups. Thus our map is an equivalence.

	Now the claim of the theorem follows formally, if we can show that the square
\[\begin{tikzcd}
	{\mathrm{SH}_{\et}^{S^1}(X)} & {\mathrm{SH}_{\proet}^{S^1}(X)} \\
	{\mathrm{SH}_{\et}^{S^1}(X)} & {\mathrm{SH}_{\proet}^{S^1}(X)}
	\arrow["{\nu^*}", from=1-1, to=1-2]
	\arrow["{\underline{\Hom}(\mathbb{S}[\mathbb{G}_m],-)}"', from=1-1, to=2-1]
	\arrow["{\underline{\Hom}(\mathbb{S}[\mathbb{G}_m],-)}", from=1-2, to=2-2]
	\arrow["{\nu^*}"', from=2-1, to=2-2]
\end{tikzcd}\]
	commutes in the obvious way.
	That is we have to show that the canonical map
	\[
	\nu^* \underline{\Hom}(\mathbb{S}[\mathbb{G}_m],\mathcal{F}) \to \underline{\Hom}(\mathbb{S}[\mathbb{G}_m],\nu^* \mathcal{F})
	\]
	is an equivalence for any $\mathcal{F} \in \mathrm{SH}_{\et}^{S^1}(X)$.
	Note that here the internal Homs a priori are taken in $\mathrm{SH}_{\et}^{S^1}(X)$ and $\mathrm{SH}_{\proet}^{S^1}(X)$, respectively, but the latter in fact agree with the internal Homs in $\Sh_{\et}(\Sm_{X},\Sp) $ and $ \Sh_{\proet}(\WSm_{X},\Sp)$, respectively.
	Thus we may argue in the categories of sheaves, instead of effective motives.
	Then the claim follows from \Cref{prop:big-proet-f_*-compatible} and \Cref{prop:hypothesis_big-proet-f_*-compatible_are_satisfied} below (indeed note that $\underline{\Hom}(\mathbb{S}[Y],-)$, for $f\colon Y\to X$ a smooth map, is isomorphic to $f_*f^*$ by adjunction and smooth projection formula). 
\end{proof}


\begin{prop}
	\label{prop:big-proet-f_*-compatible}
	Let $S$ be an \'etale bounded qcqs scheme.
	Let $f \colon T \to  S$ be a map of schemes such that for any base change $ f' \colon X \times_S T \to X $ to a smooth $S$-scheme $X$, the functor $f_* \colon \Sh((X \times_S T)_\et,\mathbb{Z}) \to \Sh(X_\et,\mathbb{Z})$ has finite cohomological dimension.
	Then the square
\[\begin{tikzcd}
	{\Sh_{\et}(\Sm_T,\Sp)} & {\Sh_{\proet}(\WSm_T,\Sp)} \\
	{\Sh_{\et}(\Sm_S,\Sp)} & {\Sh_{\proet}(\WSm_S,\Sp)}
	\arrow["{\nu^*}", from=1-1, to=1-2]
	\arrow["{f_*}"', from=1-1, to=2-1]
	\arrow["{f_*}", from=1-2, to=2-2]
	\arrow["{\nu^*}"', from=2-1, to=2-2]
\end{tikzcd}\]
	commutes.
\end{prop}
\begin{proof}
	For any smooth map $g \colon X \to S$, we consider the cube
\[\begin{tikzcd}
	& {\Sh((X\times_S T)_{\et},\Sp)} & {\Sh((X\times_S T)_{\proet},\Sp)} \\
	& {\Sh(X_{\et},\Sp)} & {\Sh(X_{\proet},\Sp)} \\
	{\Sh_{\et}(\Sm_T,\Sp)} & {\Sh_{\proet}(\WSm_T,\Sp)} \\
	{\Sh_{\et}(\Sm_S,\Sp)} & {\Sh_{\proet}(\WSm_S,\Sp)}
	\arrow["{\nu^*}", from=1-2, to=1-3]
	\arrow["{f'_*}", from=1-2, to=2-2]
	\arrow["{f'_*}", from=1-3, to=2-3]
	\arrow["{\nu^*}", from=2-2, to=2-3]
	\arrow[from=3-1, to=1-2]
	\arrow["{\nu^*}", from=3-1, to=3-2]
	\arrow["{f_*}"', from=3-1, to=4-1]
	\arrow[from=3-2, to=1-3]
	\arrow["{f_*}", from=3-2, to=4-2]
	\arrow[from=4-1, to=2-2]
	\arrow["{\nu^*}"', from=4-1, to=4-2]
	\arrow[from=4-2, to=2-3]
\end{tikzcd}\]
	where the diagonal functors are given by base changing and then restricting.
	Note that the left-hand square and the right-hand square of the cube clearly commute.
	The top and bottom already commute at the level of sheaves of spaces as we have seen in the proof of the theorem.
	Since the collection of all functors of the form $\Sh_\proet(\WSm_S,\mathrm{Sp})\to \Sh(X_\proet,\mathrm{Sp})$ for all smooth $g \colon X \to S$ is conservative, we may reduce to showing that the back square commutes.
	We observe that for bounded above sheaves of spectra, this follows easily by adapting the proof of \cite[Lemma~5.4.3]{zbMATH06479630}.

	To extend to all sheaves of spectra, note that the finiteness assumptions imply that the horizontal functors are fully faithful.
	It therefore suffices to see that $f'_*(\nu^* \mathcal{F})$ is in the image of $\nu^*$: indeed then, $f'_*(\nu^* \mathcal{F})=\nu^*\nu_*(f'_*(\nu^* \mathcal{F}))$,
	\[
		\nu_* (f'_*(\nu^* \mathcal{F}))(U) = f'_*(\nu^* \mathcal{F})(U) = \nu^* \mathcal{F}(f'^{-1}(U)) = \mathcal{F}(f'^{-1}(U)) = f'_*(\mathcal{F})(U)
	\]
	for any $U$ \'etale over $X$.
	Now, observe that we may identify the full subcategory of $\Sh_{\proet}(X,\Sp)$ spanned by the image of $\nu^*$ with the full subcategory spanned by those pro\'etale sheaves whose homotopy sheaves are classical by \cite[Lemma~3.5]{bachmannrigidity}.
	It therefore remains to see that $\pi_k(f'_*(\nu^* \mathcal{F})) \in \Sh(X_{\proet},\Sp)^{\heart}$ is classical.
	By Postnikov-completeness, we note that
	\[
		f'_*(\nu^*\mathcal{F}) = \lim_n f'_*(\nu^* \tau_{\leq n} \mathcal{F}).
	\]
	It follows that the inverse system $ \left(\pi_k f'_*(\nu^* \tau_{\leq n} \mathcal{F})\right)_k$ is eventually constant since the connectivity of the fiber of
	\[
		\nu^* f'_*(\pi_n \mathcal{F})=  f'_*(\nu^* \pi_n \mathcal{F}) = \mathrm{fib} (f'_*(\nu^* \tau_{\leq n} \mathcal{F}) \to f'_*(\nu^* \tau_{\leq n-1} \mathcal{F}))
	\]
	goes to infinity as $n$ grows, by our cohomological finiteness assumption.
	Thus for $n >\!\!> 0$ we have $ \pi_k f'_*(\nu^* \mathcal{F}) = \pi_k f'_*( \nu^*\tau_{\leq n} \mathcal{F})$ and the latter is classical since $\tau_{\leq n} \mathcal{F}$ is bounded above.
\end{proof}

\begin{lem}\label{prop:hypothesis_big-proet-f_*-compatible_are_satisfied}
	Let $S$ be an \'etale bounded qcqs scheme.
	Then any finite presentation morphism $ f \colon X \to S $ satisfies the assumptions of \Cref{prop:big-proet-f_*-compatible}.
\end{lem}
\begin{proof}
	First note that if $T\to S$ is any finitely presented map, 
	then $T$ is also étale bounded by \cite[Lemma 2.24]{MattisUnstable} 
	(the statement is for smooth morphisms but the proof works for any finite presentation morphism). In particular it suffices to prove that the functor 
	$Rf_*$ has finite cohomological dimension.
	Also, any étale bounded qcqs scheme of finite dimension has finite étale cohomological 
	dimension by \cite[Corollary 3.29]{MR4296353}, even better, for any $U\to X$ étale, we have  $\mathrm{DimCoh}(U)\leqslant \sup_{x\in X}\sup_{p\in\mathbb{P}}\mathrm{cd}_p(\kappa(x)) + \mathrm{dim}(X)=:N$. 
	Finally, by \cite[\href{https://stacks.math.columbia.edu/tag/072W}{Lemma 072W}]{stacks-project}, if $A\in\Sh(X_\et,\Z)$, the sheaf 
	$R^if_*A$ is the sheafification of the presheaf that sends an étale scheme $V\to S$ to $\HH^i_\et(f^{-1}(V),A)$. In particular, for $i>N$, we have $R^if_*A=0$, finishing the proof.
\end{proof}

\begin{rem}
\Cref{thm:embedding_etale_motives} also works if $X$ is finite dimensional, and \'etale locally,
\[\sup_{x\in X,p\in\mc{P}} \mathrm{cd}_p(k(x))<+\infty\] with $\mc{P}$ the set of primes that are not invertible in $\Lambda$.
The proof works the same, replacing sheaves of spectra with sheaves of $\mc{P}$-local spectra.
Indeed, the analogue of \Cref{prop:hypothesis_big-proet-f_*-compatible_are_satisfied} is still a consequence of \cite[Corollary 3.29]{MR4296353} and the Postnikov-completeness of objects in $\Sh(X_{\et},\Sp_{\mathcal{(P)}})$ used in the proofs of \Cref{thm:embedding_etale_motives} and \Cref{prop:big-proet-f_*-compatible} follows from \cite[Theorem 4.28]{MR4296353}.
\end{rem}

\begin{cor}
	Over locally étale bounded qcqs schemes, étale motives have non-effective pro-étale hyperdescent. That is, if $U_\bullet \to X$ is a pro-étale hypercovering with $X$ and each $U_n$ locally étale bounded, and if $M,N\in \mathrm{SH}_\et(X)$, the map 
	\[\mathrm{map}_{\mathrm{SH}_\et(X)}(M,N)\to \lim_{[n]\in\Delta}\mathrm{map}_{\mathrm{SH}_\et(U_n)}(M_{\mid U_n},N_{\mid U_n})\] is an equivalence of spectra. For $M=\mathbb{S}$ the unit and $N= H\Z_\et(i)$ the twisted étale-local motivic Eilenberg-Mac Lane spectrum with $i\in\Z$, this means that étale motivic cohomology has pro-étale descent.
\end{cor}
\begin{proof}
	This is true in $\mathrm{SH}_\proet$ by pro-étale descent, thus \Cref{thm:embedding_etale_motives} implies the result.
\end{proof}

\section{The rigidity theorem.}\label{sec:rigidity}
The main goal of this section is to prove a version of the \emph{rigidity theorem} in our context, see \Cref{thm:solid_rigidity_V1}.
In \Cref{sec:solidmotives}, we define categories of \emph{solid motives} by applying the abstract solidification procedure of \Cref{sec:solidcats} to categories of pro-\'etale motives.
We prove that solid pro-\'etale sheaves embed fully faithfully into these categories in \Cref{sec:embeddingsolid} and prove a rigidity theorem with torsion coefficients \Cref{prop:rig-torsion}.
Finally in \Cref{sec:solidrigidity}, we prove a version of the rigidity theorem with $\ZhatP$-coefficients.
In order to do so, we need to pass to a further localization of the categories introduced in \Cref{sec:solidmotives}, see \Cref{def:effectivesolidmotives}.
As a consequence of the rigidity theorem \Cref{thm:solid_rigidity_V1}, it follows that the categories solid pro-\'etale sheaves introduced in \Cref{sec:solid} have the six operations.

In this section, we fix a solid $\ZhatP$-algebra $\Lambda$, where $\mathcal{P}$ is a set of prime numbers, and a qcqs scheme $X$.
\subsection{Solid motives.}\label{sec:solidmotives}

\begin{defi}\label{def:cond_cat_usual}
We define condensed presentably monoidal $\infty$-categories by \[\underline{\D}_\proet(\WSm_X,\Lambda)\colon S \mapsto \D_\proet(\WSm_{X\times S},\Lambda)\]
\[\underline{\D}^{\A^1}_\proet(\WSm_X,\Lambda)\colon S \mapsto \D^{\A^1}_\proet(\WSm_{X\times S},\Lambda)\]
\[\underline{\DM}_\proet(X,\Lambda)\colon S \mapsto \DM_\proet(X\times S,\Lambda)\]
The $\infty$-category of \emph{solid $\A^1$-invariant sheaves} over $X$ is then defined as
	\[\D^{\A^1}(\WSm_X,\Lambda)^\blacksquare\coloneqq \underline{\D}^{\A^1}_\proet(\WSm_X,\Lambda)^\blacksquare(*).\]
We also define a Tate-stable version, as
	\[\D^{\A^1,\Gm}(\WSm_X,\Lambda)^\blacksquare\coloneqq \underline{\DM}_\proet(X,\Lambda)^\blacksquare(*).\]
	Finally, we set 	\[\D(\WSm_X,\Lambda)^\blacksquare\coloneqq \underline{\D}_\proet(\WSm_X,\Lambda)^\blacksquare(*).\]
\end{defi}
\begin{rem}
	All of the above categories coincide with $\Lambda$-modules in the analogous categories with $\widehat{\Z}_\mathcal{P}$-coefficients.
\end{rem}

\begin{thm}
	The functor $\D^{\A^1}(\WSm_{(-)},\Lambda)^\blacksquare \colon \Sch^\op \to \CAlg(\PrL)$ is an unstable qcqs motivic coefficient system.
	More explicitly we have the following properties
	\begin{enumerate}
		\item For every weakly smooth map $ f \colon X \to Y$ the functor $f^*$ has a $\D^{\A^1}(\WSm_{Y},\Lambda)^\blacksquare$-linear left adjoint that is compatible with base change.
		\item For any scheme $X$ the functor $\pr_1^* \colon \D^{\A^1}(\WSm_{X},\Lambda)^\blacksquare \to \D^{\A^1}(\WSm_{X \times \A^1},\Lambda)^\blacksquare$ is fully faithful.
		\item For any closed immersion $i \colon Z \to X$ with qcqs open complement $j \colon U \to X$, the canonical square
\[\begin{tikzcd}
	{\D^{\A^1}(\WSm_{U},\Lambda)^\blacksquare} & {\D^{\A^1}(\WSm_{X},\Lambda)^\blacksquare} \\
	\ast & {\D^{\A^1}(\WSm_{Z},\Lambda)^\blacksquare}
	\arrow["{j_\sharp}", from=1-1, to=1-2]
	\arrow[from=1-1, to=2-1]
	\arrow["{i^*}", from=1-2, to=2-2]
	\arrow[from=2-1, to=2-2]
\end{tikzcd}\]
	in $\PrL$ is cocartesian.
	\end{enumerate}
	The same is true for $\D^{\A^1,\Gm}(\WSm_{(-)},\Lambda)^\blacksquare$. 
\end{thm}
\begin{proof}
	Let us start by proving 1.
	Extending the notations of \Cref{def:cond_cat_usual} to arbitrary condensed $\mathbb{E}_\infty$-algebra, we see that for a weakly smooth map $f \colon X \to S$ we have a condensed adjunction 
	\[
		{f_\sharp}\colon {\underline{\D}^{\A^1}(\WSm_{X},\Sp)}\leftrightarrows{\underline{\D}^{\A^1}(\WSm_{S},\Sp)}\colon {f^*}
	\]
	in $\PrL_{\cond}$ and the left adjoint $f_\sharp$ is compatible with base change.
	Using \Cref{lem:tensoring_adjoints}, it follows that the functor $f^* \colon \D^{\A^1}(\WSm_{S},\Lambda)^\blacksquare\to \D^{\A^1}(\WSm_{X},\Lambda)^\blacksquare$ has a left adjoint given by $f_\sharp^\blacksquare = (f_\sharp \otimes^{\cond} \underline{\mathrm{Solid}}_\Lambda)(*)$.
	Thus, compatibility with base change follows immediately from the bifunctoriality of $-\otimes - $.
	To see that $f_\sharp^\blacksquare$ satisfies the projection formula, we note that by \cite[Proposition A]{zbMATH07785229} 
	there is a canonical comparison map
	\[
		f_\sharp^\blacksquare(F \otimes f^* G) \to f_\sharp^\blacksquare(F) \otimes G
	\]
	and it just remains to show that it is an equivalence.
	For this, we note that all the functors in question are given by applying their non-solid versions and then solidifying and thus the claim follows from the projection formula for $\D^{\A^1}(\WSm_{X},\Sp)$.

	We now prove 2.
	Similarly as before the adjunction $p_\sharp\dashv p^*$ lifts to an adjunction in $\PrL_{\cond}$ where the right adjoint is fully faithful.
	Thus the same is still true after applying $- \otimes^{\cond}\underline{\mathrm{Solid}}_\Lambda $ using \Cref{rem:2-func-of-tensor}.

	Finally for 3., we note that we have an analogous cocartesian square in $\PrL_{\cond}$ before applying $- \otimes^{\cond}\mathrm{Solid}_\Lambda$ by \Cref{thm:loc}.
	Since the latter functor is a left adjoint it preserves cocartesian squares and the claim follows using \Cref{lem:colimits_in_Prlcond}.
\end{proof}

\subsection{Embedding solid sheaves into solid motives}\label{sec:embeddingsolid}
We now prove that solid sheaves embed fully faithfully into the categories of solid motives that we consider.
\begin{lem}
	\label{lem:small-to-big-preserves-limits}
	Consider the fully faithful functor of condensed categories
	\[
		\rho_{\sharp} \colon \underline{\Sh}(X_{\proet}) \to \underline{\Sh}_{\proet}(\WSm_X).
	\]
	Then this functor is a condensed right adjoint.
\end{lem}
\begin{proof} By \Cref{prop:cond_left_adj}, we have to show that for any profinite set $K$, the functor \[\rho_{\sharp} \colon \Sh((X\times K)_{\proet}) \to \Sh_{\proet}(\WSm_{X\times K})\] preserves all limits and that for any map $\alpha\colon K\to K'$ of profinite sets, the commutative square
	\[\begin{tikzcd}
	{{\Sh}((X \times K')_\proet)} & {\Sh_{\proet}(\WSm_{X \times K'})} \\
	{{\Sh}((X \times K)_\proet) } & {\Sh_{\proet}(\WSm_{X \times K})}
	\arrow["{\rho_\sharp}", from=1-1, to=1-2]
	\arrow["{\alpha^*}", from=1-1, to=2-1]
	\arrow["{\rho_\sharp}"', from=2-1, to=2-2]
	\arrow["{\alpha^*}"', from=1-2, to=2-2]
\end{tikzcd}\]
is horizontally left adjointable, or equivalently vertically right adjointable.

	For this we consider the restriction functors to any small pro\'etale site $\Sh(Y_\proet)$ for $f \colon Y \to X \times K$ weakly smooth and consider the composite squares
\[\begin{tikzcd}
	{{\Sh}((X \times K')_{\proet})} & {\Sh_{\proet}(\WSm_{X \times K'})} & {{\Sh}((Y')_{\proet})} \\
	{{\Sh}((X \times K)_{\proet})} & {\Sh_{\proet}(\WSm_{X \times K})} & {{\Sh}(Y_{\proet})}
	\arrow["{\rho_\sharp}", from=1-1, to=1-2]
	\arrow["\mathrm{res}_{Y'}", from=1-2, to=1-3]
	\arrow["{\alpha^*}", from=1-1, to=2-1]
	\arrow["{\rho_\sharp}"', from=2-1, to=2-2]
	\arrow["{\alpha^*}", from=1-2, to=2-2]
	\arrow["{\mathrm{res}_Y}"', from=2-2, to=2-3]
	\arrow["{\alpha_Y^*}"', from=1-3, to=2-3]
\end{tikzcd}\]
	where we write $Y' = Y \times_{(X \times K')} (X \times K)$.
	Observe that the collection of functors $\mathrm{res}_Y$ for all $Y$ is jointly conservative.
	Since $\mathrm{res}_Y$ commutes with all limits and we have $ \mathrm{res}_Y \circ \rho_\sharp \simeq f^*$, it follows that $\rho_\sharp$ commutes with limits pointwise by \Cref{lem:pullbackslimits}.
	Similarly, it follows that we may check vertical right adjointability of left square above after applying  $\mathrm{res}_Y$ for all $Y$.
	Now note that the right square is vertically right adjointable as it is clearly horizontally left adjointable.
	Thus it suffices to show vertical right adjointability of the composite square.
	Via the identification $ \mathrm{res}_Y \circ \rho_\sharp \simeq f^*$, this follows from integral base change \cite[Proposition 6.18]{bastietal}, as in the proof of \Cref{lem:f-sharp-is-condensed}.
\end{proof}

\begin{cor}
	\label{cor:small-to-big-tensored}
	For any $ \mathcal{C} \in \PrL_{\cond} $, the functor $$\underline{\Sh}(X_{\proet}) \otimes^{\cond} \mathcal{C} \to \underline{\Sh}_{\proet}(\WSm_X) \otimes^{\cond} \mathcal{C}$$ is a fully faithful condensed right adjoint.
\end{cor}
\begin{proof}
	By \Cref{lem:small-to-big-preserves-limits}, 
	the functor $\rho_\sharp$ has a  condensed left adjoint. The lemma then follows from \Cref{lem:tensoring_adjoints_fully_faithful}.
\end{proof}

\emptypar \label{rho_sharp_commutes_with_limits} As a particular case of the above corollary we obtain a fully faithful and  limit-preserving functor
\[
	\rho_\sharp^\blacksquare \colon \D(X,\Lambda)^\blacksquare \to \D(\WSm_X,\Lambda)^\blacksquare
\]
by letting $\mathcal{C}=\underline{\mathrm{Solid}}_\Lambda$ and evaluating at $*$.

\begin{rem}
	Note that the canonical square
\[\begin{tikzcd}
	{\D(\WSm_{X\times\mathbb{A}^1},\Lambda)^\blacksquare} & {\D(\WSm_{X\times\mathbb{A}^1},\Lambda)} \\
	{\D(\WSm_X,\Lambda)^\blacksquare} & {\D(\WSm_{X},\Lambda)}
	\arrow[from=1-1, to=1-2]
	\arrow["{\pr_1^*}", from=2-1, to=1-1]
	\arrow[from=2-1, to=2-2]
	\arrow["{\pr_1^*}"', from=2-2, to=1-2]
\end{tikzcd}\]
	where the horizontal maps are the inclusions given by \Cref{lem:tensoring_adjoints_fully_faithful} commutes.
	Indeed, the square obtained by taking left adjoints commutes, since $(\pr_1)_\sharp$ commutes with solidification.
	It follows that $F \in\D(\WSm_X,\Lambda)^\blacksquare$ is in $\D^{\A^1}(\WSm_X,\Lambda)^\blacksquare$ if and only if the canonical map $F(Y) \to F(Y \times \mathbb{A}^1)$ is an equivalence for any weakly smooth $Y$ over $X$.
\end{rem}

\begin{cor}\label{cor:A1invrho}
	Assume that all primes in $\mathcal{P}$ are invertible on $X$. Then, the fully faithful functor $\rho_\sharp^\blacksquare \colon \D(X,\Lambda)^\blacksquare \to \D(\WSm_X,\Lambda)^\blacksquare$ factors through the full subcategory $\D^{\A^1}(\WSm_X,\Lambda)^\blacksquare$.
\end{cor}
\begin{proof}
	We may assume that $\Lambda=\widehat{\Z}_\mathcal{P}$. Note that the square
\[\begin{tikzcd}
	{\D(X,{\widehat{\Z}_\mathcal{P}})^\blacksquare} & {\D(\WSm_X,{\widehat{\Z}_\mathcal{P}})^\blacksquare} \\
	{\D(X_{\proet},{\widehat{\Z}_\mathcal{P}})} & {\D(\WSm_X,{\widehat{\Z}_\mathcal{P}})}
	\arrow["{\rho_{\sharp}^{\blacksquare}}", from=1-1, to=1-2]
	\arrow[hook, from=1-1, to=2-1]
	\arrow[hook, from=1-2, to=2-2]
	\arrow["{\rho_\sharp}"', from=2-1, to=2-2]
\end{tikzcd}\]
	commutes, because it is obtained by tensoring an adjunction in $\PrL_{\cond}$ with $\underline{\mathrm{Solid}}_{\widehat{\Z}_\mathcal{P}}$ (see the proof of \Cref{LikeFS}).
	Thus we have to check that for $ F \in \D(X,{\widehat{\Z}_\mathcal{P}})^\blacksquare$ the canonical map
	\[
		\rho_\sharp(F)(Y) \to \rho_\sharp(F)(Y \times \A^1)
	\]
	is an equivalence for any $f \colon Y \to X \in \WSm_X$.
	This map identifies with the canonical map
	$ f^*F(Y) \to \pr_1^* f^*F(Y \times \A^1)$, as $f^*F$ is solid by \Cref{lem:pullback-preserves-solid}, we may assume that $Y=X$. As $\mathrm{pr}_1^*$ commutes with limits by \Cref{lem:f*-commutes-with-iL}, we can use Postnikov-completeness to assume that $F$ is bounded below, and we can then further reduce to $F$ being in the heart by a spectral sequence argument. Then, we can reduce to $F$ being in $\mathrm{Pro}(\Sh_\mc{P}^\mathrm{cons}(X_\et))$ by \Cref{thm:solid} and then in $\Sh_\mc{P}^\mathrm{cons}(X_\et)$ in which case it follows from \cite[Exposé~XV, Corollaire~2.2]{sga4}.
\end{proof}

\begin{thm}\label{thm:embedding_rigidity}
	Assume that all primes in $\mc{P}$ are invertible on $X$. Then, the functor
	\[\rho_\sharp^\blacksquare\colon \D(X,\Lambda)^\blacksquare\to\D^{\A^1}(\WSm_X,\Lambda)^\blacksquare\] of \Cref{cor:A1invrho} is fully faithful, commutes with all small colimits and limits, all pullbacks and with $f_\sharp$ for $f$ weakly étale.
%
\end{thm}
\begin{proof}
	The commutation with colimits is formal (we are taking the global sections of a map in $\PrL_\mathrm{cond}$), the commutation with limits is \Cref{rho_sharp_commutes_with_limits}. The commutation with pullbacks and with $f_\sharp$ for $f$ weakly étale is true at the condensed level before solidification and taking global sections so that it remains true. 
\end{proof}

From this, we can deduce the solid rigidity theorem with torsion coefficients.
\begin{prop}
	\label{prop:rig-torsion}
	Let $n\in\N$ be invertible on $X$. Assume that $\Lambda$ is a $\Z/n\Z$-algebra.
	Then, the functors
	\[\D(X,\Lambda)^\blacksquare\to\D^{\A^1}(\WSm_X,\Lambda)^\blacksquare\to \D^{\A^1,\Gm}(\WSm_X,\Lambda)^\blacksquare\]
	are all equivalences.
\end{prop}
\begin{proof}
	It suffices to deal with the case $\Lambda = \Z/n\Z$.
	Using the functor frome effective étale motives we see that in $\D^{\A^1}(\WSm_X,\Z/n\Z)^\blacksquare$ the Tate twist is already invertible, thus the functor
	\[\D^{\A^1}(\WSm_X,\Z/n\Z)^\blacksquare\to \D^{\A^1,\Gm}(\WSm_X,\Z/n\Z)^\blacksquare\] is an equivalence. Hence, using \Cref{thm:embedding_rigidity}, the only thing left to prove is that the functor \[\D(X,\Lambda)^\blacksquare\to\D^{\A^1}(\WSm_X,\Lambda)^\blacksquare\] is essentially surjective knowing that it is fully faithful. Hence, we only have to prove that collection of objects that generate the target category under colimits are in the image. 
	Now, the $f_\sharp(\Z/n\Z)$ for $f$ weakly smooth generate $\D^{\A^1}(\WSm_X,\Z/n\Z)^\blacksquare$. Zariski-locally on the source such an $f$ can be written as a composition \[U\xrightarrow{g} \A^n_X \xrightarrow{j}\mb{P}^n_X \xrightarrow{p} X\] with $g$ weakly étale, $j$ the canonical open immersion and $p$ the canonical projection (here we just use \Cref{lem:locWL} and decompose the canonical projection $\A^n_X\to X$); as $j\circ g$ is still weakly étale, we see that $f$ is Zariski-locally on the source of the form \[U\xrightarrow{h} \mb{P}^n_X \xrightarrow{p} X\] with $h$ weakly étale. Hence, the $p_\sharp h_\sharp (\Z/n\Z)$ for $p$ and $h$ as above generate $\D^{\A^1}(\WSm_X,\Z/n\Z)^\blacksquare$. 
	But then, ambidexterity (\Cref{ambidextry_stable}) ensures that  $p_\sharp\simeq p_*(n)[2n]$ and thus the $p_*h_\sharp(\Z/n\Z)$ generate $\D^{\A^1}(\WSm_X,\Z/n\Z)^\blacksquare$. 

	Note now that on both sides we have proper base change for finitely presented morphism by \Cref{prop:solid_BC} and \Cref{prop:PBC}. This formally implies that $p_*$ commutes with $\rho_\sharp^\blacksquare$: indeed, we have an exchange morphism $\rho_\sharp^\blacksquare p_* \to p_*\rho_\sharp^\blacksquare$ as in \cite[4.4.2]{MR3477640} and we can test whether it is an isomorphism by taking the $\Hom(g_\sharp(\Z/n\Z),-)$ for $g$ weakly étale over $X$; the proof is then the same as in \cite[Proposition~4.4.3]{MR3477640}. As we already know that $h_\sharp$ commutes with $\rho_\sharp^\blacksquare$, we can reach the generators which finishes the proof.
\end{proof}

\subsection{The rigidity theorem}\label{sec:solidrigidity}
Our goal is now to prove the rigidity theorem. As we have seen, both $\D^{\A^1}(\WSm_{(-)},\Lambda)^\blacksquare$ and $\D^{\A^1,\Gm}(\WSm_(-),\Lambda)^\blacksquare$ are unstable qcqs motivic coefficient systems. Rigidity would imply Tate-stability and as a means of enforcing it, we  further localize $\D^{\A^1}(\WSm_{(-)},\Lambda)^\blacksquare$ at a certain set of maps. To see that this procedure still yields an unstable qcqs motivic coefficient system, we first need a bit of general theory.

Let $ \ccal \colon \Sch \to \CAlg(\PrL)$ be an unstable qcqs motivic coefficient system.
Fix a small collection of morphisms in $ T $ in $ \ccal(\Spec(\Z)) $.
Then we may define an induced functor $ \ccal_T \colon \Sch \to \CAlg(\PrL) $ via
\[
	f \colon Y \to \Spec(\Z) \mapsto \ccal_T(Y) \coloneqq \ccal(Y)[\{f^* \lambda \otimes M \mid \lambda \in T; M \in \ccal(Y)\}^{-1}].
\]

\begin{rem}
	Note that there is a natural equivalence of functors
	\[
		\ccal_T \simeq \ccal(-) \otimes_{\ccal(\Spec(\Z))} \ccal_T(\Spec(\Z)).
	\]
\end{rem}

\begin{prop}
	\label{prop:inverting_maps_in_a_6ff}
	In the above setting, also ${\ccal}_T$ is an unstable qcqs motivic coefficient system.
\end{prop}
\begin{proof}
	This is immediate once one realises that all properties of being an unstable qcqs motivic coefficient system are properties internal to the $(\infty,2)$-category $\mathbf{Pr}^\mathrm{L}_{\ccal(\Spec(\Z))}$. As the functor $-\otimes_{\ccal(\Spec(\Z))}\ccal_T(\Spec(\Z))$ is an ($\infty$,2)-functor (\cite[Section 4.4]{MR3607274}), the claim follows.
\end{proof}

Recall that $\Pi\mathcal{P}$ is the set of integers which are divisible only by primes in $\mathcal{P}$. 
\begin{defi}\label{def:effectivesolidmotives}
	The $\infty$-category $\mathrm{DM}^\mathrm{eff}(X,\Lambda)^\blacksquare$ of \emph{effective solid motives} over $X$ is the localization of the $\infty$-category
		$\D^{\A^1}(\WSm_X,\Lambda)^\blacksquare$ at the maps \[\Lambda[\Gm_X]\otimes_\Lambda\Lambda[Y]\simeq\Lambda[\Gm_Y]\to (\lim_{n\in \Pi\mathcal{P}}\Z/n\Z[\Gm_X])\otimes_{\widehat{\Z}_\mathcal{P}} \Lambda[Y]\] where $Y$ is weakly smooth over $X$.

	\end{defi}
\begin{rem}
	\begin{enumerate}
		\item Note that by \Cref{prop:rig-torsion}, the objects $\Z/n\Z[\Gm_X]$ are actually in the image of the functor $\rho_\sharp^\blacksquare$. Moreover, because $\rho_\sharp^\blacksquare$ commutes with limits by \Cref{rho_sharp_commutes_with_limits}, we see that $\lim_n\Z/n\Z[\Gm_X]$ is again in the image of $\rho_\sharp^\blacksquare$. Using that over the small pro-étale site,
	pullbacks commute with all limits, we see that the maps we invert are stable under pullbacks.
		\item Note also that in fact, if $f$ is a map that we invert, and $M$ is any object, then the map $f\otimes\mathrm{Id}_M$ is also inverted. In particular, the localization functor is symmetric monoidal.
	\end{enumerate}
\end{rem}

\begin{rem}
	\label{rem:connection_to_mu}Assume that all the primes in $\mathcal{P}$ are invertible on $X$. By \cite[Proposition 3.2.2]{MR3477640} we have $\Z(1)[1]\simeq \Gm$ in $\mathrm{DM}_{\et}^{\mathrm{eff}}(\Spec(\Z),\Z)$. Thus for $n\in\Pi\mc{P}$, the classical Kummer sequences identifies $\Z(1)[1]/n \simeq \mu_n[1]$ (see \cite[3.2]{MR3477640}). In particular, we obtain maps $\Z(1)\to\mu_n$ for all $n\in\Pi\mc{P}$. As both sides a stable under pullbacks, those maps exist over any scheme $X$. Their images in $\D^{\A^1}(\WSm_X,\Lambda)^\blacksquare$ provides after taking a limit a map $$\Lambda(1)\to (\lim_{n\in\Pi\mc{P}}\mu_n)\otimes_{\ZhatP}\Lambda.$$ Thus, in $\D^{\A^1}(\WSm_X,\Lambda)^\blacksquare$, the map
	\[\Lambda[\Gm_X]\otimes_\Lambda\Lambda[Y]\to (\lim_{n\in \Pi\mathcal{P}} \Z/n\Z[\Gm_X])\otimes_{\widehat{\Z}_\mathcal{P}} \Lambda[Y]\] we invert can in fact be written as
	\[\Lambda\oplus \Lambda(1)[1]\to (\lim_{n\in \Pi\mathcal{P}}\Z/n\Z\oplus \lim_{n\in \Pi\mathcal{P}} \mu_n)\otimes_{\widehat{\Z}_\mathcal{P}}\Lambda[Y].\] 
	We have ${\widehat{\Z}_\mathcal{P}} = \lim_{n\in \Pi\mathcal{P}}\Z/n\Z$ in $\D^{\A^1}(\WSm_X,{\widehat{\Z}_\mathcal{P}})^\blacksquare$ because this is true over the small site, and $\rho_\sharp^\blacksquare$ commutes with limits. 
	Moreover, $\mu_n$ comes from the small étale site, and the limit $\lim_{n\in\Pi\mathcal{P}}\mu_n$ is representable in the small pro-étale site by $\mu_{\infty,\mathcal{P}}$. This proves that our localization is also the localization of $\D^{\A^1}(\WSm_X,\Lambda)^\blacksquare$ at the maps
	\[M(1)\to M\otimes_{\widehat{\Z}_\mathcal{P}}\mu_{\infty,\mathcal{P}}\] for all $M$.
	
\end{rem}

The discussion above in particular shows that the functor
\[
\mathrm{DM}^\mathrm{eff}(-,\Lambda) ^\blacksquare\colon \Sch \to \CAlg(\PrL)
\]
is as in \Cref{prop:inverting_maps_in_a_6ff}, for the map $\Lambda[\Gm_\mathbb{Z}] \to (\lim_{n\in \Pi\mathcal{P}} \Z/n\Z[\Gm_\mathbb{Z}])\otimes_{\widehat{\Z}_\mathcal{P}} \Lambda$.

\begin{prop}
	The functor $ \mathrm{DM}^\mathrm{eff}(-,\Lambda)^\blacksquare $ is a qcqs motivic coefficient system. Hence we have the six operations on $ \mathrm{DM}^\mathrm{eff}(-,\Lambda)^\blacksquare $ in the sense of \Cref{thm:motivic_6FF}.
\end{prop}
\begin{proof}
	The only thing that does not follow from \Cref{prop:inverting_maps_in_a_6ff} is Tate-stability, but it is true since in $ \mathrm{DM}^\mathrm{eff}(-,\Lambda)^\blacksquare $, the Tate twist is by definition equivalent to $\mu_{\infty,\mathcal{P}}\otimes_{\widehat{\Z}_\mathcal{P}}\Lambda$ which is invertible by \Cref{lem:muinfty_invertible}.
\end{proof}

Let us now state our first rigidity result:

\begin{thm}\label{thm:solid_rigidity_V1}
	Assume that all prime numbers $\mc{P}$ are invertible on $X$. Then, the functor $ \rho_\sharp^\blacksquare \colon \D(X,\Lambda)^\blacksquare \to \D^{\mathbb{A}^1}(\WSm_X,\Lambda)^\blacksquare$ induces an equivalence
	\[
		\D(X,\Lambda)^\blacksquare \to \mathrm{DM}^\mathrm{eff}(X,\Lambda)^\blacksquare.
	\]
	In particular, the $\D(-,\Lambda)^\blacksquare$ are endowed with the six functors in the sense of \Cref{thm:motivic_6FF}. 
\end{thm}
\begin{proof}We can assume that $\Lambda=\widehat{\Z}_\mathcal{P}$. 
	To prove full faithfulness, we need to prove that for any $ \mathcal{F} \in \D(X,\widehat{\Z}_\mathcal{P})^\blacksquare  $ the canonical map
	\[
			\underline{\Hom}_{\D^{\mathbb{A}^1}(\WSm_X,\widehat{\Z}_\mathcal{P})^\blacksquare}(\lim_{n\in \Pi\mathcal{P}}\Z/n\Z[\Gm_X],\rho_\sharp^\blacksquare  \mathcal{F}) \to \underline{\Hom}_{\D^{\mathbb{A}^1}(\WSm_X,\widehat{\Z}_\mathcal{P})^\blacksquare}(\widehat{\Z}_\mathcal{P}[\mathbb{G}_m],\rho_\sharp^\blacksquare  \mathcal{F})
	\]
	is an equivalence.
	Since $ \rho_\sharp^\blacksquare $ commutes with limits we may replace $ \mathcal{F} $ by $ \tau_{\leq n} \mathcal{F} $.
	We claim that both sides preserve filtered colimits in $ \mathcal{F} $.
	For the domain this follows immediately because the object $\lim_{n\in \Pi\mathcal{P}}\Z/n\Z[\Gm_X]$ is dualizable in $\D^{\mathbb{A}^1}(\WSm_X,\widehat{\Z}_\mathcal{P})^\blacksquare$, by \Cref{rem:connection_to_mu} and \Cref{lem:muinfty_invertible}.
	For the codomain, we observe that since the localization functor $\D(\WSm_X,\widehat{\Z}_\mathcal{P}) \to \D^{\mathbb{A}^1}(\WSm_X,\widehat{\Z}_\mathcal{P})^\blacksquare$ is symmetric monoidal, we have that
	\[
		\underline{\Hom}_{\D^{\mathbb{A}^1}(\WSm_X,\widehat{\Z}_\mathcal{P})^\blacksquare}(\widehat{\Z}_\mathcal{P}[\mathbb{G}_m],\rho_\sharp^\blacksquare  \mathcal{F}) \simeq \underline{\Hom}_{\D(\WSm_X,\widehat{\Z}_\mathcal{P})}(\widehat{\Z}_\mathcal{P}[\mathbb{G}_m],\iota \rho_\sharp^\blacksquare  \mathcal{F})
	\]
	where we write $\iota \colon \D^{\mathbb{A}^1}(\WSm_X,\widehat{\Z}_\mathcal{P})^\blacksquare \to \D(\WSm_X,\widehat{\Z}_\mathcal{P})$ for the inclusion.
	To check compatibility with filtered colimits, it thus suffices to show that the functor
	\[
		\underline{\Hom}_{\D(\WSm_X,\widehat{\Z}_\mathcal{P})}(\widehat{\Z}_\mathcal{P}[\mathbb{G}_m],\rho_\sharp^\blacksquare \tau_{\leq n}(-)) \colon \D(X,\widehat{\Z}_\mathcal{P})^\blacksquare \to \D(\WSm_X,\widehat{\Z}_\mathcal{P})
	\]
	commutes with filtered colimits.
	This we may check after mapping out of $ \widehat{\Z}_\mathcal{P}[W] $ for a $w$-contractible weakly smooth scheme, since (shifts of) these compactly generate $\D(\WSm_X,\widehat{\Z}_\mathcal{P})$.
	Writing $p \colon W \to X$ for the structure map and $\pi \colon W \times \mathbb{G}_m \to W$ for the projection, we get
	\begin{align*}
		& \Hom_{\D(\WSm_X,\widehat{\Z}_\mathcal{P})}(\widehat{\Z}_\mathcal{P}[W], \underline{\Hom}_{\D(\WSm_X,\widehat{\Z}_\mathcal{P})}(\widehat{\Z}_\mathcal{P}[\mathbb{G}_m],\rho_\sharp^\blacksquare \tau_{\leq n}(-)))  \\ \simeq & \Hom_{\D(\WSm_W,\widehat{\Z}_\mathcal{P})}(\widehat{\Z}_\mathcal{P}[\mathbb{G}_m], p^* \rho_\sharp^\blacksquare \tau_{\leq n} (-)) \\   \simeq & \Hom_{\D(\WSm_{\mathbb{G}_{m} \times W},\widehat{\Z}_\mathcal{P})^\blacksquare}(\widehat{\Z}_\mathcal{P}, \pi^* p^* \rho_\sharp^\blacksquare \tau_{\leq n}(-))
	\end{align*}
	Now note that both sides are in the image of $ \rho_\sharp^\blacksquare $, since $ \rho_\sharp^\blacksquare $ commutes with pullback.
	But then the claim follows from \cite[Lemma~5.3]{MR4609461}.

	Thus we may reduce to the case where $ \mathcal{F} $ is bounded and so, by d\'evissage, to the case where $ \mathcal{F} $ is concentrated in one degree.
	Using \Cref{thm:solid} and commutation with filtered colimits again, we may reduce to the case where $ \mathcal{F} $ is torsion and the claim is then clear.
	The proof of essential surjectivity is then exactly the same as in \Cref{prop:rig-torsion}.
\end{proof}

We finish by defining the category of solid motives and proving that it is the same as the category of effective solid motives.

\begin{defi}
	We define the category of \emph{solid motives} $\mathrm{DM}(X,\Lambda)^\blacksquare$ to be the localization $\D^{\A^1,\Gm}(\WSm_X,\Lambda)^\blacksquare$ at the maps \[M(1)\to M\otimes_{\widehat{\Z}_\mathcal{P}}\Sigma^\infty_{\mathbb{P}^1}\mu_{\infty,\mathcal{P}}\] for all $M$.
\end{defi}
\begin{prop}\label{prop:solid_rigidity_V2}
	The canonical functor \[\mathrm{DM}^\mathrm{eff}(X,\Lambda)^\blacksquare\to\mathrm{DM}(X,\Lambda)^\blacksquare\] is an equivalence.
\end{prop}
\begin{proof} It suffices to show that $\D^{\A^1,\Gm}(\WSm_X,\Lambda)^\blacksquare$ is the $\Gm$-stabilization of $\D^{\A^1}(\WSm_X,\Lambda)^\blacksquare$. Indeed, both categories will then have the same universal properties as presentably symmetric monoidal $\infty$-categories because $\mu_{\infty,\mathcal{P}}$ is $\otimes$-invertible already in the small pro-étale site. 

The object $\Lambda(1)$ is symmetric in $\D^{\A^1}(\WSm_X,\Lambda)^\blacksquare$ because it is the image of the object $\Z(1)$ of $\D^{\A^1}_{\et}(\Sm_X,\Z)$ which is already symmetric by \cite[Lemma~4.4]{MR1648048}. We can therefore compute the $\Gm$-stabilization as the colimit of the diagram: 
\[\D^{\A^1}(\WSm_X,\Lambda)^\blacksquare \xrightarrow{-\otimes_\Lambda \Lambda(1)}\D^{\A^1}(\WSm_X,\Lambda)^\blacksquare \xrightarrow{-\otimes_\Lambda \Lambda(1)}\cdots.\]
On the other hand, we have a natural map from the colimit of the diagram
\[\underline{\D}^{\A^1}_\proet(\WSm_X,\Lambda)\xrightarrow{-\otimes_\Lambda \Lambda(1)} \underline{\D}^{\A^1}_\proet(\WSm_X,\Lambda) \xrightarrow{-\otimes_\Lambda \Lambda(1)} \cdots\]
 to 
$\underline{\DM}_\proet(X,\Lambda)$. This map is an equivalence because taking sections over any profinite set $S$ is compatible with colimits (see the proof of \cite[Proposition~2.6.3.2]{MWColimits} applied to condensed $S$-categories) and over $S$, this is true by definition. As the condensed tensor product commutes with small colimits in each variable and as the global section functor commutes with colimits, this implies that $\D^{\A^1,\Gm}(\WSm_X,\Lambda)^\blacksquare$ indeed is the $\Gm$-stabilization of $\D^{\A^1}(\WSm_X,\Lambda)^\blacksquare$ finishing the proof.
\end{proof}

\begin{thm}
	The obtained realization functor 
	\[\rho_\blacksquare\colon\DM_\proet(X,\Z)\to \DM(X,\Lambda)^\blacksquare \simeq \D(X,\Lambda)^\blacksquare\] is a morphism of $6$-functor formalisms on qcqs $\Spec(\Z[1/\mc{P}])$-schemes. That is, it commutes with $f^*$, $\otimes$, and $g_!$ for $g$ a finitely presented map of qcqs schemes.
\end{thm}
\begin{proof}
	By construction, the functor $\rho_\blacksquare$ is a morphism of qcqs motivic coefficient systems.
	In particular, it commutes with the functor $f^*$ for each map $f$ of schemes,  with $\otimes$, and with $f_\sharp$ if $f$ is a smooth map of schemes (in fact, it is clear that it also commutes with $j_\sharp$ for $j$ a weakly étale map).
	Using the localization triangle, we see that it also commutes with $i_*$ for $i$ a finitely presented closed immersion of qcqs schemes.
	Using ambidexterity, one then concludes that it commutes with $p_*$ for any finitely presented projective morphism of qcqs schemes.
	Using the refined Chow lemma (\cite[\href{https://stacks.math.columbia.edu/tag/01ZZ}{Lemma 2020}]{stacks-project}), this implies that it commutes with $p_*$ with $p$ any proper map of finite presentation between qcqs schemes. Thus, the functor $\rho_\blacksquare$ commutes with $f_!$ for any finitely presented map of qcqs schemes. To obtain a morphism of $6$-functor formalisms, one may use the fact that \cite[Theorem 4.35]{CLLSpan} is an equivalence of categories, in particular it is functorial, so that $\rho_\blacksquare$ lifts to a natural transformation between functors on correspondences.
\end{proof}

\begin{prop}
	Let $X$ be a qcqs scheme over which $\ell$ is invertible. Then the composition 
	\[\rho_\blacksquare\colon \mathrm{DM}_{\et,\mathrm{gm}}(X,\Z)\to \mathrm{DM}_\proet(X,\Z)\to \D(X,\Z_\ell)^\blacksquare\] lands in 
	the category $\D^b_c(X_\et,\Z_\ell)$ and coincides with the $\ell$-adic realization functor defined in \cite{MR3477640}.
\end{prop}
\begin{proof}
	By \cite[Corollary 3.7]{DMpdf}, for any object $M\in\DM_{\et,\mathrm{gm}}(X,\Z)$, there exists a stratification on which $M$ is dualizable. We claim that this implies that $\rho_\blacksquare(M)$ is $\ell$-complete. Indeed, any dualizable object of $\D(X,\Z_\ell)^\blacksquare$ is $\ell$-complete because the unit is complete, and then we can use \Cref{lem:conservativity-of-pullbacks} together with the fact that pullback functors preserves limits to conclude. Thus, to show that $\rho_\blacksquare(M)$ is an object of $\D^b_c(X_\et,\Z_\ell)$, it suffices to check that for all $n\in\N^*$, the object $\rho_\blacksquare(M)/\ell^n=\rho_\blacksquare(M/\ell^n)$ is a bounded complex of torsion constructible étale sheaves.
	This means that we reduced the question to showing that the image of the functor 
	\[\mathrm{DM}_{\et,\mathrm{gm}}(X,\Z)\xrightarrow{-/\ell^n}\Mod_{\Z/\ell^n\Z}(\mathrm{DM}_\et(X,\Z))\to \Mod_{\Z/\ell^n\Z}(\D(X,\Z_\ell)^\blacksquare)\] lands in $\D^b_c(X_\et,\Z/\ell^n\Z)$. But using rigidity for étale motives (\cite{bachmannrigidity}), the functor 
	$$\Mod_{\Z/\ell^n\Z}(\mathrm{DM}_\et(X,\Z))\to \Mod_{\Z/\ell^n\Z}(\D(X,\Z_\ell)^\blacksquare)$$ can be identified with the  functor 
	\[\D(X_\et,\Z/\ell^n\Z)\to \D(X,\Z/\ell^n\Z)^\blacksquare.\] 
	By \Cref{prop:torsion etale is solid derived}, the above functor is fully faithful after left completion, in particular, it is fully faithful on bounded objects.
	As the image of $M$ is then a dualizable object of $\D(X_\et,\Z/\ell^n\Z)$, we conclude that it is a bounded complex of constructible étale sheaves. 

	We now prove that $\rho_\blacksquare$ coincides with the functor $\rho_\ell$ of \cite{MR3477640}. Recall that the latter is the restriction to constructible motives of the $\ell$-completion functor 
	\[\mathrm{DM}_\et(X,\Z)\to \mathrm{DM}_\et(X,\Z)^\wedge_\ell\simeq \D(X_\et,\Z)^\wedge_\ell.\] Moreover, as 
	\[\D^b_c(X_\et,\Z_\ell)\simeq \lim_{n}\D^b_c(X_\et,\Z/\ell^n\Z)\] by \cite[Theorem 7.7]{MR4609461}, it suffices to prove that for each $n\in\N^*$, the two functors $\mathrm{DM}_{\et,\mathrm{gm}}(X,\Z)\to\D^b_c(X_\et,\Z/\ell^n\Z)$ are equivalent in a way compatible with $n$. This was already proven in the previous paragraph: we have shown that for $M\in\mathrm{DM}_{\et,\mathrm{gm}}(X,\Z)$, the object $\rho_\blacksquare(M)/\ell^n$ is nothing but $M/\ell^n \in\D(X_\et,\Z/\ell^n\Z)$.
\end{proof}

\begin{cor}
	\label{cor:MagicSquare}
	We have a commutative square 
	\[\begin{tikzcd}
	{\mathrm{DM}_\et(X,\Z)} & {\D(X,\Z_\ell)^\blacksquare} \\
	{\mathrm{DM}_\et(X,\Q)} & {\D(X,\Q_\ell)^\blacksquare}
	\arrow["{\rho_\blacksquare}", from=1-1, to=1-2]
	\arrow["{\otimes\Q}"', from=1-1, to=2-1]
	\arrow["{\otimes\Q}", from=1-2, to=2-2]
	\arrow["{\rho_\blacksquare}"', from=2-1, to=2-2]
\end{tikzcd}\] of symmetric monoidal left adjoints which restricts to the classical square 
\[\begin{tikzcd}
	{\mathrm{DM}_{\et,\mathrm{gm}}(X,\Z)} & {\D^b_c(X_\et,\Z_\ell)} \\
	{\mathrm{DM}_{\et,\mathrm{gm}}(X,\Q)} & {\D^b_c(X_\et,\Q_\ell)}
	\arrow["{\rho_\ell}", from=1-1, to=1-2]
	\arrow["{\otimes\Q}"', from=1-1, to=2-1]
	\arrow["{\otimes\Q}", from=1-2, to=2-2]
	\arrow["{\rho_{\Q_\ell}}"', from=2-1, to=2-2]
\end{tikzcd},\] where $\rho_\ell$ is the $\ell$-adic realization functor of \cite{MR3477640} and $\rho_{\Q_\ell}$ is the $\Q_\ell$-adic realization functor of \cite[2.1.2]{MR4061978}.
\end{cor}
\begin{proof}
	The first square exists formally. That it restricts to the second square can be seen as follows. We have already seen that the restriction of $\rho_\blacksquare$ to $\mathrm{DM}_{\et,\mathrm{gm}}(X,\Z)$ is $\rho_\ell$. The claim about the vertical functors is formal. For the low horizontal functor, this follows from the definition (as the rationalization of $\rho_\ell$) and \cite[Corollary 2.4]{zbMATH07751006}.
\end{proof}
\begin{rem}Over schemes of characteristic zero, the above \Cref{cor:MagicSquare} also holds with $\widehat{\Z}$ and the ring of finite adeles $\A_{\Q,f}$ instead of $\Z_\ell$ and $\Q_\ell$. By using right adjoints and a bit of work, this allows for another definition of the algebra $\mathscr{N}\in\mathrm{DM}^{\et}(\Spec(\Q),\Z)$ such that modules over it are Nori motives, as considered in \cite{integralNori}. In forthcoming work, we plan to study a spectral analogue of the constructions of this paper, that would lead to a definition of spectral Nori motives.
\end{rem}
\bibliographystyle{alpha}
\bibliography{BibSCNet}

\end{document}